\documentclass[leqno,12pt]{amsart}
\usepackage[thmnumber=subsubsection,eqnumber=subsubsection]{jdr-style}
\usepackage{jdr-macros,jdr-tikz}
\usepackage[all,tips]{xy}
\usepackage{longtable}
\usepackage{url}
\usepackage{verbatim}
\usepackage{tikz}
\usepackage{mathtools}
\usepackage{xr}
\CompileMatrices

\newcommand{\otherfile}[1]{\cite[#1]{gubler_rabinoff_jell:harmonic_trop}}

\mathtext{sup}

\title{Dolbeault Cohomology of Graphs and Berkovich Curves}

\author[W.~Gubler]{Walter Gubler}
\address{W. Gubler, Mathematik, Universit{\"a}t 
	Regensburg, 93040 Regensburg, Germany}
\email{walter.gubler@mathematik.uni-regensburg.de}

\author[P.~Jell]{Philipp Jell}
\address{P. Jell, Mathematik, Universit{\"a}t 
	Regensburg, 93040 Regensburg, Germany}
\email{philipp.jell@mathematik.uni-regensburg.de}

\author[J.~Rabinoff]{Joseph Rabinoff}
\address{J. Rabinoff, Department of Mathematics, Trinity College of Arts and Sciences, Duke University, Durham, NC 27708, USA}
\email{jdr@math.duke.edu}

\mathtext{Span}
\mathtext{sing}
\def\rS{\mathrm{S}}

\newcommand\secref[1]{\ref{#1}}

 \thanks{W.~Gubler and P.~Jell 
	were supported by the collaborative research 
	center SFB 1085 \emph{Higher Invariants - Interactions between Arithmetic Geometry and Global Analysis} funded by the Deutsche Forschungsgemeinschaft.}

\begin{document}
	
\begin{abstract}
  We introduce real-valued $(p,q)$-forms on weighted metric graphs with boundary similar to Lagerberg forms on polyhedral spaces. We compute the Dolbeault cohomology and prove Poincar\'e duality. Using Thuillier's thesis, the skeleton of a strictly semistable formal curve is canonically a weighted metric graph with boundary. We use that and our companion paper on weakly smooth forms to compute the Dolbeault cohomology for weakly smooth forms on any non-Archimedean compact rig-smooth analytic curve $X$, and prove Poincar\'e duality when $X$ is proper. 
\end{abstract}


\maketitle
\setcounter{tocdepth}{1}


\section{Introduction}
	
 It is well-known that compact metric graphs are a discrete analogue of compact Riemann surfaces. In both theories, we have a genus, a Laplacian, harmonic functions, a Riemann--Roch theorem, differential forms, and a Jacobian. In the following, we consider a \defi{weighted metric graph $\Sigma$ with boundary}, i.e.~a compact metric multigraph $\Sigma$ with no loop edges where each edge $e$ has a weight $w(e)\in \Z_{>0}$, and where the boundary is any subset $\partial \Sigma$ of the vertices; see~\secref{sec:weighted metric graphs} for more details.  Such graphs arise naturally as skeletons of compact non-Archimedean curves; the second half of this paper will be devoted to applications to non-Archimedean curves.  With these applications in mind, we have to allow non-connected graphs with isolated vertices in $\Sigma$, but for simplicity of the exposition we assume throughout the introduction that $\Sigma$ {\bf\emph{is connected and that $\Sigma$ has an edge}}. 
	
\subsection{Harmonic functions and maps}
A function $f\colon \Sigma \to \R$ is called \defi{harmonic} if $f$ is linear on each edge, and if for any vertex $v \not\in \partial \Sigma$, we have
\[ \sum w(e)\,\frac{\d f}{\d t_e}(v) = 0, \]
where the sum is taken over all outgoing edges $e$ at $v$ and where $\frac{\d f}{\d t_e}(v)$ denotes the outgoing slope at $v$ along $e$. See~\secref{sec:harmonic functions} for more details. 
	
We say that a map $\phi\colon \Sigma' \to \Sigma$ of metric weighted graphs with boundary is \defi{harmonic} if $\varphi$ satisfies the following properties:
\begin{enumerate}
\item After possible subdivisions, every edge of $\Sigma'$ is mapped either linearly onto an edge or onto a vertex of $\Sigma$.
\item If $\varphi$ is non-constant locally at $v' \in \Sigma'$ with $\varphi(v') \in \partial \Sigma$, then $v' \in \partial \Sigma'$.
\item Locally, $\varphi$ pulls back harmonic functions on $\Sigma$ to harmonic functions on $\Sigma'$. 
\end{enumerate}
We explore harmonic maps in~\secref{sec:harmonic-morphisms}, generalizing results from~\cite{abbr14:lifting_harmonic_morphism_I} given in the context of unweighted metric graphs without boundary. Examples of harmonic maps arise from subgraphs (see \secref{sec:subgraphs}), from \emph{modifications} obtained by attaching finitely many trees (see~\secref{sec:modifications}) and from quotients by finite groups acting by harmonic maps on $\Sigma$ (see \secref{sec:quotients}).
	
\subsection{Smooth forms on graphs and Lagerberg forms}
Lagerberg~\cite{lagerberg12:super_currents} has introduced the sheaf $\cA^{p,q}\coloneqq\cA^p \otimes_{C^\infty} \cA^q$ of $(p,q)$-forms on $\R^n$, where on the right we use the sheaf $\cA^\bullet$ of usual smooth differential forms on $\R^n$. We call such $(p,q)$-forms  \defi{smooth Lagerberg forms}. The sheaf $\cA^{\bullet,\bullet}$ is a bigraded differential sheaf  of $\R$-algebras with an alternating product $\wedge$ and natural differentials $\d'\colon\cA^{p,q} \to \cA^{p+1,q}$ and $\d''\colon\cA^{p,q} \to\cA^{p,q+1}$ which are analogues of the differential operators $\partial$ and $\bar{\partial}$ in complex analysis. For an open subset $U$ of $\R$ with coordinate function $x$, we have the following direct description:
\begin{enumerate}
\item The elements of $\cA^{0,0}(U)$ are the smooth real functions on $U$.
\item The elements of $\cA^{1,0}(U)$ have the form $f \d'x$ with $f \in \cA^{0,0}(U)$.
\item The elements of $\cA^{0,1}(U)$ have the form $f \d''x$ with $f \in \cA^{0,0}(U)$.
\item The elements of $\cA^{1,1}(U)$ have the form $f \d'x \wedge \d''x$ with $f \in \cA^{0,0}(U)$.
\end{enumerate}

Lagerberg forms are closely related to tropical geometry. 
We use now a similar construction to introduce smooth $(p,q)$-forms on the weighed metric graph $\Sigma$ with boundary.

\begin{prop} \label{prop:intrinsic approach to forms}
  There is a unique  sheaf $\cA^{\bullet,\bullet}$ of bigraded  $\R$-algebras on $\Sigma$ with an alternating product and differentials $\d',\d''$, whose elements are called \defi{smooth forms}, with the following properties:
  \begin{enumerate}
  \item \label{item:forms and subdivison}
    The sheaf $\cA^{\bullet,\bullet}$ is insensitive with to subdivisions of $\Sigma$.
  \item \label{item:forms and edge}
    If $U$ is an open subset of an edge, then $\cA^{p,q}(U)$ agrees with the smooth Lagerberg forms on $U$ identified with an open subset of $\R$. 
  \item \label{item:forms and restriction}
    A smooth $(p,q)$-form is completely determined by its restriction to all open edges.
  \item \label{item:forms and generation}
    The harmonic functions and the elements of $\cA^{0,0}$ locally generate  $\cA^{\bullet,\bullet}$ as a bigraded differential sheaf of $\R$-algebras.
  \end{enumerate}
\end{prop}
	
In \secref{sec:smooth forms}, we will explicitly construct global $(p,q)$-forms on $\Sigma$,  giving rise to a sheaf $\cA^{\bullet,\bullet}$ of bigraded  $\R$-algebras with an alternating product and differentials $\d',\d''$ such that properties (\ref{item:forms and subdivison}), (\ref{item:forms and edge}), (\ref{item:forms and restriction}) are satisfied and such that harmonic functions are smooth. We will show at the end of \secref{sec:harmonic-morphisms} that the smooth forms admit functorial pull-backs along \emph{harmonic maps} of weighted metric graphs with boundary. In Section~\ref{Sec:lagerberg-forms}, we will study harmonic tropicalization maps $h \colon \Sigma \to \R^n$ and give a canonical lift of any smooth Lagerberg form on $\R^n$ to a smooth form on $\Sigma$. The main result (Proposition~\ref{prop:local.pullbacks.graphs}) is that any smooth form on $\Sigma$ is locally given by such a lift. This yields property (\ref{item:forms and generation}) in Proposition~\ref{prop:intrinsic approach to forms}.

\subsection{Dolbeault cohomology on graphs} \label{sec:Dolbeault cohomology on graphs}
In~\secref{sec:integration}, we give a theory of integration for smooth $(1,1)$ forms on $\Sigma$ such that the formula of Stokes holds. Roughly speaking, this is obtained from the classical theory replacing $\d'x \wedge \d''x$ on an edge $e$ by $\d x$ to obtain a standard smooth $1$-form on the corresponding closed interval of $\R$.   As for Riemann surfaces, we define the \defi{Dolbeault cohomology groups} for $p \in \{0,1\}$ by 
	
\[\begin{split}
    H^{p,0}(\Sigma,\del\Sigma) &= \mathrm{\phantom{co}ker}\bigl( \d''\colon\cA^{p,0}(\Sigma,\del\Sigma)\To\cA^{p,1}(\Sigma,\del\Sigma) \bigr) \\
    H^{p,1}(\Sigma,\del\Sigma) &= \coker\bigl( \d''\colon\cA^{p,0}(\Sigma,\del\Sigma)\To\cA^{p,1}(\Sigma,\del\Sigma) \bigr)
  \end{split}\]
using the sheaf of smooth forms from Proposition~\ref{prop:intrinsic approach to forms}. We set $h^{i,j} \coloneqq \dim H^{i,j}(\Sigma,\del\Sigma)$.

\begin{prop}\label{prop:introprop:dolbeault.graphs.intro}
  Let  $g$ be the genus of $\Sigma$. 
  \begin{enumerate}
  \item We have $h^{0,0}= 1$.
  \item We have $h^{1,1}=1$ if $\partial \Sigma=\emptyset$, and $h^{1,1}=0$ if $\partial \Sigma \neq\emptyset$.
  \item We have $h^{1,0}=h^{0,1}=g$ if $\partial \Sigma=\emptyset$, and $h^{1,0}-\#\del\Sigma+1=h^{0,1}=g$ if  $\partial \Sigma \neq\emptyset$.
  \end{enumerate}
\end{prop}

This will be shown in Proposition~\ref{prop:dolbeault.graphs} by using standard methods of graphs and the theory of integration from \secref{sec:integration}. In \secref{sec:poincare-duality-graphs}, we will show that \defi{Poincar\'e duality} holds for Dolbeault cohomology on a weighted metric graph without boundary. 
For a harmonic morphism $h\colon \Sigma' \to \Sigma$ of weighted metric graphs, we get a functorial pull-back $h^*\colon H^{p,q}(\Sigma,\partial \Sigma)\to H^{p,q}(\Sigma',\partial \Sigma')$ induced by the pull-back of smooth $(p,q)$-forms. This is an isomorphism if $h$ is a modification (Lemma~\ref{lem:modif.isom.cohom}). If $\Sigma$ is the quotient of $\Sigma'$ by a finite group $G$ acting harmonically, then we will show in Proposition~\ref{prop:quotient.graph.cohom} that $H^{p,q}(\Sigma,\partial \Sigma) \simeq H^{p,q}(\Sigma',\partial \Sigma')^G$.
	
\subsection{Forms on Berkovich spaces}
	
Let $k$ be a non-trivially valued non-Archimedean field. In complex algebraic geometry, transcendental methods are often very useful. To initiate a non-Archimedean analogue, Chambert-Loir and Ducros~\cite{chambert_ducros12:forms_courants} have introduced a notion of smooth real-valued forms on a Berkovich analytic space. Their construction roughly proceeds as follows.  Let $X$ be a  non-Archimedean analytic space over $k$, and let $f\colon W\to\bG_m^{n,\an}$ be a morphism on a compact  analytic domain $W$ of $X$.  The associated \defi{smooth tropicalization map} is the composition $h = \trop\circ f$, where $\trop\colon\bG_m^{n,\an}\to\R^n$ is the \emph{tropicalization map} $(x_1,\ldots,x_n) \mapsto (-\log|x_1|,\ldots,-\log|x_n|)$.  A \defi{smooth $(p,q)$-form} on $X$ is then  locally given on a compact analytic domain $W$ of $X$ by pullback of a smooth  Lagerberg form on $\R^n$ of bidegree $(p,q)$  via a smooth tropicalization map~$h\colon W \to \R^n$.
	
To a certain extent, these forms behave similarly as in complex geometry: Chambert-Loir and Ducros~\cite{chambert_ducros12:forms_courants} show that there is a theory of integration with the Theorem of Stokes, and give a Poincar\'e--Lelong formula and  a local approach to the Monge--Amp\`ere operator~\cite[6.9]{chambert_ducros12:forms_courants} using curvature. However, this theory has certain deficiencies, which are already visible in dimension~$1$. It is shown in~\cite{jell19:tropical_hodge_numbers} that for certain non-trivially valued algebraically closed non-Archimedean ground fields $k$, the  Dolbeault cohomology group $H^{1,1}(X)$ may be infinite dimensional, and Poincar\'e duality may fail, even when $X$ is a smooth, proper curve.  More precisely, if the residue field of $k$ is $\C$, then both finite dimensionality and Poincar\'e duality hold precisely if $X$ is a Mumford curve.  See also~\cite{jell_wanner18:poincare_duality}.
	
The fundamental reason for the phenomena listed above is that \emph{harmonic functions are not smooth} in the theory of Chambert-Loir--Ducros.  This goes back to a result of Thuillier~\cite{thuillier05:thesis}, as noticed in Wanner's master thesis~\cite{wanner16:harmonic_functions}. For example, let $k$ be a non-trivially valued algebraically closed non-Archimedean field whose residue field $\td k$ is not algebraic over a finite field. Consider an elliptic curve $E$ over $k$ with good reduction over the ring of integers $k^\circ$.  Choose points $P,Q\in E(k)$ with distinct reductions $\td P,\td Q$ in the special fiber $E_s$ of the smooth model of $E$, and suppose that $[\td Q]-[\td P]$ is not a torsion element of $\Pic^0(E_s)$.  We may identify the residue discs $\red\inv(\td P),\red\inv(\td Q)$ over $\td P,\td Q$, respectively, with copies of the open unit disc $B = \{x\colon|x|<1\}$, such that $P$ and $Q$ are identified with $0\in B(k)$.  Define $h\colon E^\an\setminus\{P,Q\}\to\R$ by
\[ h(x) =
  \begin{cases}
    -\log|x| & \text{if $x\in\red\inv(\td P)$,}\\
    \log|x| & \text{if $x\in\red\inv(\td Q)$,}\\
    0 & \text{otherwise.}
  \end{cases}
\]
This is a harmonic function on $E^\an$ (in the sense of Thuillier~\cite{thuillier05:thesis}) because the outgoing slopes of $h$ at all points of $E^\an$ sum to zero.  This sum is only interesting at the unique point $\zeta$ of $E^\an$ reducing to the generic point of $E_s$, at which the slope in the direction of $P$ is $+1$, and the slope in the direction of $Q$ is $-1$.

However, this function is not smooth.  To see that, let $\sH$ denote the sheaf of harmonic functions and $\sF$ the subsheaf of smooth harmonic functions. As explained in~\cite[5.2]{wanner16:harmonic_functions}, it follows from~\cite[Lemme~2.3.22]{thuillier05:thesis}  that the quotient $\sH_\zeta/\sF_\zeta$ is canonically isomorphic to $\Pic^0(E_s)\tensor\R$.  The image of the germ of $h$ at $\zeta$ is equal to the class of $[\td Q]-[\td P]$, which is nonzero by hypothesis.  See~\otherfile{Section~\ref*{I-abelian threefold}} for a more detailed example.

In~\cite{gubler_rabinoff_jell:harmonic_trop}, we modify the construction of~\cite{chambert_ducros12:forms_courants} in order to force harmonic functions to be smooth; the resulting forms are called \defi{weakly smooth}. This effectively means that one is allowed to pull back smooth Lagerberg forms via harmonic maps $h\colon W\to\R^n$ on compact strictly analytic domains $W$ of $X$, as in the smooth case above. Every smooth form on $X$  is weakly smooth, and we show in~\cite{gubler_rabinoff_jell:harmonic_trop} that the weakly smooth forms retain essentially all of the properties of smooth forms.
	
\subsection{Results on Curves} \label{sec:results on curves}
The goal of the current paper is to verify that weakly smooth forms have good cohomological properties for arbitrary compact (rig-)smooth curves.  The following theorem shows that the dimensions of the Dolbeault cohomology groups are as expected from the theory of compact Riemann surfaces.
	
\begin{oneoffthm*}[theorem]{Theorem~\ref{thm:dolbeault.cohom.curve}}[Dolbeault Cohomology]
  Let $k$ be a non-Archimedean field with a nontrivial valuation.  
  Let $X$ be a connected, compact, rig-smooth curve, and let $g = h^1_{\sing}(|X|)$ be the first Betti number of its underlying topological space.  Let $H^{i,j}(X)$ denote the Dolbeault cohomology groups of $X$ with respect to its weakly smooth forms, and let $h^{i,j} = \dim_\R H^{i,j}(X)$.  Then we have:
  \begin{center}
    \begin{tikzpicture}
      \node (a) at (0,0) {If $\del X=\emptyset$:};
      \node[below=5mm] (b) at (a) {$h^{1,1} = 1$};
      \node[below left=1mm] (c) at (b.south west) {$h^{1,0} = g$};
      \node[below right=1mm] (d) at (b.south east) {$h^{0,1} = g$};
      \node[below right=1mm] (e) at (c.south east) {$h^{0,0} = 1$};
      \coordinate (x) at ($(a.south east)!(current bounding box.north west)!(a.south west)$);
      \coordinate (y) at ($(a.south east)!(current bounding box.north east)!(a.south west)$);
      \draw[thick] (x) -- (y);
    \end{tikzpicture}\qquad
    \begin{tikzpicture}
      \node (a) at (0,0) {If $\del X\neq\emptyset$:};
      \node[below=5mm] (b) at (a) {$h^{1,1} = 0$};
      \node[below left=1mm] (c) at (b.south west) {$h^{1,0} = g+\#\del X-1$};
      \node[below right=1mm] (d) at (b.south east) {$h^{0,1} = g\phantom{+\#\del X-1}$};
      \node[below right=1mm] (e) at (c.south east) {$h^{0,0} = 1$};
      \coordinate (x) at ($(a.south east)!(current bounding box.north west)!(a.south west)$);
      \coordinate (y) at ($(a.south east)!(current bounding box.north east)!(a.south west)$);
      \draw[thick] (x) -- (y);
    \end{tikzpicture}
  \end{center}
\end{oneoffthm*}

Here the boundary $\del X$ is defined by Berkovich~\cite[3.1]{berkovic90:analytic_geometry}; it is empty precisely when $X$ is proper (equivalently, projective).

When $X$ is a proper curve over $k$, there is an integration map $\int_X$ on the space of weakly smooth $(1,1)$-forms on $X$; by Stokes' theorem, this gives rise to pairings
\begin{equation}\label{eq:poincare.pairing.curves.intro}
  H^{0,0}(X)\times H^{1,1}(X)\To\R, \qquad
  H^{1,0}(X)\times H^{0,1}(X)\To\R.
\end{equation}
The following theorem is a version of Poincar\'e duality in this setting.

\begin{oneoffthm*}[theorem]{Theorem~\ref{thm:poincare.duality}}[Poincar\'e Duality]
  If $X$ is a smooth, proper curve over $k$,  then the pairings~\eqref{eq:poincare.pairing.curves.intro} are perfect.
\end{oneoffthm*}

\subsection{Differential Forms on Skeletons}
To compute the Dolbeault cohomology of a compact rig-smooth curve $X$ over the non-trivially valued non-Archimedean field $k$, we will give an alternative characterization of its weakly smooth forms using the theory of skeletons from~\cite{berkovic99:locally_contractible_I} and~\cite{thuillier05:thesis}.  To a semistable model $\fX$ of $X$ over the valuation ring $k^\circ$, one associates a canonical metric graph $\Sigma_\fX$,  called its \defi{skeleton}.  There is a canonical embedding $\Sigma_\fX\inject X$ and a deformation retraction $\tau_\fX\colon X\to\Sigma_\fX$.  
We will see that $\Sigma$ is canonically a weighted metric graph with boundary equal to $\del X$, which we will study in Section~\ref{Sec:curves-skeletons}. In Section~\ref{Sec:forms-curves}, we give a procedure for pulling back smooth $(p,q)$-forms on  $\Sigma_\fX$ under the retraction $\tau_\fX$.  We will show that the resulting forms are weakly smooth forms in the sense of~\cite{gubler_rabinoff_jell:harmonic_trop}. In fact, we have the following  converse after potentially passing to a finite extension of $k$.

\begin{oneoffthm*}[theorem]{Theorem~\ref{thm:pullback-form-skeleton}}
  Let $X$ be a compact rig-smooth curve over $k$, let $p,q\in\{0,1\}$, and let $\eta\in\cA^{p,q}(X)$ be a weakly smooth form.  There exists a finite Galois extension $k'/k$, a strictly semistable model $\fX'$ of $X' = X\tensor_k{k'}$, and a smooth form $\omega\in\cA^{p,q}(\Sigma_{\fX'},\del\Sigma_{\fX'})$ on the graph $\Sigma_{\fX'}$, such that $\eta'=\tau_{\fX'}^*\omega$, where $\eta'$ is the pullback of $\eta$ to $X'$.
\end{oneoffthm*}

The definition of weakly smooth forms involves a pullback from $\R^n$ via harmonic maps $h\colon W\to\R^n$ on strictly analytic domains $W$ of $X$.  The above theorem allows us to work with forms pulled back from a skeleton, which is an \emph{intrinsic} part of the structure theory of $X$.  This is only possible due to the semistable reduction theorem; in higher dimensions, the definition via harmonic maps is much more general. 

The following theorem relates the Dolbeault cohomologies.

\begin{oneoffthm*}[theorem]{Theorem~\ref{thm:pullback-skeleton-isom}}
  Let $X$ be a compact rig-smooth curve with a strictly semistable model $\fX$.  Suppose that $\fX_s$ contains no connected component that is both smooth and proper.  Then the pullback from the skeleton induces isomorphisms
  \[ \tau_\fX^*\colon H^{p,q}(\Sigma_\fX,\del\Sigma_\fX) \isom H^{p,q}(X) \]
  for all $p,q\in\{0,1\}$.
\end{oneoffthm*}

This together with the semistable reduction theorem and Proposition~\ref{prop:introprop:dolbeault.graphs.intro} lead to the description of the Dolbeault cohomology of $X$ given in \secref{sec:results on curves}.
Much of the technical difficulty in this paper is in proving that all of our constructions are functorial with respect to morphisms of curves and graphs and extension of the ground field.  This allows us to relate forms on different skeletons of the same curve, and to compute forms after a Galois extension of the ground field.

\subsection{Acknowledgments}
We are very grateful to Robert Bryant for providing us with an argument for Lemma~\ref{lem:extend.smooth}, which is a key fact needed to relate Lagerberg forms with differential forms on graphs.


\section{Differential Forms on Graphs}
\label{Sec:forms-graphs}

In this section we first fix our notions regarding weighted metric graphs with boundary, then define smooth forms on graphs and compute the associated Dolbeault cohomology groups.

\subsection{Weighted Metric Graphs With Boundary}  \label{sec:weighted metric graphs}
Our definitions in this section are essentially the same as Thuillier's~\cite{thuillier05:thesis}, except that we use graph-theoretic language.

\begin{defn}\label{def:metric.graph}
  A \defi{metric graph} is a (finite) multigraph $\Sigma$, with no loop edges, endowed with a parameterization $t_e\colon [0,\ell(e)]\isom e$ for every oriented edge $e$, with $0$ mapping to the tail of $e$ and $\ell(e)$ to the head.  The quantity $\ell(e) > 0$ is the \defi{length} of $e$.  If $\bar e$ is the edge $e$ with the opposite orientation, then we require $t_{\bar e}(x) = t_e(\ell(e)-x)$; in particular, $\ell(e)=\ell(\bar e)$.
\end{defn}

Note that we do not assume a metric graph to be connected, and that isolated vertices are allowed.  Our graphs will have edge weights, arising from the degree of the singular points of a semistable model of a curve.

\begin{defn}\label{def:weighted.graph}
  A \defi{weight} on a metric graph $\Sigma$ is a function $e\mapsto w(e)\in\Z_{>0}$ on the edges of $\Sigma$, satisfying $w(e)=w(\bar e)$.  A \defi{weighted metric graph} is a metric graph equipped with a weight.  The edge weights are said to be \defi{trivial} if $w(e) = 1$ for all edges.

  If $\Sigma$ is a weighted metric graph, its \defi{unweighting} is the weighted metric graph $\Sigma_0$ with the same underlying graph as $\Sigma$, with trivial edge weights, and with edge parameterizations $t_{0,e}\colon[0,\ell(e)/w(e)]\isom e$ defined by $t_{0,e}(x) = t_e(w(e)x)$.  The unweighting is equipped with a canonical map $\nu\colon\Sigma_0\to\Sigma$ that is the identity on underlying graphs.
\end{defn}

In particular, an edge $e$ of $\Sigma$ has length $\ell_0(e) = \ell(e)/w(e)$ in $\Sigma_0$.

\begin{rem}\label{rem:only.use.unweighting}
  The entire theory of forms on graphs will only depend on the lengths and weights of weighted metric graphs via the quotient $\ell_0(e) = \ell(e)/w(e)$, so we will always be able to pass to unweightings in the proofs.  In particular, it is logically redundant to specify both edge lengths and edge weights for purely graph-theoretic considerations.  However, it is necessary to retain both in order to make the connection with non-Archimedean geometry.  For instance, the slope of a linear function on an edge depends on the edge length but not the weight, and harmonic functions coming from non-Archimedean geometry have integer slopes.
\end{rem}

Our metric graphs will have a boundary, corresponding to the boundary of a Berkovich curve.

\begin{defn}\label{def:graph.boundary}
  A \defi{metric graph with boundary} is a metric graph $\Sigma$ along with a distinguished set $\del\Sigma\subset V(\Sigma)$.  The set $\del\Sigma$ is the \defi{boundary} of $\Sigma$, and $\Sigma\setminus\del\Sigma$ is the \defi{interior}.  If $(\Sigma,\del\Sigma)$ is a weighted metric graph with boundary, then we regard its unweighting $\Sigma_0$ as a metric graph with boundary $\del\Sigma_0=\del\Sigma$.
\end{defn}

Subdividing a metric graph will correspond to blowing up smooth points on the special fiber of a semistable model of a curve.

\begin{defn}\label{def:subdivision}
  A \defi{subdivision} of a weighted metric graph $\Sigma$ is the weighted metric graph $\Sigma'$ obtained by adding a finite number of vertices to the interiors of the edges of $\Sigma$.  The weights, edge parameterizations, and edge lengths of $\Sigma'$ are defined in the evident manner.
\end{defn}

The choice of vertices will play little role in this section, so we will often identify a weighted metric graph $\Sigma$ with its geometric realization, and we will pass freely to subdivisions of $\Sigma$.

\begin{notn}\label{notn:graph.notation}
We will use the following notations concerning graphs.\\[1mm]
\null\begin{tabular}{rl}
  \hbox to 1.5cm{\hfil$\Sigma$} & A weighted metric graph with no loop edges. \\
  $E(\Sigma)$ & The set of oriented edges of $\Sigma$. \\
  $E^+(\Sigma)$ & The set of unoriented edges of $\Sigma$. \\
  $V(\Sigma)$ & The set of vertices of $\Sigma$. \\
  $\bar e$ & The edge $e$ with the opposite orientation. \\
  $\ell(e)$ & The length of the edge $e$. \\
  $e^+$ & The head vertex of an edge $e$. \\
  $e^-$ & The tail vertex of an edge $e$. \\
  $t_e$ & A fixed parameterization $[0,\ell(e)]\to e$, with $0\mapsto e^-$ and $\ell(e)\mapsto e^+$. \\
  $w(e)$ & The weight of the edge $e$. \\
  $\Sigma_0$ & The unweighting of $\Sigma$. \\
  $\ell_0(e)$ & $= \ell(e)/w(e)$, the length of $e$ in $\Sigma_0$. \\
  $\nu$ & $\colon\Sigma_0\to\Sigma$, the identity on underlying graphs. \\
  $\del\Sigma$ & A subset of $V(\Sigma)$, called the \defi{boundary}. \\
\end{tabular}
\end{notn}

\subsection{Smooth forms: definitions}  \label{sec:smooth forms}
Fix a weighted metric graph with boundary $(\Sigma,\del\Sigma)$.
For an edge $e\subset\Sigma$, we say that a function $f\colon e\to\R$ is \emph{smooth} if $f\circ t_e\colon[0,\ell(e)]\to\R$ is smooth (smooth on one side at the endpoints).  We write $\d f/\d t_e(v)$ for the derivative of $f\circ t_e$ at $x = t_e\inv(v)$.  If $v\in\Sigma$ is a vertex of valency $2$ with outgoing edges $e_1,e_2$, and $f_i\colon e_i\to\R$ are smooth functions for $i=1,2$, we say that $(f_1,f_2)$ is \defi{smooth at $v$} if
\[ w(e_1)^n\,\frac{\d^n f_1}{\d t_{e_1}^n}(v) = w(e_2)^n\,(-1)^n \frac{\d^n f_2}{\d t_{e_2}^n}(v) \]
for all $n\geq 0$.  This means that $f_1,f_2$ glue to give a smooth function on $e_1\cup e_2$ in the obvious sense, where $e_1,e_2$ are regarded as edges \emph{in the unweighting}.

\begin{defn}\label{def:smooth.on.graph}
  A \defi{smooth function} on $(\Sigma,\del\Sigma)$ is a continuous function $f\colon\Sigma\to\R$ which is smooth when restricted to each edge, and that satisfies the following compatibility conditions at a vertex $v$ in $\Sigma\setminus\del\Sigma$.
  \begin{enumerate}
  \item If $v$ has valency $1$, then $f$ is constant in a neighborhood of $v$.
  \item If $v$ has valency $2$ with outgoing edges $e_1,e_2$, then $(f|_{e_1},f|_{e_2})$ is smooth at $v$.
  \item If $v$ has valency greater than $2$, then we require
    \[ \sum_{v=e^-}w(e)\,\frac{\d f}{\d t_e}(v) = 0. \]
    (The sum is taken over all outgoing edges at $v$.)
  \end{enumerate}
  The $\R$-algebra of smooth functions on $(\Sigma,\del\Sigma)$ is denoted $\cA^{0,0}(\Sigma,\del\Sigma)$, or by $\cA^{0,0}(\Sigma)$ if $\del\Sigma=\emptyset$.
\end{defn}

\begin{rem}\  \label{rem:partition of unity}

  \begin{enumerate}
  \item There are no conditions on $f$ at vertices in $\del\Sigma$ (other than continuity), or at isolated vertices.
  \item Condition~(3) says that the weighted outgoing slopes of $f$ at $v$ sum to zero.
  \item It is clear that smoothness can be checked in a neighborhood of a point, and that it is insensitive to subdivisions of $\Sigma$.
  \item The same argument as in Analysis shows that there is a smooth partition of unity subordinated to a given open covering of $\Sigma$.
\end{enumerate}
\end{rem}

\begin{defn}\label{def:diff.form.graph}
  A \defi{smooth $(1,0)$-form} on $(\Sigma,\del\Sigma)$ is the data $\omega = (f_e\,\d't_e)_{e\in E(\Sigma)}$, where $f_e\colon e\to\R$ is a smooth function and $\d't_e$ is a formal symbol.  We require $f_{\bar e} = -f_e$, and the following compatibility conditions at a vertex $v$ in $\Sigma\setminus\del\Sigma$.
  \begin{enumerate}
  \item If $v$ has valency $1$ with outgoing edge $e$, then $f_e$ is zero in a neighborhood of $v$.
  \item If $v$ has valency $2$ with outgoing edges $e_1,e_2$, then $(w(e_1)f_{e_1},-w(e_2)f_{e_2})$ is smooth at $v$.
  \item If $v$ has valency greater than $2$, then we require
    \[ \sum_{v=e^-} w(e)\,f_e(v) = 0. \]
  \end{enumerate}
  The space of smooth $(1,0)$-forms on $\Sigma$ is denoted by $\cA^{1,0}(\Sigma,\del\Sigma)$, or by $\cA^{1,0}(\Sigma)$ if $\del\Sigma=\emptyset$.  This is naturally an $\cA^{0,0}(\Sigma,\del\Sigma)$-module by $g\cdot(f_e\,\d't_e) = (gf_e\,\d't_e)$.  
  There is a derivation
  \[ \d'\colon\cA^{0,0}(\Sigma,\del\Sigma)\To\cA^{1,0}(\Sigma,\del\Sigma) \qquad
  \d'f = \left(\frac{\d f}{\d t_e}\,\d't_e \right). \]

  We define the space $\cA^{0,1}(\Sigma,\del\Sigma)$ (resp.\ $\cA^{0,1}(\Sigma)$ if $\del\Sigma=\emptyset$) of smooth $(0,1)$-forms and the derivation $\d'':\cA^{0,0}(\Sigma,\del\Sigma)\to\cA^{0,1}(\Sigma,\del\Sigma)$ in exactly the same way, replacing the symbol $\d't_e$ with $\d''t_e$.
\end{defn}

\begin{rem*}\

  \begin{enumerate}
  \item There are no conditions on $(f_e\,\d't_e)$ or $(f_e\,\d''t_e)$ at vertices in $\del\Sigma$.
  \item The condition $f_{\bar e} = -f_e$ should be interpreted as saying that $\d' t_{\bar e} = -\d' t_e$, which is reasonable as $t_{\bar e}(x) = t_e(\ell(e)-x)$.
  \item The spaces $\cA^{1,0}(\Sigma,\del\Sigma)$ and $\cA^{0,1}(\Sigma,\del\Sigma)$ are insensitive to subdivisions of $\Sigma$.
  \end{enumerate}
\end{rem*}

\begin{defn}\label{def:diff.11.form.graph}
  A \defi{smooth $(1,1)$-form} on $(\Sigma,\del\Sigma)$ is the data $\omega = (f_e\,\d' t_e\d''t_e)_{e\in E(\Sigma)}$, where $f_e\colon e\to\R$ is a smooth function and $\d't_e\d''t_e$ is a formal symbol.  We require that $f_{\bar e} = f_e$, and the following compatibility conditions at a vertex $v\in\Sigma\setminus\del\Sigma$.
  \begin{enumerate}
  \item If $v$ has valency $1$ with outgoing edge $e$, then $f_e$ is zero in a neighborhood of $v$.
  \item If $v\in\Sigma\setminus\del\Sigma$ is a vertex of valency $2$ with outgoing edges $e_1,e_2$, then the pair $(w(e_1)^2f_{e_1},w(e_2)^2f_{e_2})$ is smooth at $v$.
  \end{enumerate}
  The space of smooth $(1,1)$-forms on $\Sigma$ is denoted $\cA^{1,1}(\Sigma,\del\Sigma)$, or by $\cA^{1,1}(\Sigma)$ if $\del\Sigma=\emptyset$.  This is naturally an $\cA^{0,0}(\Sigma,\del\Sigma)$-module by $g\cdot(f_e\,\d't_e\d''t_e) = (gf_e\,\d't_e\d''t_e)$.  There is a wedge product
  \[ \wedge\colon\cA^{1,0}(\Sigma,\del\Sigma)\times\cA^{0,1}(\Sigma,\del\Sigma)\to\cA^{1,1}(\Sigma,\del\Sigma) \]
  defined by
  \[ (f_e\,\d't_e)\wedge(g_e\,\d''t_e) = (f_eg_e\,\d't_e\d''t_e), \]
  and there are natural maps
  \[ \d''\colon\cA^{1,0}(\Sigma,\del\Sigma)\To\cA^{1,1}(\Sigma,\del\Sigma)\qquad
    \d'\colon\cA^{0,1}(\Sigma,\del\Sigma)\To\cA^{1,1}(\Sigma,\del\Sigma)
  \]
  defined by
  \[ \d''(f_e\d't_e) = \left( -\frac{\d f_e}{\d t_e}\,\d't_e\d''t_e \right)
    \qquad
    \d'(f_e\d''t_e) = \left( \frac{\d f_e}{\d t_e}\,\d't_e\d''t_e \right),
  \]
  respectively.
\end{defn}

\begin{rem*}\

  \begin{enumerate}
  \item There are no conditions on $(f_e\,\d't_e\d''t_e)$ at vertices in $\del\Sigma$ or at vertices of valency greater than $2$.
  \item The condition $f_{\bar e} = f_e$ is consistent with the convention that $\d't_{\bar e}=-\d't_e$ and $\d''t_{\bar e}=-\d''t_e$.
  \item The space $\cA^{1,1}(\Sigma,\del\Sigma)$ is insensitive to subdivisions of $\Sigma$.
  \item The \defi{Lagerberg involution} is the involution of bigraded differential algebras
\[ J\colon\cA^{p,q}(\Sigma,\del\Sigma)\isom\cA^{q,p}(\Sigma,\del\Sigma) \]
defined by exchanging $\d'$ with $\d''$.  (Note that $J = 1$ on $\cA^{0,0}$ and $J = -1$ on $\cA^{1,1}$.)
  \end{enumerate}
\end{rem*}

With the product structures defined in Definitions~\ref{def:smooth.on.graph}, \ref{def:diff.form.graph}, and~\ref{def:diff.11.form.graph}, we get a uniquely defined wedge product on $\cA^{\bullet,\bullet}(\Sigma,\del\Sigma)=\Dsum\cA^{p,q}(\Sigma,\del\Sigma)$ making $\cA^{\bullet,\bullet}(\Sigma,\del\Sigma)$ into an alternating bigraded algebra.  The differential $\d'$ satisfies the Leibniz rule
\[ \d'(\omega\wedge\eta) = \d'\omega\wedge\eta + (-1)^p\omega\wedge\d'\eta \]
for $\omega\in\cA^{p,q}(\Sigma,\del\Sigma)$ and $\eta\in\cA^{r,s}(\Sigma,\del\Sigma)$, and likewise for $\d''$.

\begin{rem}
  Removing the isolated vertices from a graph does not affect its $(p,q)$-forms for $p+q\geq 1$.  In particular, if $\Sigma = \{x\}$ then $\cA^{1,0}(\Sigma,\del\Sigma)=\cA^{0,1}(\Sigma,\del\Sigma)=\cA^{1,1}(\Sigma,\del\Sigma)=0$.
\end{rem}

The definition of smooth forms only depends on the quotients $\ell_0(e) = \ell(e)/w(e)$; this is made explicit by the following lemma.

\begin{lem}\label{lem:unweighting.forms}
  Let $(\Sigma,\del\Sigma)$ be a weighted metric graph with boundary, and let $\Sigma_0$ be its unweighting.  Define $\nu^*\colon\cA^{p,q}(\Sigma,\del\Sigma)\to\cA^{p,q}(\Sigma_0,\del\Sigma_0)$ as follows:
  \begin{align*}
    \nu^*(f) &= f & f&\in\cA^{0,0}(\Sigma,\del\Sigma) \\
    \nu^*(f_e\,\d't_{e}) &= (w(e)f_e\,\d't_{0,e}) &
                                                    (f_e\,\d't_e)&\in\cA^{1,0}(\Sigma,\del\Sigma) \\
    \nu^*(f_e\,\d''t_e) &= (w(e)f_e\,\d''t_{0,e}) &
                                                    (f_e\,\d''t_{e})&\in\cA^{0,1}(\Sigma,\del\Sigma) \\
    \nu^*(f_e\,\d't_{e}\d''t_{e}) &= (w(e)^2f_e\,\d't_{0,e}\d''t_{0,e})  &
                                                                           (f_e\,\d't_{e}\d''t_{e})&\in\cA^{1,1}(\Sigma,\del\Sigma).
  \end{align*}
  Then $\nu^*$ is an isomorphism of differential bigraded algebras.
\end{lem}

The claim follows easily from $\frac{\d f}{\d t_{0,e}}=w(e)\frac{\d f}{\d t_e}$, using the definition of the unweighting.  In Section~\ref{Sec:harmonicity} we will interpret $\nu^*$ as the pullback with respect to the map $\nu:\Sigma_0 \to \Sigma$ that is the identity on the underlying graphs: see Remark~\ref{rem:unweighting.and.harmonic}.

\begin{rem} \label{rem:comparison to polyhedral spaces}
	 The $(p,q)$-forms on a weighted metric graph $\Sigma$ with boundary are similar to the Lagerberg forms (also called superforms) on polyhedral spaces  introduced in~\cite{jell_shaw_smacka19:superforms}. A metric graph is  a compact polyhedral space of dimension $\leq 1$. The difference is that the boundary points for $\Sigma$ are connected by edges of finite lengths, while the boundary points of one-dimensional polyhedral spaces  are at the end of edges of infinite length; they play the role of an interior edge of $\Sigma$ of valency $1$. Our definitions are inspired by applications to skeletons of non-archimedean analytic curves starting in Section~\ref{Sec:curves-skeletons}. 
\end{rem}

\subsection{Integration}\label{sec:integration}
Now we define the integral of a smooth $(1,1)$-form and prove Stokes' theorem.

\begin{defn}\label{def:integration.11}
  For $\omega = (f_e\,\d't_e\d''t_e)\in\cA^{1,1}(\Sigma,\del\Sigma)$ we define
  \[ \int_\Sigma\omega = \frac 12\sum_{e\in E(\Sigma)}w(e)\int_0^{\ell(e)} f_e\circ t_e(x)\,\d x. \]
\end{defn}

\noindent
Here the sum is taken over \emph{oriented} edges.  Since $f_e = f_{\bar e}$ we have
\[ \int_0^{\ell(e)}f_e\circ t_e(x)\,\d x = \int_0^{\ell(e)}f_{\bar e}\circ t_{\bar e}(x)\,\d x,\]
so alternatively, one can drop the factor of $1/2$ by summing over each edge once with a choice of orientation.

\begin{defn}\label{def:integration.10}
  For $\eta' = (f_e\,\d't_e)\in\cA^{1,0}(\Sigma,\del\Sigma)$ and
  $\eta'' = (g_e\,\d''t_e)\in\cA^{0,1}(\Sigma,\del\Sigma)$
  we define
  \[ \int_{\del\Sigma}\eta' = \sum_{v\in\del\Sigma}\sum_{v=e^-}w(e)\,f_e(v)
    \qquad
    \int_{\del\Sigma}\eta'' = \sum_{v\in\del\Sigma}\sum_{v=e^+}w(e)\,g_e(v),
  \]
  where the second sum in $\int_{\del\Sigma}\eta'$ is taken over all outgoing edges at $v$, and the second sum in $\int_{\del\Sigma}\eta''$ is taken over all incoming edges at $v$.
\end{defn}

For a vertex $v\notin\del\Sigma$ we have $\sum_{v=e^-}w(e)\,f_e(v) = 0 = \sum_{v=e^+}w(e)\,f_e(v)$, so we may equivalently define
\begin{equation}\label{eq:bdy.int.2}
  \begin{split}
    \int_{\del\Sigma}\eta' &= \sum_{v\in V(\Sigma)}\sum_{v=e^-}w(e)\,f_e(v)
    = \sum_{e\in E(\Sigma)} w(e)\,f_e(e^-) \\
    \int_{\del\Sigma}\eta'' &= \sum_{v\in V(\Sigma)}\sum_{v=e^+}w(e)\,g_e(v)
    = \sum_{e\in E(\Sigma)} w(e)\,f_e(e^+).
  \end{split}
\end{equation}

The following lemma shows that integrals are not changed by unweighting.  We will see later that it is a special case of Lemma~\ref{lem:pullback.integrate}.

\begin{lem}\label{lem:integral.unweighting}
  Let $(\Sigma,\del\Sigma)$ be a weighted metric graph with boundary, and let $\Sigma_0$ be the unweighting.  Then
  \[ \int_{\Sigma_0}\nu^*\omega = \int_\Sigma\omega \qquad
    \int_{\del\Sigma_0}\nu^*\eta' = \int_{\del\Sigma}\eta' \qquad
    \int_{\del\Sigma_0}\nu^*\eta'' = \int_{\del\Sigma}\eta''
  \]
  for all $\omega\in\cA^{1,1}(\Sigma,\del\Sigma),~\eta'\in\cA^{1,0}(\Sigma,\del\Sigma),$ and $\eta''\in\cA^{0,1}(\Sigma,\del\Sigma)$.
\end{lem}

\begin{proof}
  Let $\omega = (f_e\,\d't_e\d''t_e)\in\cA^{1,1}(\Sigma,\del\Sigma)$, so $\nu^*\omega = (w(e)^2f_e\,\d't_{0,e}\d''t_{0,e})$.  For an edge $e$ of $\Sigma$, we have
  \[\begin{split}
      \int_0^{\ell(e)/w(e)}w(e)^2 f_e\circ t_{0,e}(x)\,\d x
      &= w(e)\int_0^{\ell(e)/w(e)} f_e\circ t_e(w(e)x)\,\d(w(e)x) \\
      &= w(e)\int_0^{\ell(e)} f_e\circ t_e(x)\,\d x,
    \end{split}\]
  which proves that $\int_{\Sigma_0}\nu^*\omega = \int_\Sigma\omega$.  The proofs for $\eta'$ and $\eta''$ are immediate.
\end{proof}

Now we prove Stokes' Theorem.

\begin{thm}[Stokes' Theorem]\label{thm:stokes}
  For $\omega\in\cA^{1,0}(\Sigma,\del\Sigma)$ and $\eta\in\cA^{0,1}(\Sigma,\del\Sigma)$
  we have
  \[ \int_\Sigma\d''\omega = \int_{\del\Sigma}\omega \sptxt{and}
    \int_\Sigma\d'\eta = \int_{\del\Sigma}\eta.
  \]
\end{thm}

\begin{proof}
  We only prove the first equality, as the second is the same.  By Lemma~\ref{lem:integral.unweighting}, we may assume that $\Sigma$ has trivial edge weights.  We have
  \[\begin{split}
    \int_\Sigma\d''\omega &= \frac12\sum_e \int_0^{\ell(e)}-\frac{\d(f_e\circ t_e)}{\d x}\,\d x
    = -\frac12\sum_e \bigl( f_e(e^+) - f_e(e^-) \bigr) \\
    &= \frac12\sum_e \bigl( f_{\bar e}(\bar e{}^-) + f_e(e^-) \bigr)
    = \sum_e f_e(e^-)
    = \int_{\del\Sigma}\omega,
  \end{split}\]
  where the last equality holds by~\eqref{eq:bdy.int.2}.
\end{proof}

\subsection{Dolbeault cohomology of graphs}
In this subsection we define and compute the Dolbeault cohomology groups of a weighted metric graph with boundary $(\Sigma,\del\Sigma)$.

\begin{defn}\label{def:dolbeault.graphs}
  The \defi{Dolbeault cohomology groups} of $(\Sigma,\del\Sigma)$ with respect to $\d''$ are
  \[\begin{split}
      H^{0,0}(\Sigma,\del\Sigma) &= \mathrm{\phantom{co}ker}\bigl( \d''\colon\cA^{0,0}(\Sigma,\del\Sigma)\To\cA^{0,1}(\Sigma,\del\Sigma) \bigr) \\
      H^{1,0}(\Sigma,\del\Sigma) &= \mathrm{\phantom{co}ker}\bigl( \d''\colon\cA^{1,0}(\Sigma,\del\Sigma)\To\cA^{1,1}(\Sigma,\del\Sigma) \bigr) \\
      H^{0,1}(\Sigma,\del\Sigma) &= \coker\bigl( \d''\colon\cA^{0,0}(\Sigma,\del\Sigma)\To\cA^{0,1}(\Sigma,\del\Sigma) \bigr) \\
      H^{1,1}(\Sigma,\del\Sigma) &= \coker\bigl( \d''\colon\cA^{1,0}(\Sigma,\del\Sigma)\To\cA^{1,1}(\Sigma,\del\Sigma) \bigr).
    \end{split}\]
  We also write $H^{i,j}(\Sigma)$ if $\del\Sigma=\emptyset$, and we write $h^{i,j}(\cdots) = \dim H^{i,j}(\cdots)$.
\end{defn}

The \emph{genus} of a connected (weighted metric) graph $\Sigma$ is the dimension of the first singular cohomology group $H^1_{\sing}(\Sigma,\R)$; this is equal to $\#E^+(\Sigma)-\#V(\Sigma)+1$.  The main result of this section is as follows.

\begin{prop}\label{prop:dolbeault.graphs}
  Let $(\Sigma,\del\Sigma)$ be a connected weighted metric graph with boundary, and let $g$ be the genus of $\Sigma$.  Suppose that $\Sigma$ does not consist of a single vertex.  Then $h^{i,j} = h^{i,j}(\Sigma,\del\Sigma)$ are equal to:
  \begin{center}
    \begin{tikzpicture}
      \node (a) at (0,0) {If $\del\Sigma=\emptyset$:};
      \node[below=5mm] (b) at (a) {$h^{1,1} = 1$};
      \node[below left=1mm] (c) at (b.south west) {$h^{1,0} = g$};
      \node[below right=1mm] (d) at (b.south east) {$h^{0,1} = g$};
      \node[below right=1mm] (e) at (c.south east) {$h^{0,0} = 1$};
      \coordinate (x) at ($(a.south east)!(current bounding box.north west)!(a.south west)$);
      \coordinate (y) at ($(a.south east)!(current bounding box.north east)!(a.south west)$);
      \draw[thick] (x) -- (y);
    \end{tikzpicture}\qquad
    \begin{tikzpicture}
      \node (a) at (0,0) {If $\del\Sigma\neq\emptyset$:};
      \node[below=5mm] (b) at (a) {$h^{1,1} = 0$};
      \node[below left=1mm] (c) at (b.south west) {$h^{1,0} = g+\#\del\Sigma-1$};
      \node[below right=1mm] (d) at (b.south east) {$h^{0,1} = g\phantom{+\#\del\Sigma-1}$};
      \node[below right=1mm] (e) at (c.south east) {$h^{0,0} = 1$};
      \coordinate (x) at ($(a.south east)!(current bounding box.north west)!(a.south west)$);
      \coordinate (y) at ($(a.south east)!(current bounding box.north east)!(a.south west)$);
      \draw[thick] (x) -- (y);
    \end{tikzpicture}
  \end{center}
  More precisely, if $\del\Sigma=\emptyset$ then $\d''\cA^{1,0}(\Sigma) = \ker\int_\Sigma$.
\end{prop}

\begin{rem}
  If $\Sigma = \{x\}$ is a single vertex then Proposition~\ref{prop:dolbeault.graphs} remains true (with $g=0$), except that $h^{1,1}=0$ in both cases.
\end{rem}

In order to compute the numbers $h^{p,q}$, by Lemmas~\ref{lem:unweighting.forms} and~\ref{lem:integral.unweighting} we may and do assume that all edge weights are equal to $1$.

\subsubsection{Computation of $h^{0,0}$:}\label{sec:computation-h0-0}
Since $H^{0,0}(\Sigma,\del\Sigma)$ consists of all constant functions $f\colon\Sigma\to\R$, it has dimension one.

\subsubsection{Computation of $h^{1,1}$:}\label{sec:computation-h1-1} To compute $h^{1,1}(\Sigma,\del\Sigma)$, first we assume that $\del\Sigma = \emptyset$.  By Stokes' theorem we have $\d''\cA^{1,0}(\Sigma) \subset \ker\int_\Sigma$, and  $\int_\Sigma\colon\cA^{1,1}(\Sigma)\to\R$ is surjective (integrate a bump function on an edge).  Hence the equality $\d''\cA^{1,0}(\Sigma) = \ker\int_\Sigma$ follows from $h^{1,1}(\Sigma)=1$.  Hence we must show that if $\int_\Sigma\omega = 0$ then $\omega = d''\eta$ for some $\eta\in\cA^{1,0}(\Sigma)$.

  Let $\omega = (f_e\,\d't_e\d''t_e)\in\cA^{1,1}(\Sigma)$, and suppose that $\int_\Sigma\omega = 0$.  For each edge $e$ define $g_e\colon e\to\R$ by $g_e\circ t_e(x) = -\int_0^x f_e\circ t_e(t)\,\d t + C_e$, where $C_e$ is a constant that has yet to be determined.  If $\eta = (g_e\,\d't_e)$ is an element of $\cA^{1,0}(\Sigma)$ then $\d''\eta = \omega$ by construction, so we must choose the constants $C_e$ such that $(g_e\,\d't_e)\in\cA^{1,0}(\Sigma)$.  For an edge $e$ we let
  $\mu_e = \int_0^{\ell(e)}f_e\circ t_e(x)\,\d x,$ so $\mu_{\bar e} = \mu_e$ and $\sum_e \mu_e = 2\int_\Sigma\omega = 0$.
  One checks that
  \[ g_e + g_{\bar e} = -\mu_e + C_e + C_{\bar e}, \]
  so the condition $g_{\bar e} = -g_e$ is equivalent to $C_e + C_{\bar e} = \mu_e$.  If $v$ is a vertex with outgoing edge $e$ then $g_e(v) = C_e$, so the $C_e$ must also satisfy $\sum_{v=e^-}C_e = 0$ for each vertex $v$.  In particular, if $v$ has valency one then $C_e$ must equal zero, so $g_e$ is zero in a neighborhood of $v$ because $f_e$ is so.  Moreover, if $v$ has valency two, then smoothness of $f$ at $v$ readily implies that $g$ is smooth at $v$ once the above sum condition is satisfied at $v$.

  We have now reduced the question to a linear algebra problem.  Choosing an orientation of each edge, we want numbers $C_e$ for each (unoriented) edge satisfying
  \[ \sum_{v=e^-} C_e + \sum_{v=e^+}(\mu_e-C_e) = 0 \]
  or equivalently,
  \[ \sum_{v=e^-} C_e - \sum_{v=e^+} C_e = -\sum_{v=e^+} \mu_e \]
  at each vertex $v$.  Let $\R V(\Sigma)$ and $\R E^+(\Sigma)$ be the free vector spaces on the vertices and unoriented edges of $V$, respectively.  Let $B$ be the oriented incidence matrix of $\Sigma$: this is the matrix of the linear transformation $B\colon\R E^+(\Sigma)\to\R V(\Sigma)$ defined by $Be = (e^+) - (e^-)$.  In terms of this matrix, we want a vector $c = (C_e)\in\R E^+(\Sigma)$ such that $Bc$ is the vector $m = \sum_v\sum_{v=e^+} \mu_e\,(v)$.

  The Laplacian matrix of the graph underlying $\Sigma$ is $BB^T$.  Since $\Sigma$ is connected, it is known that $\dim\ker(BB^T)=1$.  It follows that $\ker(B^T) = \ker(BB^T)$ has dimension $1$, and clearly $\mathbf 1 = \sum_v (v)\in\ker(B^T)$, so $\ker(B^T)$ is spanned by $\mathbf 1$.  The kernel of $B^T$ is the orthogonal complement of the image of $B$ (with respect to the usual dot product), so it suffices to show that $m\cdot\mathbf1 = 0$.  This is true because
  \[ m\cdot\mathbf1 = \sum_v\sum_{v=e^+}\mu_e = \sum_e \mu_e = 0. \]
  This concludes the proof that $h^{1,1}(\Sigma) = 1$.

  Now suppose that $\del\Sigma\neq\emptyset$.  The above linear algebra argument shows that  $\ker\int_\Sigma$ is contained in $\d''\cA^{1,0}(\Sigma,\del\Sigma)$, so it remains to show that there exists $\eta\in\cA^{1,0}(\Sigma,\del\Sigma)$ such that $\int_\Sigma\d''\eta\neq 0$, or by Stokes' theorem, such that $\int_{\del\Sigma}\eta\neq 0$.  Choose a vertex $v_0\in\del\Sigma$ and an edge $e_0$ such that $e_0^-=v_0$, and define $\eta=(f_e\,\d't_e)$ as follows.  For all edges $e\neq e_0$ we let $f_e = 0$, and we define $f_{e_0}$ to be a function on $e_0$ such that $f_{e_0}(v_0)=1$ and $f_{e_0}$ is zero on a neighborhood of $e_0^+$.  Then $\eta\in\cA^{1,0}(\Sigma,\del\Sigma)$ and $\int_{\del\Sigma}\eta = 1$, as required.

\subsubsection{Computation of $h^{1,0}$:}\label{sec:computation-h1-0} The kernel of $\d''\colon\cA^{1,0}(\Sigma,\del\Sigma)\to\cA^{1,1}(\Sigma,\del\Sigma)$ consists of all smooth $(1,0)$-forms that are \emph{constant} on edges.  Choose an orientation of the edges of $\Sigma$.  If $B$ is the oriented incidence matrix of $\Sigma$, as in~\secref{sec:computation-h1-1}, then
  \[ H^{1,0}(\Sigma,\del\Sigma) = \left\{(c_e\,\d't_e)\colon B\sum c_e(e)\in\Span\{(v)\colon v\in\del\Sigma\}\right\}. \]
  We observed above that the image of $B$ is $\mathbf1^\perp =\bigl\{\sum d_v(v)\colon\sum d_v=0\bigr\}$; in particular, the rank of $B$ is $\#V(\Sigma)-1$.  If $\del\Sigma=\emptyset$ then $h^{1,0}(\Sigma) = \dim\ker(B) = \#E^+(\Sigma)-\rank(B) = \#E^+(\Sigma)-\#V(\Sigma)+1 = g$.  Suppose then that $\del\Sigma\neq\emptyset$.  Let $W = \Span\{(v)\colon v\in\del\Sigma\}\subset\R V(\Sigma)$.    Since $W$ is not contained in $\mathbf1^\perp$, we have $W+\mathbf1^\perp = \R V(\Sigma)$, so $\dim(W\cap\mathbf1^\perp) = \#\del\Sigma-1$.  It follows that
  \[ h^{1,0}(\Sigma,\del\Sigma) = \dim(W\cap\mathbf1^\perp) + \dim\ker(B)
    = h^{1,0}(\Sigma) + \#\del\Sigma - 1 = g + \#\del\Sigma - 1, \]
  as desired.

\subsubsection{Computation of $h^{0,1}$:} \label{sec:computation-h0-1}
Let $\gamma$ be a path in $\Sigma$: this is a union of edges $e_1,\ldots,e_n$ such that $e_i^+=e_{i+1}^-$ for all $i=1,\ldots,n-1$.  For $\omega=(f_e\,\d'' t_e)\in\cA^{0,1}(\Sigma,\del\Sigma)$ we define
\begin{equation}\label{eq:integral.along.path}
  \int_\gamma\omega = \sum_{i=1}^n \int_0^{\ell(e_i)}f_{e_i}\circ t_{e_i}(x)\,\d x.
\end{equation}
  Choose a maximal spanning tree $T$ inside $\Sigma$, and choose a base vertex $v_0\in T$.  Then $\Sigma\setminus T$ is the union of $g$ edges $e_1,\ldots,e_g$, and for each $i$ there is a unique simple closed path $\gamma_i\subset T\cup e_i$ starting and ending at $v_0$ and passing through $e_i$.  Consider the homomorphism $I\colon\cA^{0,1}(\Sigma,\del\Sigma)\to\R^g$ defined by
  \[ I(\omega) = \left( \int_{\gamma_1}\omega,\;\ldots,\;\int_{\gamma_g}\omega \right). \]
  First we claim that $I$ is surjective.  For $i\in\{1,\ldots,g\}$ define $\omega_i = (f_e\,\d''t_e)$ as follows.  Set $f_e=0$ for $e\neq e_i$.  Choose a smooth bump function $g_i$ on $\R$ with support in $[0,\ell(e_i)]$ and integral equal to $1$, and set $f_{e_i} = g_i\circ t_{e_i}\inv$.  Then $\omega_i\in\cA^{0,1}(\Sigma,\del\Sigma)$ and $I(\omega_i)$ is the $i$th unit coordinate vector.  This proves surjectivity.

  To prove $h^{0,1}(\Sigma,\del\Sigma)=g$, it suffices to show $\ker(I) = \d''\cA^{0,0}(\Sigma,\del\Sigma)$.  We have the inclusion $\d''\cA^{0,0}(\Sigma,\del\Sigma)\subset\ker(I)$ by the fundamental theorem of calculus and continuity of $f$.  Now let $\omega=(f_e\,\d''t_e)\in\ker(I)$, and define $g\colon\Sigma\to\R$ as follows.  For $v\in V(\Sigma)$ choose any path $\gamma$ from $v_0$ to $v$, and set $g(v) = \int_\gamma\omega$.  This is well-defined because $\{\gamma_1,\ldots,\gamma_n\}$ generates $\pi_1(\Sigma,v_0)$.  We define $g$ on an edge $e$ by $g\circ t_e(x)=g(e^-) + \int_0^x f_e\circ t_e(t)\,\d t$.  Then $g\in\cA^{0,0}(\Sigma,\del\Sigma)$ and $\d''g = \omega$, as desired.\qed

\subsection{Poincar\'e Duality}\label{sec:poincare-duality-graphs}
In this subsection we observe that the Dolbeault cohomology of a metric graph without boundary satisfies Poincar\'e duality.  The wedge product followed by integration defines bilinear pairings
\begin{equation}\label{eq:poincare.pairing.0}
 \cA^{0,0}(\Sigma)\times\cA^{1,1}(\Sigma)\To\R, \qquad
  \cA^{1,0}(\Sigma)\times\cA^{0,1}(\Sigma)\To\R.
\end{equation}
By Stokes' theorem, these descend to pairings
\begin{equation}\label{eq:poincare.pairing}
  H^{0,0}(\Sigma)\times H^{1,1}(\Sigma)\To\R, \qquad
  H^{1,0}(\Sigma)\times H^{0,1}(\Sigma)\To\R.
\end{equation}
Poincar\'e duality is the assertion that these pairings are perfect.

\begin{prop}\label{prop:poincare.duality.graphs}
  Let $\Sigma$ be a weighted metric graph such that $\del\Sigma=\emptyset$, and suppose that $\Sigma$ has no isolated vertices.  Then the pairings~\eqref{eq:poincare.pairing} are perfect.
\end{prop}

\begin{proof}
  Since the cohomology of a disjoint union is a direct sum, we may assume that $\Sigma$ is connected.  By Lemmas~\ref{lem:unweighting.forms} and~\ref{lem:integral.unweighting}, we may assume that $\Sigma$ has trivial edge weights.

  Since $h^{0,0}(\Sigma)=h^{1,1}(\Sigma)=1$, we only need to show that the pairing $H^{0,0}(\Sigma)\times H^{1,1}(\Sigma)\to\R$ is nonzero.  This is clear by integrating the constant function against a bump function on an edge.

  For the pairing $H^{1,0}(\Sigma)\times H^{0,1}(\Sigma)\to\R$, we have $h^{1,0}(\Sigma)=h^{0,1}(\Sigma)=g$ by Proposition~\ref{prop:dolbeault.graphs}, so it is enough to show nondegeneracy on the right.  We use the notation in~\secref{sec:computation-h0-1}, where we constructed forms $\omega_1,\ldots,\omega_g\in\cA^{0,1}(\Sigma)$ whose classes modulo $\d''\cA^{0,0}(\Sigma)$ are a basis of $H^{0,1}(\Sigma)$.  Let $\omega = a_1\omega_1+\cdots+a_g\omega_g\in\cA^{0,1}(\Sigma)$ represent an arbitrary element of $H^{0,1}(\Sigma)$, and assume that $\int_\Sigma\eta\wedge\omega = 0$ for all $\eta\in H^{1,0}(\Sigma)$.  Suppose that the path $\gamma_1$ of~\secref{sec:computation-h0-1} traverses the edges $f_1,f_2,\ldots,f_n$, with $f_i^+=f_{i+1}^-$ for $i < n$ and $f_n^+=f_1^-$.  Define $\eta = (c_e\,\d't_e)\in H^{1,0}(\Sigma)$ to be the cycle associated to $\gamma_1$: that is,
  \[ c_e =
    \begin{cases}
      1 & e\in\{f_1,f_2,\ldots,f_n\} \\
      -1 & \bar e\in\{f_1,f_2,\ldots,f_n\} \\
      0 & \text{otherwise.}
    \end{cases} \]
  Then $\int_\Sigma\eta\wedge\omega = a_1$, since the support of $\omega_i$ does not intersect the path $\gamma_1$ for $i > 1$.  It follows that $a_1=0$, and a similar argument shows $a_i=0$ for all $i$, i.e., $\omega=0$.
\end{proof}


\section{Harmonicity}\label{Sec:harmonicity}

In this section we define harmonic functions on graphs and harmonic maps between graphs, and we will see (Remark~\ref{rem:harmonic functions vs morphisms}) that the former is a special case of the latter.  This is the correct context in which to discuss functoriality of smooth forms on graphs.  We then present the classes of harmonic morphisms we will encounter in the sequel, namely subgraphs, modifications, and quotients.  We finish by computing the integral of the pullback of a form.

\subsection{Harmonic Functions} \label{sec:harmonic functions}
We introduce here harmonic functions.

\begin{defn}\label{def:harmonic.func.graph}
  A continuous function $f\colon\Sigma\to\R$ is \defi{linear} on an edge $e$ if it has the form $f\circ t_e(x) = ax + b$ for $a,b\in\R$; in this case, its \emph{slope} along the (oriented) edge $e$ is the number $a$, and is denoted $\d f/\d t_e$.

  A \defi{harmonic function} on $(\Sigma,\del\Sigma)$ is a smooth function that is linear on the edges.  For a subring $R\subset\R$, we say that a function $f$ is \defi{$R$-harmonic} provided that the slopes of $f$ are contained in $R$; for an  $R$-submodule $\Gamma\subset\R$, we say that a function $f$ is \defi{$(R,\Gamma)$-harmonic} if it is $R$-harmonic and takes values in $\Gamma$ on the vertices of $\Sigma$.
\end{defn}

\begin{rem*}\

  \begin{enumerate}
  \item The notion of harmonicity depends on the boundary $\del\Sigma$.
  \item A harmonic function is constant on an edge adjacent to an interior leaf.
  \item A function $h\colon\Sigma\to\R$ is harmonic if and only if $\nu^*h\colon\Sigma_0\to\R$ is harmonic: see Lemma~\ref{lem:unweighting.forms}.
  \item Our notion of harmonic function coincides with Thuillier's~\cite[1.2.1]{thuillier05:thesis}; however, our $\del\Sigma$ is denoted $\Gamma$ in \textit{loc.\ cit.}, whereas Thuillier's $\del S$ is just the set of leaves.
  \end{enumerate}
\end{rem*}

\begin{rem}
  A smooth function $f\in\cA^{0,0}(\Sigma,\del\Sigma)$ is harmonic if and only if $\d'\d''f = 0$ (equivalently, $\d''\d'f = 0$).  Compare~\otherfile{Proposition~\ref*{I-weakly smooth harmonic functions}}.
\end{rem}

\subsection{Harmonic Maps}\label{sec:harmonic-morphisms}
Now we discuss harmonic maps of graphs and pullbacks of forms.  The harmonicity condition is necessary for smooth forms to pull back to smooth forms.  See~\cite[Section~2]{abbr14:lifting_harmonic_morphism_I} for a discussion of harmonic maps in the context of unweighted metric graphs without boundary.

\begin{defn}\label{def:morphism}
  Let $(\Sigma',\del\Sigma'),(\Sigma,\del\Sigma)$ be weighted metric graphs with boundary.  A \defi{piecewise linear map} from $\Sigma'$ to $\Sigma$ is a continuous map  $\phi\colon\Sigma'\to\Sigma$ of underlying topological spaces such that, after potentially subdividing $\Sigma$ and $\Sigma'$:
  \begin{enumerate}
  \item The vertices of $\Sigma'$ map to vertices of $\Sigma$.
  \item The edges of $\Sigma'$ map to edges (homeomorphically) or to vertices of $\Sigma$.
  \item If $\phi(v')\in\del\Sigma$ and $\phi$ is not constant in any neighborhood of $v'$, then $v'\in\del\Sigma'$.
  \item $\phi$ is linear on edges, in the sense that if $e'$ is an edge of $\Sigma'$ and $e$ is an edge of $\Sigma$ with $\phi(e') = e$ (respecting orientations), then $\phi\circ t_{e'}(x) = t_{e}(\ell(e)/\ell(e')\,x)$.
  \end{enumerate}
  The number $\ell(e)/\ell(e')$ is called the \defi{expansion factor} of $\phi$ along $e'$ and is denoted $d_{e'}(\phi)$.  We take $d_{e'}(\phi)=0$ if $e'$ maps to a vertex.  For a subring $R\subset\R$, if all the expansion factors are contained in $R$, then $\varphi$ is called \emph{piecewise $R$-linear}.
\end{defn}

Note that a piecewise linear map $\phi\colon\Sigma'\to\Sigma$ induces a piecewise linear map on the unweightings $\phi_0\colon\Sigma'_0\to\Sigma_0$, with $d_{e'}(\phi_0) = \frac{w(e')}{w(e)}d_{e'}(\phi)$.

\begin{defn}\label{def:harmonic.morphism}
  Let $\phi\colon\Sigma'\to\Sigma$ be a piecewise linear map of weighted metric graphs with boundary, let $v'\in\Sigma'$ be a point, and let $v = \phi(v')$.  Choose subdivisions of $\Sigma'$ and $\Sigma$ such that $v'$ is a vertex.  We say that $\phi$ is \defi{harmonic at $v'$} provided that, for every edge $e$ beginning at~$v$, the number
  \begin{equation}\label{eq:dvprimephi}
   d_{v'}(\phi) \coloneq \sum_{\substack{(e')^- = v'\\e'\mapsto e}}
    \frac{w(e')}{w(e)} d_{e'}(\phi)
    = \frac{\ell(e)}{w(e)}\sum_{\substack{(e')^- = v'\\e'\mapsto e}}\frac{w(e')}{\ell(e')}
    = \ell_0(e)\sum_{\substack{(e')^- = v'\\e'\mapsto e}}\frac 1{\ell_0(e')}
  \end{equation}
  is independent of the choice of $e$.  We set $d_{v'}(\phi) = 0$ if every edge adjacent to $v'$ is contracted (e.g.\ if $v'$ is isolated).  If $\phi$ is harmonic at $v'$ then we call $d_{v'}(\phi)$ the \defi{local degree of $\phi$ at $v'$}.

  We say that $\phi$ is \defi{harmonic} provided that it is harmonic at every \emph{interior} vertex $v'\in\Sigma'\setminus\del\Sigma'$ (equivalently, every interior point $v'\in\Sigma'\setminus\del\Sigma'$).  More generally, for a subring $R\subset\R$, an \defi{$R$-harmonic map} is a harmonic piecewise $R$-linear map.
\end{defn}

\begin{rem}\ \label{rem:remark for harmonic maps}

  \begin{enumerate}
  \item The property of being harmonic is insensitive to subdivision of the source and target.
  \item The numbers $d_{v'}(\phi)$ need not be integers.
  \item Harmonicity only depends on the quantities $\ell_0(e') = \ell(e')/w(e')$ and $\ell_0(e) = \ell(e)/w(e)$, hence can be checked on unweightings: that is, a piecewise linear map $\phi\colon\Sigma'\to\Sigma$ is harmonic if and only if the associated map $\phi_0\colon\Sigma_0'\to\Sigma_0$ is harmonic, and $d_{v'}(\phi) = d_{v'}(\phi_0)$ for all $v'\in\Sigma'\setminus\del\Sigma'$.
  \item As a special case of~(3), let $\nu:\Sigma_0 \to \Sigma$ be the canonical piecewise linear map that is the identity on underlying graphs.  Then $\nu$ is harmonic and $d_{v'}(\nu) = 1$ for all $v'\in\Sigma_0$.
  \item A constant map and the identity map are $\Z$-harmonic.
  \end{enumerate}
\end{rem}

\begin{defn}\label{def:harmonic.degree}
  Let $\phi\colon\Sigma'\to\Sigma$ be a piecewise linear map and let $e\subset\Sigma$ be an edge.  The \defi{degree of $\phi$ above $e$} is the number
  \begin{equation}\label{eq:degree.edge.sum}
    d_e(\phi) = \sum_{e'\mapsto e}\frac{w(e')}{w(e)}d_{e'}(\phi) = \ell_0(e)\sum_{e'\mapsto e}\frac 1{\ell_0(e')}.
  \end{equation}
  If $E(\Sigma)\neq\emptyset$ and $e\mapsto d_e(\phi)$ is constant, we call $d(\phi) = d_e(\phi)$ the \defi{degree} of $\phi$.
\end{defn}

\begin{rem}\  \label{rem:remark for degree of harmonic maps}

  \begin{enumerate}
  \item The unweighting $\nu:\Sigma_0 \to \Sigma$ has degree $1$ if $\Sigma$ has an edge.
  \item If $\phi$ is harmonic and $v\notin\phi(\del\Sigma')$ then
    \[ d_e(\phi) = \sum_{v'\mapsto v}d_{v'}(\phi) \]
    for any edge $e$ adjacent to $v$.
  \item It follows from~(2) that if $\phi$ is harmonic then $e\mapsto d_e(\phi)$ is locally constant on $\Sigma\setminus\phi(\del\Sigma')$.  In particular, if $\Sigma$ is connected and has an edge and if $\del\Sigma'=\emptyset$, then $\phi$ has a degree; if $d(\phi) > 0$ then $\phi$ is surjective.  See~\cite[Section~2]{abbr14:lifting_harmonic_morphism_I}.
  \item If $\Sigma$ has an edge and $\phi\colon\Sigma'\to\Sigma$ is constant then $d(\phi)=0$.
  \end{enumerate}
\end{rem}

\begin{rem} \label{rem:harmonic functions vs morphisms}
  If we endow the real line with the structure of an infinite metric graph with trivial weights and empty boundary, then a piecewise linear map $\varphi:\Sigma \to \R$ is harmonic if and only if $\varphi$ is a harmonic function in the sense of Definition~\ref{def:harmonic.func.graph}.
\end{rem}

\begin{prop} \label{prop:characterization of harmonic morphisms}
  Let $\phi\colon\Sigma'\to\Sigma$ be a piecewise linear map of weighted metric graphs with boundary.  Then $\phi$ is harmonic at an interior vertex $v'\in\Sigma'$ if and only if $\phi$ locally pulls back harmonic functions at $v = \phi(v')$ to harmonic functions at $v'$.
\end{prop}

The proof is left to the reader. See~\cite[Proposition~2.10]{abbr14:lifting_harmonic_morphism_I} and~\cite[Proposition~2.6]{baker_norine09:harmonic_morphism_hyperell_graphs} for similar statements.

\begin{rem}\label{rem:harmonic.isomorphism}
  Suppose that $\phi\colon\Sigma'\to\Sigma$ is a \defi{harmonic isomorphism}, i.e., that $\phi$ is a harmonic map that admits a harmonic inverse $\psi\colon\Sigma\to\Sigma'$.  Then $\phi$ induces an isomorphism of underlying graphs with respect to suitable subdivisions of $\Sigma$ and $\Sigma'$.  If $v'\in\Sigma'$ is a non-isolated vertex, then $\phi$ is not constant in any neighborhood of $v'$, so $\phi(v')\in\del\Sigma$ if and only if $v'\in\del\Sigma'$.  Since $e\mapsto d_e(\phi)$ is locally constant on $\Sigma\setminus\phi(\del\Sigma')$ by Remark~\ref{rem:remark for degree of harmonic maps}(3), we can describe $\phi$ as follows: for each connected component $C$ of $\Sigma\setminus\del\Sigma$, there is a positive real number $\alpha(C)$ such that $\ell_0(e)=d_e(\phi)\ell_0(e')=\alpha(C) \ell_0(e')$ for each edge $e$ of $C$ and $e' \coloneq \psi(e)$.  In particular, if $\Sigma$ and $\Sigma'$ have trivial edge weights, then $\phi$ scales the length of each edge of $C$ by $\alpha(C)$.

  A harmonic isomorphism need not take isolated interior vertices to interior vertices: for instance, if $x$ is an interior isolated vertex of $\Sigma$, then the identity map is a harmonic isomorphism $(\Sigma,\del\Sigma)\to(\Sigma,\del\Sigma\cup\{x\})$.

  Conversely, let $\varphi:\Sigma' \to \Sigma$ be a bijective piecewise linear map  of weighted metric graphs with boundary which maps the set of non-isolated boundary points of $\Sigma'$ onto the set of non-isolated boundary points of $\Sigma$.  Then $\phi$ is a harmonic isomorphism if for each connected component $C$ of $\Sigma\setminus\del\Sigma$, there is a positive real number $\alpha(C)$ such that $\ell_0(e)=\alpha(C) \ell_0(e')$ for each edge $e'$ of $\phi\inv(C)$ and $e \coloneq \phi(e')$.
\end{rem}

Harmonicity can be checked after pulling back by a surjective harmonic morphism:

\begin{lem} \label{lem:converse.to.harmonic.composition}
  Let $\Sigma,\Sigma''$ be weighted metric graphs with boundary, let $\phi\colon\Sigma\to\Sigma''$ be any continuous map, and let  $\psi\colon\Sigma'\to\Sigma$ be a surjective harmonic map with $\psi(\del \Sigma') \subset \del \Sigma$.  If $\phi \circ \psi$ is harmonic, then $\phi$ is harmonic.
\end{lem}

\begin{proof}
  First we show that $\phi$ is piecewise linear.  Let $e$ be any edge of $\Sigma$. Since $\psi$ is surjective, we may assume after subdivision that there is an edge $e'$ of $\Sigma'$ mapping by isomorphically onto $e$. Again after subdivision, this edge is either mapped by $\phi \circ \psi$ homeomorphically onto an edge $e''$ of $\Sigma''$ or is crushed to a vertex of $\Sigma''$.  In the first case, it follows that $\varphi$ maps $e$ homeomorphically onto $e''$, and in the second, $e$ is crushed to a vertex of $\Sigma''$.  If $e$ maps homeomorphically onto $e''$, then $\phi$ is linear on $e$ with expansion factor $d_{e'}(\phi\circ\psi)/d_{e'}(\psi)$.

  It remains to prove condition~(3) of Definition~\ref{def:morphism}. Let $v\in\Sigma$ with $\phi(v) \in \del \Sigma''$. We assume that $\varphi$ is not constant in any neighborhood of $v$. By surjectivity of $\psi$, there is $v' \in \Sigma'$ such that $\phi \circ \psi$ is not constant in any neighborhood of $v'$. Since $\phi \circ \psi$ is a piecewise linear map of weighted metric graphs with boundary, we deduce from $\phi(v)=\phi \circ \psi(v')\in \del \Sigma$ that $v' \in \del \Sigma'$. Using $\psi(\del \Sigma') \subset \del \Sigma$, we get $v=\psi(v')\in \del \Sigma$.

  Now we proceed to harmonicity: we need to check the independence of local degrees for $\phi \circ \psi$ as required in Definition~\ref{def:harmonic.morphism}.  By passing to unweightings, we may assume all edge weights are trivial.  Let $v$ be a vertex of $\Sigma \setminus \del \Sigma$, let $v'' \coloneq \phi(v)$ be the corresponding vertex of $\Sigma''$, and pick any edge $e''$ of $\Sigma''$ with $e''^- = v''$. We may assume that $v$ is not isolated. Then by our assumptions, we can find a vertex $v' \in \Sigma' \setminus \del \Sigma'$ with $\psi(v')=v$ such that $\psi$ is not constant in any neighborhood of $v'$.
  We compute:
  \[\begin{split}
      d_{v'}(\phi \circ \psi) &= \sum_{\substack{e'^- = v'\\e'\mapsto e''}}d_{e'}(\phi \circ \psi)
      = \sum_{\substack{e^- = v\\e\mapsto e''}}  \sum_{\substack{e'^- = v'\\e'\mapsto e}}d_{e}(\phi)d_{e'}(\psi) \\
      &= \sum_{\substack{e^- = v\\e\mapsto e''}} d_{e}(\phi)
      \sum_{\substack{e'^- = v'\\e'\mapsto e}}d_{e'}(\psi)
      = d_{v'}(\psi)\sum_{\substack{e^- = v\\e\mapsto e''}} d_{e}(\phi) = d_{v'}(\psi) d_{v}(\phi).
    \end{split}\]
  Since $\psi$ is not constant in any neighborhood of $v'$, we have $d_{v'}(\psi)\neq 0$, and hence $d_{v}(\varphi)$ is independent of the choice of the edge $e''$.
\end{proof}

\begin{lem}\label{lem:pullback.harmonic}
  Let $\phi\colon\Sigma'\to\Sigma$ be a harmonic map of weighted metric graphs with boundary.  Define
  \begin{align*}
    \phi^*f &= f\circ\phi  &  f&\in\cA^{0,0}(\Sigma,\del\Sigma) \\
    \phi^*(f_e\,\d't_e) &= (f_{\phi(e')}\circ\phi\,d_{e'}(\phi)\,\d't_{e'}) &  (f_e\,\d't_e)&\in\cA^{1,0}(\Sigma,\del\Sigma) \\
    \phi^*(f_e\,\d''t_e) &= (f_{\phi(e')}\circ\phi\,d_{e'}(\phi)\,\d''t_{e'}) &  (f_e\,\d''t_e)&\in\cA^{0,1}(\Sigma,\del\Sigma) \\
    \phi^*(f_e\,\d't_e\d''t_e) &= (f_{\phi(e')}\circ\phi\,d_{e'}(\phi)^2\,\d't_{e'}\d''t_{e'}) &  (f_e\,\d't_e\d''t_e)&\in\cA^{1,1}(\Sigma,\del\Sigma).
  \end{align*}
  (If $\phi(e')$ is a vertex of $\Sigma$ then we take $\phi\circ f_{\phi(e')}\,d_{e'}(\phi)$ to be zero.)  Then $\phi^*$ is a homomorphism of bigraded differential algebras $\cA^{\bullet,\bullet}(\Sigma,\del\Sigma)\to\cA^{\bullet,\bullet}(\Sigma',\del\Sigma')$.
\end{lem}

\begin{proof}
  Let $\nu\colon\Sigma_0\to\Sigma$ and $\nu'\colon\Sigma_0'\to\Sigma'$ be the unweightings, and let $\phi_0\colon\Sigma_0'\to\Sigma_0$ be the associated map.  We have defined isomorphisms $\nu^*\colon\cA^{p,q}(\Sigma,\del\Sigma)\isom\cA^{p,q}(\Sigma_0,\del\Sigma_0)$ and $\nu'^*\colon\cA^{p,q}(\Sigma',\del\Sigma')\isom\cA^{p,q}(\Sigma_0',\del\Sigma_0')$ in Lemma~\ref{lem:unweighting.forms}.
It is immediate from the definitions that $\nu'^*\phi^* = \phi_0^*\nu^*$, so we may pass to unweightings to assume all edge weights are trivial.

  First we need to show that $\phi^*\cA^{p,q}(\Sigma,\del\Sigma)\subset\cA^{p,q}(\Sigma',\del'\Sigma')$, i.e., that the pullback of a smooth form is smooth.  We treat only the case of smooth functions, as the other cases are similar.  Let $f\in\cA^{0,0}(\Sigma,\del\Sigma)$ and let $g = f\circ\phi$.   We verify the conditions of Definition~\ref{def:smooth.on.graph}.  Clearly $g$ is continuous and smooth on edges.  Let $v'\in\Sigma'\setminus\del\Sigma'$, and let $v = \phi(v')$.  If $\phi$ is constant in a neighborhood of $v'$ then evidently $g$ is smooth in a neighborhood of $v'$; assume then that $\phi$ is not constant in any neighborhood of $v'$, so $v\notin\del\Sigma$ by Definition~\ref{def:morphism}.

  Condition~(1): Suppose that $v'$ has valency $1$.  If the edge adjacent to $v'$ maps to a vertex of $\Sigma$ then clearly $g$ is constant in a neighborhood of $v'$.  Otherwise, $v$ has valency one by harmonicity, so $f$ is constant in a neighborhood of $v$, and hence $g$ is constant in a neighborhood of $v'$.

  Condition~(2): Suppose that $v'$ has valency $2$, with edges $e_1',e_2'$ starting at $v'$.  There are several sub-cases.
  \begin{enumerate}\renewcommand{\labelenumi}{{(\alph{enumi})}}
    \item If $e_1'$ and $e_2'$ both map to vertices, then $\phi$ is constant in a neighborhood of $v'$.
    \item If $e_1'$ maps to a vertex $v$ and $e_2'$ does not, then $v$ has valency one by harmonicity, so $f$ is constant in a neighborhood of $v$, hence $g$ is constant in a neighborhood of $v'$, so $(g|_{e_1'},g|_{e_2'})$ is smooth at $v'$.
    \item If $\phi(e_1')$ and $\phi(e_2')$ are equal to the same edge of $\Sigma$, then again $v$ has valency one by harmonicity, and we proceed as above.
    \item Suppose then that $e_1 = \phi(e_1')$ and $e_2 = \phi(e_2')$ are distinct edges of $\Sigma$.  Then $v$ has valency $2$ by harmonicity, and
      \[ g\circ t_{e_i'}(x) = f\circ\phi\circ t_{e_i'}(x)
        = f\circ t_{e_i}(d_{e_i'}(\phi) x) \]
      for $i=1,2$ and $x\in[0,\ell(e_i')]$.  Hence
      \[ \frac{\d^n g}{\d t_{e_i'}^n}(v') = d_{e_i'}(\phi)^n\frac{\d^n f}{\d t_{e_i}^n}(v) \]
      for $n\geq 0$.  It follows from smoothness of $(f|_{e_1},f|_{e_2})$ at $v$ that
      \[ \frac{1}{d_{e_1'}(\phi)^n}\frac{\d^n g}{\d t_{e_1'}^n}(v) = (-1)^n\frac{1}{d_{e_2'}(\phi)^n}\frac{\d^n g}{\d t_{e_2'}^n}(v). \]
      The harmonicity condition says that $d_{e_1'}(\phi)=d_{e_2'}(\phi)$; canceling this factor from both sides proves smoothness of $(g|_{e_1'},g|_{e_2'})$ at $v'$.
    \end{enumerate}

  Condition~(3): Suppose that $v'$ has valency greater than $2$.  Whatever the valency of $v$, we have $\sum_{v=e^-}\d f/\d t_e(v) = 0$.  If $v$ is an isolated vertex then $\phi$ is constant in a neighborhood of $v'$.  Otherwise, choose an edge $e$ starting at $v$.  If $e'$ is an edge starting at $v'$ mapping to $e$ then $\d g/\d t_{e'}(v') = d_{e'}(\phi)\,\d f/\d t_e(v)$ as above, so
    \[ \sum_{\substack{(e')^- = v'\\e'\mapsto e}} \frac{\d g}{\d t_{e'}}(v')
      = \frac{\d f}{\d t_e}(v)\sum_{\substack{(e')^- = v'\\e'\mapsto e}} d_{e'}(\phi)
      = \frac{\d f}{\d t_e}(v)d_{v'}(\phi).
    \]
    Summing over all edges starting at $v$, and noting that $\d g/\d t_{e'}(v')$ is zero if $e'$ is crushed to the vertex $v$, we have
    \[ \sum_{v'=(e')^-}\frac{\d g}{\d t_{e'}}(v)
      = d_{v'}(\phi)\sum_{v=e^-}\frac{\d f}{\d t_e}(v)
      = 0. \]
  This finishes the proof that $g$ is smooth.

  Compatibility of $\phi^*$ with $\d'$ and $\d''$ is easily verified and is left to the reader.
\end{proof}

\begin{rem} \label{rem:functoriality of pull-back}
  If $\phi\colon\Sigma'\to\Sigma$ and $\psi\colon\Sigma''\to\Sigma'$ are harmonic maps, then $\phi\circ\psi$ is harmonic, and $(\phi\circ\psi)^*=\psi^*\circ\phi^*$.  This follows from Proposition~\ref{prop:characterization of harmonic morphisms}.  Moreover, one computes easily that for $v'' \in \Sigma'' \setminus \del \Sigma''$ and $v' \coloneq \psi(v'')$, we have $d_{v''}(\phi \circ \psi)=d_{v'}(\phi)d_{v''}(\psi)$.
\end{rem}

\begin{rem} \label{rem:unweighting.and.harmonic}
Let $\nu\colon\Sigma_0\to\Sigma$ be the unweighting of $\Sigma$ as in Definition~\ref{def:weighted.graph}.  The map $\nu$ is harmonic (of degree $1$ if $\Sigma$ has an edge), and $\nu^*$ agrees with the isomorphism in Lemma~\ref{lem:unweighting.forms}.
\end{rem}

\subsection{Subgraphs}\label{sec:subgraphs}
The most basic example of a harmonic map is the inclusion of a subgraph.

\begin{defn}\label{def:subgraph}
  Let $\Sigma$ be a weighted metric graph with boundary.  A subset $\Sigma'\subset\Sigma$ is a \defi{subgraph} if it is a union of vertices and closed edges of a subdivision of $\Sigma$.  A subgraph inherits the structure of a weighted metric graph; its boundary $\del\Sigma'$ is defined to be the union of $\Sigma'\cap\del\Sigma$ with the relative boundary of $\Sigma'$ in $\Sigma$.
\end{defn}

The inclusion of a subgraph $\iota\colon\Sigma'\inject\Sigma$ is $\Z$-harmonic: if $v'\in\Sigma'$ is a vertex, then either $\iota$ is an isomorphism in a neighborhood of $v'$, or $v'\in\del\Sigma'$.  The  homomorphisms
\[\iota^*\colon\cA^{p,q}(\Sigma,\del\Sigma) \To \cA^{p,q}(\Sigma',\del\Sigma') \]
are called \defi{restrictions}.

\subsection{Modifications}\label{sec:modifications}
Our next important example of a harmonic map arises by attaching trees to a graph.

\begin{defn}\label{def:modification}
  Let $\Sigma$ be a weighted metric graph.  A \defi{modification} of $\Sigma$ is a weighted metric graph $\Sigma'$ obtained from $\Sigma$ by attaching finitely many (weighted metric) trees to points of $\Sigma$.  If $\Sigma$ has boundary $\del\Sigma$, then we consider $\Sigma'$ as a weighted metric graph with boundary by setting $\del\Sigma' = \del\Sigma$.

  If $\Sigma'$ is a modification of $\Sigma$, then there are natural inclusion and retraction maps $\Sigma\inject\Sigma'$ and $\Sigma'\surject\Sigma$, respectively.  The retraction map is $\Z$-harmonic (of degree $1$ if $\Sigma$ has an edge), while inclusion is not harmonic.
\end{defn}

\begin{lem}\label{lem:modif.isom.cohom}
  Let $\Sigma$ be a weighted metric graph with boundary and let $\Sigma'$ be a modification of $\Sigma$.  Suppose that $\Sigma$ has no isolated interior vertices.  Then the retraction map $\Sigma'\surject\Sigma$ induces isomorphisms $H^{p,q}(\Sigma,\del\Sigma)\isom H^{p,q}(\Sigma',\del\Sigma')$ for all $p,q\in\{0,1\}$.
\end{lem}

\begin{proof}
  As usual we may replace $\Sigma$ and $\Sigma'$ by their unweightings to assume all edge weights are trivial.  Since the cohomology of a disjoint union is a direct sum, we may assume $\Sigma$ to be connected.  Let $\tau\colon\Sigma'\surject\Sigma$ denote the retraction.  If $\Sigma$ consists of a single boundary vertex $v$ then $\Sigma'$ is a tree with $\del\Sigma' = \{v\}$, so the only nonzero cohomology group of $\Sigma'$ is $H^{0,0}$, and hence $\tau^*$ is an isomorphism on all $H^{p,q}$.  Assume then that $\Sigma$ has an edge.  It is clear that $\tau^*$ is an isomorphism on $H^{0,0}$ and $H^{1,1}$ by~\secref{sec:computation-h0-0} and~\secref{sec:computation-h1-1}, respectively.

In~\secref{sec:computation-h0-1} we constructed forms $\omega_1,\ldots,\omega_g\in\cA^{0,1}(\Sigma,\del\Sigma)$ whose cohomology classes form a basis of $H^{0,1}(\Sigma,\del\Sigma)$.  The pullbacks $\tau^*\omega_1,\ldots,\tau^*\omega_g$ are simply the forms in $\cA^{0,1}(\Sigma',\del\Sigma')$ obtained by the same construction on $\Sigma'$, so $\tau^*$ is an isomorphism on $H^{0,1}$.

  The kernel of $\d''$ on $\cA^{1,0}$ consists of all forms that are constant on edges. Any such form on $\Sigma'$ is zero on the edges in $\Sigma'\setminus\Sigma$, which is a disjoint union of boundaryless trees.  Any such form is the image under $\tau^*$ of a form on $\Sigma$ which is constant on edges, so $\tau^*$ is an isomorphism on $H^{1,0}$.
\end{proof}

\subsection{Quotients}\label{sec:quotients}
Our next example of a harmonic map will arise from a Galois extension of the ground field of a curve: see Proposition~\ref{prop:skeleton.basechange}.  Here we treat the more general case of a quotient by a finite group.

\begin{defn}\label{def:quotient.graph}
  Let $\pi\colon\Sigma'\to\Sigma$ be a surjective harmonic map of weighted metric graphs with boundary.  Let $G$ be a finite group acting on $\Sigma'$ by harmonic maps such that $\pi\circ\sigma = \pi$ for all $\sigma\in G$.  We say that $\Sigma$ is the \defi{quotient} of $\Sigma'$ by $G$, and we write $\Sigma = \Sigma'/G$, provided that:
  \begin{enumerate}
  \item \label{item:quotient1}
  $G$ acts transitively on the fibers of $\pi$, and
  \item \label{item:quotient2}
  $\pi\inv(\del\Sigma) = \del\Sigma'$.
  \end{enumerate}
\end{defn}

\begin{rem}\label{rem:quotient.graph.props}\

  \begin{enumerate}
  \item A quotient map has finite fibers, so no edge of $\Sigma'$ is
    crushed to a vertex of~$\Sigma$.
  \item Saying that $G$ acts transitively on fibers means $\Sigma$ is the quotient of $\Sigma'$ by $G$ in the category of sets.
  We will see in Lemma~\ref{lem:quotient.in.top.cat} below that it is also the quotient in the category of topological spaces.
  \item The map $\pi\colon\Sigma'\to\Sigma$ makes $\Sigma$ into the quotient of $\Sigma'$ by $G$ if and only if the associated map $\pi_0\colon\Sigma_0'\to\Sigma_0$ makes the unweighting $\Sigma_0$ into the quotient of the unweighting $\Sigma_0'$ by $G$.
 \end{enumerate}
\end{rem}

\begin{rem}\label{rem:subdivision and G-action}
Let $G$ be a finite group acting on a weighted metric graph $\Sigma'$ with boundary by piecewise linear maps.  There is a subdivision of $\Sigma'$  such that $G$ acts on the vertices and edges, and such that for every edge of $\Sigma'$, the tail is not in the same $G$-orbit as the head. In particular, we have $\sigma(e') \neq \bar e{}'$ for every edge $e'$ and every $\sigma \in G$. Then the quotient $\pi\colon\Sigma'\to\Sigma$ of $\Sigma'$ by $G$ exists in the category of simplicial complexes, and gives a graph $\Sigma$ with boundary $\partial \Sigma \coloneqq \pi(\partial \Sigma')$.
\end{rem}

\begin{rem} \label{lem:quotient.in.top.cat}
  Let $\pi\colon\Sigma'\to\Sigma$ be a harmonic map of weighted metric graphs with boundary, and suppose that $\Sigma$ is the quotient of $\Sigma'$ by a finite group $G$.   Choose subdivisions of $\Sigma'$ and $\Sigma$ such that every edge $e'$ of $\Sigma'$ is mapped linearly onto an edge of $\Sigma$.  This subdivision of $\Sigma'$ satisfies the hypotheses of Remark~\ref{rem:subdivision and G-action}, so $\Sigma$ is the quotient simplicial complex of Remark~\ref{rem:subdivision and G-action}.  In particular, $\pi$ is the quotient map in the category of topological spaces, and it is open.
\end{rem}

The following lemma shows that $\Sigma'/G$ is the quotient in the category of weighted metric graphs with boundary, with harmonic maps as morphisms.

\begin{lem}\label{lem:categorical.graph.quotient}
  Let $\pi\colon\Sigma'\to\Sigma$ be a harmonic map of weighted metric graphs with boundary, and suppose that $\Sigma$ is the quotient of $\Sigma'$ by a finite group $G$.  If $\phi'\colon\Sigma'\to\Sigma''$ is a harmonic map of weighted metric graphs with boundary such that $\phi'\circ\sigma = \phi'$ for all $\sigma\in G$, then there exists a unique harmonic map $\phi\colon\Sigma\to\Sigma''$ such that $\phi\circ\pi = \phi'$.
\end{lem}

\begin{proof}
  Since $\Sigma$ is the quotient of $\Sigma'$ in the category of topological spaces, there exists a unique continuous map $\phi\colon\Sigma\to\Sigma''$ such that $\phi\circ\pi=\phi'$.  This map is harmonic by Lemma~\ref{lem:converse.to.harmonic.composition}.
\end{proof}

Next we show that quotients by finite groups always exist, and are unique up to harmonic isomorphism (Remark~\ref{rem:harmonic.isomorphism}).

\begin{prop}\label{prop:graph.quotients.exist}
  Let $G$ be a finite group acting on a weighted metric graph with boundary $(\Sigma',\del\Sigma')$ by harmonic maps preserving $\del\Sigma'$.  Then there exists a weighted metric graph with boundary $(\Sigma,\del\Sigma)$ and a surjective harmonic map $\pi\colon\Sigma'\to\Sigma$ such that $\Sigma$ is the quotient of $\Sigma'$ by $G$ in the sense of Definition~\ref{def:quotient.graph}.  Moreover, the quotient $\Sigma$ is unique up to harmonic isomorphism.
\end{prop}

\begin{proof}
  We replace $\Sigma'$ by its unweighting to assume all edge weights are trivial.  We subdivide $\Sigma'$ as in Remark~\ref{rem:subdivision and G-action}.  Let $\Sigma$ be the quotient simplicial complex $\Sigma'/G$: this is a graph with vertex set $V(\Sigma) = V(\Sigma')/G$ and edge set $E(\Sigma) = E(\Sigma')/G$ endowed with the obvious incidence relation.  Note that $\Sigma$ has no loop edges.  Let $\pi\colon\Sigma'\to\Sigma$ be the quotient map.  We set $\del\Sigma = \pi(\del\Sigma')$.  For an edge $e\in E(\Sigma)$ we set $w(e) = 1$, and we define $\ell(e)$ by
  \[ \frac 1{\ell(e)} = \sum_{e'\mapsto e} \frac{1}{\ell(e')}. \]

  In order to prove that $\Sigma$ is the quotient $\Sigma'/G$ in the sense of Definition~\ref{def:quotient.graph}, we must show that $\pi$ is harmonic.  Let $v'\in\Sigma'$ be an interior vertex, let $v = \pi(v')$, and let $e$ be an edge of $\Sigma$ beginning at $v$.  Fix an edge $e'$ beginning at $v'$ and mapping to $e$.  If $e''$ is another such edge, then there exists $\sigma\in G$ such that $\sigma(e') = e''$: the orientation is preserved because the vertices of $e''$ belong to distinct $G$-orbits.  Hence the stabilizer $G_{v'}$ of $v'$ acts transitively on the set of edges beginning at $v'$ and mapping to $e$, and the number of such edges is $[G_{v'}:G_{e'}]$.  This number is independent of the choice of $v'\mapsto v$ and $e'\mapsto e$, so we denote it by $\epsilon(e)$.  For $\sigma\in G_{v'}$ we have $\ell(\sigma e') = \ell(e')d_{v'}(\sigma)$ by~\eqref{eq:dvprimephi}; hence $\ell(\sigma^2 e') = \ell(e')d_{v'}(\sigma)^2$, etc.  Since $\sigma^me' = e'$ for some $m\geq 1$, this shows $d_{v'}(\sigma)^m = 1$, so $d_{v'}(\sigma) = 1$.  It follows that $\ell(e'')$ is the same number for all $e''$ beginning at $v'$ and mapping to $e$; we denote this number by $\ell(e,v')$.  If $\sigma(v') = v''$ and $\sigma(e') = e''$ for some $\sigma\in G$ then $\ell(e'') = \ell(e')d_{v'}(\sigma)$, so that $\ell(e,v'') = d_{v'}(\sigma)\,\ell(e,v')$.  It follows that $d_{v'}(\sigma)$ is independent of the choice of $\sigma$ sending $v'$ to $v''$; we set $d(v',v'') = d_{v'}(\sigma)$ for any such $\sigma$, so that $\ell(e,v'') = d(v',v'')\,\ell(e,v')$.

  Now we compute
  \[\begin{split}
      d_{v'}(\pi) &= \ell(e)\sum_{\substack{(e')^-=v'\\e'\mapsto e}}\frac 1{\ell(e')}
      = \frac{\epsilon(e)/\ell(e,v')}{\sum_{e''\mapsto e}1/\ell(e'')}
      = \frac{\epsilon(e)/\ell(e,v')}{\sum_{v''\mapsto v}\epsilon(e)/\ell(e,v'')} \\
      &= \frac{1/\ell(e,v')}{\sum_{v''\mapsto v}1/(d(v',v'')\,\ell(e,v'))}
      = \left( \sum_{v''\mapsto v}\frac 1{d(v',v'')} \right)\inv.
    \end{split} \]
  This is independent of the choice of $e$, so $\pi$ is harmonic, as desired.

  Unicity is automatic from the universal property of Lemma~\ref{lem:categorical.graph.quotient}.
\end{proof}

Our next goal is to show that $H^{p,q}(\Sigma,\del\Sigma) = H^{p,q}(\Sigma',\del\Sigma')^G$ if $\Sigma = \Sigma'/G$.  If this were true on the level of forms, i.e.\ if it were true that $\cA^{p,q}(\Sigma,\del\Sigma) = \cA^{p,q}(\Sigma',\del\Sigma')^G$ as in~\otherfile{Proposition~\ref*{I-Galois invariant forms}}, then the result would follow formally as in~\otherfile{Corollary~\ref*{I-Galois action on cohomology}}.  However, this is not the case in general.  For instance, let $\Sigma'$ be the circle of circumference $2$ around the origin in $\R^2$, with the points on the $x$-axis being vertices and with no boundary, and let $G = \Z/2\Z$ act on $\Sigma'$ by flipping over the $x$-axis.  Let $\pi\colon\Sigma'\to\Sigma$ be the quotient.  Then $\Sigma$ is a line segment $e$ of length $1$.  Consider the function $f\colon \Sigma\to\R$ defined by $f\circ t_e(x) = \exp(-1/x^2)\cdot \exp(-1/(1-x)^2)$, and let $g = f\circ\pi$.  Then $g\in\cA^{0,0}(\Sigma')^G$, but $f\notin\cA^{0,0}(\Sigma)$ because $f$ is not constant in a neighborhood of the vertices.  Similar examples can be constructed in any bidegree, and similar problems can arise at valence-$2$ vertices of $\Sigma$.  This issue also makes it awkward to define a transfer map $\cA^{p,q}(\Sigma',\del\Sigma')\to\cA^{p,q}(\Sigma,\del\Sigma)$.

One could imagine modifying our definition of smooth forms on graphs to satisfy $\cA^{p,q}(\Sigma,\del\Sigma) = \cA^{p,q}(\Sigma',\del\Sigma')^G$ without changing the Dolbeault cohomology groups, but our definition is designed to be compatible with pullback of Lagerberg forms (Section~\ref{Sec:lagerberg-forms}) and with the non-Archimedean situation (Section~\ref{Sec:forms-curves}).  Therefore we must live with a tedious, ad-hoc proof of the following Proposition.

\begin{prop}\label{prop:quotient.graph.cohom}
  Let $\pi\colon\Sigma'\to\Sigma$ be a harmonic map of weighted metric graphs with boundary, and suppose that $\Sigma$ is the quotient of $\Sigma'$ by a finite group $G$.  Then the canonical homomorphisms
  \begin{equation}\label{eq:graph.cohom.quotient}
    \pi^*\colon H^{p,q}(\Sigma,\del\Sigma)\To H^{p,q}(\Sigma',\del\Sigma')^G
  \end{equation}
  are isomorphisms for all $p,q\in\{0,1\}$.
\end{prop}

\begin{proof}
  We may assume that $\Sigma$ is connected.  If $\Sigma$ has no edges then neither does $\Sigma'$, in which case the result is clear.  Hence we assume that $\Sigma$ contains an edge.  By passing to unweightings, we may assume that all edge weights are trivial.

  The pullback homomorphisms $\pi^*\colon\cA^{p,q}(\Sigma,\del\Sigma)\to\cA^{p,q}(\Sigma',\del\Sigma')$ are injective for all $p,q$ because a form on $\Sigma$ can be recovered from its pullback to $\Sigma'$.  We show that~\eqref{eq:graph.cohom.quotient} is an isomorphism in each bidegree separately.

\subsubsection{$H^{0,0}$:} By Lemma~\ref{lem:quotient.in.top.cat}, the map $\pi$ is open, so a connected component of $\Sigma'$ surjects onto $\Sigma$ (the image of such a connected component is open, and is closed because connected components are compact).  It follows that $G$ acts transitively on $\pi_0(\Sigma')$, so $\pi^*$ is an isomorphism on $H^{0,0}$.

\subsubsection{$H^{0,1}$:} The higher cohomology groups of a finite group are torsion, so $W\mapsto W^G$ is an exact functor on the category of $\R$-vector spaces.  Hence we have an exact sequence
  \[ 0 \To H^{0,0}(\Sigma',\del\Sigma')^G \To \cA^{0,0}(\Sigma',\del\Sigma')^G
    \To \cA^{0,1}(\Sigma',\del\Sigma')^G \To H^{0,1}(\Sigma',\del'S')^G \To 0.
  \]
  This gives rise to a homomorphism of short exact sequences
  \[\xymatrix@R=.3in{
      0 \ar[r] &
      {\cA^{0,0}(\Sigma,\del\Sigma)/H^{0,0}(\Sigma,\del\Sigma)} \ar[r] \ar[d] &
      {\cA^{0,1}(\Sigma,\del\Sigma)} \ar[r] \ar[d] &
      {H^{0,1}(\Sigma,\del\Sigma)} \ar[r] \ar[d] & 0 \\
      0 \ar[r] &
      {\cA^{0,0}(\Sigma',\del\Sigma')^G/H^{0,0}(\Sigma',\del\Sigma')^G} \ar[r] &
      {\cA^{0,1}(\Sigma',\del\Sigma')^G} \ar[r] &
      {H^{0,1}(\Sigma',\del\Sigma')^G} \ar[r] & {0\rlap.}
    }\]
  The middle vertical arrow is an injection, and since $H^{0,0}(\Sigma,\del\Sigma)\isom H^{0,0}(\Sigma',\del\Sigma')^G$, the cokernel of the left vertical arrow is the same as the cokernel of $\cA^{0,0}(\Sigma,\del\Sigma)\to\cA^{0,0}(\Sigma',\del\Sigma')^G$.  By the snake lemma, showing that the right vertical arrow is an injection is equivalent to proving that the homomorphism
  \begin{equation}\label{eq:homom.of.cokernels}
    \cA^{0,0}(\Sigma',\del\Sigma')^G/\cA^{0,0}(\Sigma,\del\Sigma)
    \To \cA^{0,1}(\Sigma',\del\Sigma')^G/\cA^{0,1}(\Sigma,\del\Sigma)
  \end{equation}
  of cokernels is an injection.

  Let $g\in\cA^{0,0}(\Sigma',\del\Sigma')^G$, and suppose that $\d''g = \pi^*\omega$ for some $\omega = (h_e\,\d''t_e)\in\cA^{0,1}(\Sigma,\del\Sigma)$.   Since $\Sigma$ is the quotient of $\Sigma'$ in the category of topological spaces, there exists a unique continuous function $f\colon\Sigma\to\R$ such that $f\circ\pi = g$.  We claim that $f$ is smooth.  It is clear that $f$ is smooth on edges, so we must verify the other conditions of Definition~\ref{def:smooth.on.graph}.  If $e'$ is an edge of $\Sigma'$ and $e = \pi(e')$ then
  \[ d_{e'}(\pi)\frac{\d f}{\d t_e}\circ\pi = \frac{\d g}{\d t_{e'}} = h_{e}\circ \pi\,d_{e'}(\pi), \]
  so $\d f/\d t_e = h_e$ for each edge $e$.

  Let $v\in\Sigma\setminus\del\Sigma$.
  \begin{enumerate}
  \item Suppose that $v$ has valency $1$, with adjacent edge $e$.  Since $h_e$ is zero in a neighborhood of $v$, it follows that $f$ is constant in a neighborhood of $v$.
  \item Suppose that $v$ has valency $2$ with outgoing edges $e_1,e_2$.  Then smoothness of $(h_{e_1},-h_{e_2})$ at $v$ implies smoothness of $f$ at $v$.
  \item If $v$ has valency greater than $2$, then
    \[ 0 = \sum_{v=e^-} h_e(v) = \sum_{v=e^-} \frac{\d f}{\d t_e}(v). \]
  \end{enumerate}
  This proves that~\eqref{eq:homom.of.cokernels} is injective.

  Now we show that $H^{0,1}(\Sigma,\del\Sigma)\to H^{0,1}(\Sigma',\del\Sigma')^G$ is surjective.  Fix $\omega' = (g_{e'}\,\d''t_{e'})\in\cA^{0,1}(\Sigma',\del\Sigma')^G$.  Being fixed by the $G$-action means that there exist smooth functions $(f_e)_{e\in E(\Sigma)}$ such that $g_{e'} = f_e\circ\pi\,d_{e'}(\pi)$ for $e'\mapsto e$.   Note that $\omega = (f_e\,\d''t_e)$ need not be a smooth $(0,1)$-form: for instance, it need not be zero in a neighborhood of an interior vertex of valency~1.  However, for a path $\gamma$ in $\Sigma$ it still makes sense to integrate $\omega$ along $\gamma$ as in~\eqref{eq:integral.along.path}.  Choose a vertex $v_0\in\Sigma$ and closed paths $\gamma_1,\ldots,\gamma_g$ generating $\pi_1(\Sigma,v_0)$.  We showed in~\secref{sec:computation-h0-1} that there exists $\eta\in\cA^{0,1}(\Sigma,\del\Sigma)$ such that $\int_{\gamma_i}\eta = \int_{\gamma_i}\omega$ for all $i$.  This implies $\int_{\gamma}\eta = \int_{\gamma}\omega$ for any closed path $\gamma$.

  We claim that $[\omega'] = \pi^*[\eta]$ in $H^{0,1}(\Sigma',\del\Sigma')$.  Let $\gamma' = e_1'\cup\cdots\cup e_n'$ be a path in $\Sigma'$, let $e_i' = \pi(e_i)$, and let $\gamma = \pi(\gamma') = e_1\cup\cdots\cup e_n$.  Using
  \[ g_{e'}\circ t_{e'}(x) = d_{e'}(\pi)\,f_e\circ\pi\circ t_{e'}(x) = d_{e'}(\pi)\,f_e\circ t_e(d_{e'}(\pi)x) \]
  for $e'\mapsto e$, we calculate
  \[\begin{split}
      \int_{\gamma'}\omega' &= \sum_{i=1}^n \int_0^{\ell(e_i')}g_{e_i'}\circ t_{e_i'}(x)\,\d x \\
      &= \sum_{i=1}^n \int_0^{\ell(e_i')}f_{e_i}\circ t_{e_i}(d_{e_i'}(\pi)x)\,d_{e_i'}(\pi)\,\d x \\
      &= \sum_{i=1}^n \int_0^{\ell(e_i)}f_{e_i}\circ t_{e_i}(x)\,\d x = \int_\gamma\omega.
    \end{split}\]
  In particular, if $\gamma'$ is a closed path in $\Sigma'$ then $\gamma$ is a closed path in $\Sigma$, so $\int_{\gamma'}\omega' = \int_\gamma\omega = \int_\gamma\eta$.  The same calculation shows $\int_{\gamma'}\pi^*\eta = \int_\gamma\eta$, so $\int_{\gamma'}\pi^*\eta = \int_{\gamma'}\omega'$ for all closed paths $\gamma'$.  This implies $[\omega'] = [\pi^*\eta]$ as in~\secref{sec:computation-h0-1}.  This concludes the proof that $H^{0,1}(\Sigma,\del\Sigma)\to H^{0,1}(\Sigma',\del\Sigma')^G$ is an isomorphism.

  \subsubsection{$H^{1,0}$:} We must show that if $\omega'\in\cA^{1,0}(\Sigma',\del\Sigma')$ is a smooth form that is constant on edges (see~\secref{sec:computation-h1-0}) and invariant under $G$, then $\omega' = \pi^*\omega$ for $\omega\in\cA^{1,0}(\Sigma,\del\Sigma)$.  Write $\omega' = (c_{e'}\,\d't_{e'})$ for $c_{e'}\in\R$.  By $G$-invariance, there exist numbers $C_e\in\R$ for $e\in E(\Sigma)$ such that $c_{e'} = C_e\,d_{e'}(\pi)$ for $e'\mapsto e$.  We claim that $\omega=(C_e\,\d't_e)$ is in $\cA^{1,0}(\Sigma,\del\Sigma)$.  Since the $C_e$ are constant, all three conditions of Definition~\ref{def:diff.form.graph} are equivalent to showing $\sum_{v=e^-}C_e = 0$ for vertices $v\in\Sigma\setminus\del\Sigma$.  Let $v\in\Sigma\setminus\del\Sigma$, and choose $v'\in\Sigma'$ mapping to $v$; note that $v'\notin\del\Sigma'$.  For $e$ an edge starting at $v$, we have
  \[ d_{v'}(\pi) = \sum_{\substack{(e')^-=v'\\e'\mapsto e}} d_{e'}(\pi) \]
  by definition. Multiplying by $C_e = c_{e'}/d_{e'}(\pi)$ and rewriting gives
  \[ C_e = \frac 1{d_{v'}(\pi)}\sum_{\substack{(e')^-=v'\\e'\mapsto e}}d_{e'}(\pi)\cdot\frac{c_{e'}}{d_{e'}(\pi)}
    = \frac 1{d_{v'}(\pi)}\sum_{\substack{(e')^-=v'\\e'\mapsto e}} c_{e'}.
  \]
  Summing over all edges $e$ starting at $v$ then gives
  \[ \sum_{e^-=v}C_e
    = \frac 1{d_{v'}(\pi)}\sum_{(e')^-=v'} c_{e'} = 0,
  \]
  as required.

\subsubsection{$H^{1,1}$:} By Proposition~\ref{prop:dolbeault.graphs}, we only need to show that the map $\pi^*\colon H^{1,1}(\Sigma)\to H^{1,1}(\Sigma')$ is nonzero when $\del\Sigma=\emptyset$.  This is true because pulling back a bump function on an edge gives a $(1,1)$-form with nonzero integral: see~\secref{sec:computation-h1-1}.
\end{proof}

\subsection{Integration of Pullbacks}
We conclude this section by showing that integration of forms has the expected behavior with respect to pullback.

\begin{lem}\label{lem:pullback.integrate}
  Let $\phi\colon\Sigma'\to\Sigma$ be a harmonic morphism of weighted metric graphs with boundary, and suppose that $\phi$ has a well-defined degree $d(\phi)$ in the sense of Definition~\ref{def:harmonic.degree}.  Then
  \[ \int_{\Sigma'}\phi^*\omega = d(\phi)\int_\Sigma\omega \qquad
    \int_{\del\Sigma'}\phi^*\eta' = d(\phi)\int_{\del\Sigma}\eta' \qquad
    \int_{\del\Sigma'}\phi^*\eta'' = d(\phi)\int_{\del\Sigma}\eta''
  \]
  for all $\omega\in\cA^{1,1}(\Sigma,\del\Sigma),~\eta'\in\cA^{1,0}(\Sigma,\del\Sigma),$ and $\eta''\in\cA^{0,1}(\Sigma,\del\Sigma)$.
\end{lem}

\begin{proof}
  By Lemma~\ref{lem:integral.unweighting} we may replace $\Sigma$ and $\Sigma'$ by their unweightings to assume that all edge weights are trivial.  If $\omega = (f_e\,\d't_e\d''t_e)$ then $\phi^*\omega = (f_{\phi(e')}\circ\phi\,\d_{e'}(\phi)^2\,\d't_{e'}\d''t_{e'})$.  Let $e'$ be an edge of $\Sigma'$.  If $\phi(e')$ is a vertex then $d_{e'}(\phi)=0$; otherwise we let $e=\phi(e')$, and we calculate
  \[\begin{split}
      \int_0^{\ell(e')}f_e\circ\phi\circ t_{e'}(x)\,d_{e'}(\phi)^2\,\d x
      &= d_{e'}(\phi)\int_0^{\ell(e)} f_e\circ t_e(x)\,\d x.
    \end{split}\]
  Summing over all $e'\mapsto e$ and using~\eqref{eq:degree.edge.sum} gives
  \[ \sum_{e'\mapsto e}\int_0^{\ell(e')}f_e\circ\phi\circ t_{e'}(x)\,d_{e'}(\phi)^2\,\d x
    = d(\phi)\int_0^{\ell(e)} f_e\circ t_e(x)\,\d x;
  \]
  summing over all edges $e$ proves $\int_{\Sigma'}\phi^*\omega = d(\phi)\int_\Sigma\omega$.

  If $\eta' = (g_e\,\d't_e)$ then $\phi^*\eta'=(g_{\phi(e')}\circ\phi\,d_{e'}(\phi)\,\d't_{e'})$.  Let $e$ be an edge of $\Sigma$.  We use~\eqref{eq:bdy.int.2} and~\eqref{eq:degree.edge.sum} to compute
  \[ \sum_{e'\mapsto e} g_{\phi(e')}\circ\phi(e'^-)\,d_{e'}(\phi)
    = g_e(e^-)\sum_{e'\mapsto e}d_{e'}(\phi) = d(\phi)\,g_e(e^-).
  \]
  Summing over all edges of $\Sigma$, and noting that $g_{\phi(e')}\circ\phi\,d_{e'}(\phi)=0$ if $e'$ maps to a vertex of $\Sigma$, we have
  \[ \int_{\del\Sigma'}\phi^*\eta'
    = \sum_e d(\phi)\,g_e(e^-) = d(\phi)\int_{\del\Sigma}\eta'. \]
  The proof for $\eta''\in\cA^{0,1}(\Sigma,\del\Sigma)$ is identical.
\end{proof}


\section{Relationship to Lagerberg forms}
\label{Sec:lagerberg-forms}

For an open subset $U\subset\R^n$, we let $\cA^{p,q}(U)$ denote the space of Lagerberg forms on $U$, as in~\cite{chambert_ducros12:forms_courants}.  In this section we show that our differential forms on graphs arise naturally as local pullbacks of Lagerberg forms on $\R^n$ under suitable harmonic tropicalizations.  Throughout we fix a weighted metric graph with boundary $(\Sigma,\del\Sigma)$ as in Section~\ref{Sec:forms-graphs}.

\subsection{$(R,\Gamma)$-Harmonic Tropicalizations}
Recall that we defined harmonic functions in Definition~\ref{def:harmonic.func.graph}.

\begin{defn} \label{def:harmonic tropicalization}
  A \defi{harmonic} (resp.\ \defi{$R$-harmonic}, resp.\ \defi{$(R,\Gamma)$-harmonic}) \defi{tropicalization} of $(\Sigma,\del\Sigma)$ is a function $h = (h_1,\ldots,h_n)\colon\Sigma\to\R^n$, where $h_1,\ldots,h_n$ are harmonic (resp.\ $R$-harmonic, resp.\ $(R,\Gamma)$-harmonic) functions on $(\Sigma,\del\Sigma)$.
\end{defn}

\begin{defn}\label{def:gamma.ratl.pts}
  Let $\Gamma\subset\R_+$ be an additive subgroup containing the lengths of the edges of $\Sigma$.  A point $x$ of $\Sigma$ is called \defi{$\Gamma$-rational} if $x = t_e(y)$ for some edge $e$ and some $y\in\Gamma$.  The set of $\Gamma$-rational points of $\Sigma$ is denoted $\Sigma(\Gamma)$.
\end{defn}

For the rest of this section, we fix an additive subgroup $\Gamma\subset\R_+$ containing the lengths of the edges of $\Sigma$.  Either $\Gamma$ is discrete or it is dense in $\R$.  We will need a dense subgroup in the sequel, so we use the following notation.

\begin{notn}\label{notn:barGamma}
  If $\Gamma$ is discrete, we let $\bar\Gamma$ denote the saturation of $\Gamma$ in $\R_{+}$.  Otherwise, we set $\bar\Gamma = \Gamma$.
\end{notn}

\subsection{Pullbacks Under Harmonic Tropicalizations}
In this subsection, we define pullbacks of Lagerberg forms under harmonic tropicalizations.

\begin{defn}\label{def:pullback.forms}
  Fix a harmonic tropicalization $h = (h_1,\ldots,h_n)\colon\Sigma\to\R^n$ of $(\Sigma,\del\Sigma)$ as in Definition~\ref{def:harmonic tropicalization}, and let $U\subset\R^n$ be an open neighborhood of $h(\Sigma)$.
  \begin{enumerate}
  \item For $g\in\cA^{0,0}(U)$ we define $h^*g = g\circ h$.
  \item For $\eta = \sum_{i=1}^n g_i\,\d'x_i\in\cA^{1,0}(U)$ we define $h^*\eta = (f_e\,\d't_e)$ by
    \[ f_e = \sum_{i=1}^n\frac{\d h_i}{\d t_e}\,g_i\circ h. \]
  \item For $\eta = \sum_{i=1}^n g_i\,\d''x_i\in\cA^{0,1}(U)$ we define $h^*\eta = (f_e\,\d''t_e)$ by
    \[ f_e = \sum_{i=1}^n\frac{\d h_i}{\d t_e}\,g_i\circ h. \]
  \item For $\eta = \sum_{i,j=1}^n g_{ij}\,\d'x_i\wedge\d''x_j\in\cA^{1,1}(U)$ we define $h^*\eta = (f_e\,\d't_e\d''t_e)$ by
    \[ f_e = \sum_{i,j=1}^n\frac{\d h_{i}}{\d t_e}\frac{\d h_j}{\d t_e}\,g_{ij}\circ h. \]
  \end{enumerate}
\end{defn}

\begin{lem}\label{lem:pullbacks.are.smooth}
  For $p,q\in\{0,1\}$, the pullback $h^*$ defines a homomorphism of differential bigraded algebras $\cA^{p,q}(U)\to\cA^{p,q}(\Sigma,\del\Sigma)$.  This factors through an \emph{injective} homomorphism $\cA^{p,q}(h(\Sigma))\to\cA^{p,q}(\Sigma,\del\Sigma)$.
\end{lem}

Here we use the space of smooth differential forms $\cA^{p,q}(h(\Sigma))$ on the polytopal complex $h(\Sigma)$ as defined in~\cite{gubler16:forms_currents}.

\begin{proof}
  First we need to show that the pullback of a Lagerberg form is a smooth form on $(\Sigma,\del\Sigma)$.
  We verify only the case $(p,q)=(0,0)$, as the other cases are similar.
  By Lemma~\ref{lem:unweighting.forms}, we may check smoothness on the unweighting $\Sigma_0$, hence we may assume that all weights are $1$.
  Let $g\in\cA^{0,0}(U)$, and let $f = h^*g = g\circ h$.  Let $e$ be an edge of $\Sigma$.  Then
  \[ f\circ t_e(x) = g\circ h\circ t_e(x) = g(a_1x+b_1,\,\ldots,\,a_nx+b_n) \]
  for some $a_i,b_i\in\R$, which shows that $f$ is smooth on edges.  Now we verify the three conditions of Definition~\ref{def:smooth.on.graph}.
  \begin{enumerate}
  \item Let $v\in\Sigma\setminus\del\Sigma$ have valency $1$, and let $e$ be the outgoing edge at $v$.  Then $h$ is constant on $e$, so $f$ is constant in a neighborhood of $v$.
  \item Let $v\in\Sigma\setminus\del\Sigma$ have valency $2$, with outgoing edges $e_1,e_2$.  The second condition of Definition~\ref{def:smooth.on.graph} (as applied to $h$) implies that $h$ is linear on $e_1\cup e_2$, so $(f|_{e_1}, f|_{e_2})$ is smooth at $v$.
  \item Let $v\in\Sigma\setminus\del\Sigma$ have valency greater than $2$, with outgoing edges $e_1,e_2,\ldots,e_r$.  Let $f_i = f\circ t_{e_i}$, and let $x_1,\ldots,x_n$ be the coordinates on $\R^n$.  We compute
  \begin{gather*}
    \sum_{i=1}^r\frac{\d f}{\d t_{e_i}}(v)
    = \sum_{i=1}^r\frac{\d f_i}{\d x}(0)
    = \sum_{i=1}^r\frac{\d(g\circ h\circ t_{e_i})}{\d x}(0)
    = \sum_{i=1}^r\sum_{j=1}^n \frac{\d g}{\d x_j}(h(0))\frac{\d(h_j\circ t_{e_j})}{\d x}(0) \\
    = \sum_{j=1}^n\frac{\d g}{\d x_j}(h(0))\sum_{i=1}^r\frac{\d h_j}{\d t_{e_i}}(v) = 0.
  \end{gather*}
  \end{enumerate}
  Hence $f$ is a smooth function on $\Sigma$.

  We leave it to the reader to verify that $h^*$ respects the differentials $\d'',\d''$ and the wedge product, so that $h^*$ is a homomorphism of differential bigraded algebras.  It remains then to show that $h^*$ factors through an injective homomorphism $\cA^{p,q}(h(\Sigma))\to\cA^{p,q}(\Sigma,\del\Sigma)$. Let $e$ be an edge of $\Sigma$ such that $\sigma = h(e)$ is a line segment. Note that $p_e = h\circ t_e\colon[0,\ell(e)]\to\sigma$ is an affine parameterization of $\sigma$. If $U\subset\R^n$ is an open neighborhood of $h(\Sigma)$ and $\eta\in\cA^{p,q}(U)$, then the restriction $\eta|_\sigma$ of $\eta$ to $\sigma$ is by definition the form $p_e^*\eta$ on $[0,\ell(e)]$.  Any form on $h(\Sigma)$ is by definition a collection of forms $(\eta|_\sigma)$ arising in this way; it is zero if and only if $\eta|_\sigma=0$ for each edge $\sigma$ of $h(\Sigma)$.   Definition~\ref{def:pullback.forms} is set up to be compatible with the pullback of Lagerberg forms with respect to $p_e$, as in~\cite[2.3]{gubler16:forms_currents}.  Applying functoriality of pullbacks of Lagerberg forms with respect to affine maps (using $p_e\inv$), we deduce that one can recover the restriction of $\eta$ to $h(\Sigma)$ from $h^*\eta$, which shows that $h^*\colon\cA^{p,q}(h(\Sigma))\to\cA^{p,q}(\Sigma,\del\Sigma)$ is injective.
\end{proof}

Next we show that the pullback of Lagerberg forms is compatible with with harmonic maps of graphs.  Note that if $\phi\colon\Sigma'\to\Sigma$ is a harmonic map of weighted metric graphs with boundary and $h\colon\Sigma\to\R^n$ is a harmonic tropicalization then $h\circ\phi\colon\Sigma'\to\R^n$ is again a harmonic tropicalization by Lemma~\ref{lem:pullback.harmonic}.

\begin{lem}\label{lem:pullback.lagerberg.harmonic}
  Let $\phi\colon\Sigma'\to\Sigma$ be a harmonic map of weighted metric graphs with boundary, let $h\colon\Sigma\to\R^n$ be a harmonic tropicalization, and let $h' = h\circ\phi\colon\Sigma'\to\R^n$.  Then this triangle commutes for all open neighborhoods $U$ of $h(\Sigma)$ and all $p,q\in\{0,1\}$:
  \[ \xymatrix @C=-.2in{
      {\cA^{p,q}(U)} \ar[rr]^(.45){h'^*} \ar[dr]_{h^*} & & {\cA^{p,q}(\Sigma',\del\Sigma')} \\
      & {\cA^{p,q}(\Sigma,\del\Sigma)\rlap.} \ar[ur]_{\phi^*}
    } \]
\end{lem}

\begin{proof}
  This follows easily from the definitions, using the fact that if $e'\subset\Sigma'$ is an edge mapping to $e\subset\Sigma$, then $\d(h_i\circ\phi)/\d t_{e'} = d_{e'}(\phi)\,\d h_i/\d t_e$ for a harmonic function $h_i\colon\Sigma\to\R$.
\end{proof}

In particular, if $\nu\colon\Sigma_0\to\Sigma$ is the unweighting of $\Sigma$ and $h\colon\Sigma\to\R^n$ is a harmonic tropicalization, then $h\circ\nu$ is a harmonic tropicalization and $(h\circ\nu)^*=\nu^*h^*$.  See Remark~\ref{rem:unweighting.and.harmonic}.

\subsection{Smooth Forms Are Local Pullbacks of Lagerberg Forms}
Next we show that any smooth form on $\Sigma$ is locally the pullback of a Lagerberg form under a $(\Z,\Gamma)$-harmonic tropicalization.  First we discuss what kind of neighborhoods we will consider when we define ``local'' pullbacks.  We want to allow edge lengths in $\bar\Gamma$ (because it is dense in $\R$), but we still need to make sense of $(\Z,\Gamma)$-harmonic tropicalizations in order to relate to the non-Archimedean situation in Section~\ref{Sec:forms-curves}, which leads to the following ad-hoc definition.

\begin{defn}\label{def:rational.subgraph}
  Let $U\subset\Sigma$ be a subgraph with vertices in $\Sigma(\bar\Gamma)$ (Definition~\ref{def:gamma.ratl.pts}), regarded as a weighted metric graph with boundary as in Definition~\ref{def:subgraph}.  A function $h\colon U\to\R$ is \defi{$(\Z,\Gamma)$-harmonic} on $(U,\del U)$ if it is harmonic with integral slopes, and the following condition holds for each edge $e\subset U$.  Let $e'$ be the edge of $\Sigma$ containing $e$, and let $h'\colon e'\to\R$ be the unique affine-linear function extending $h|_e$.  Then we require that $h'$ take values in $\Gamma$ on the endpoints of $e'$.

  We define a $(\Z,\Gamma)$-harmonic tropicalization of $(U,\del U)$ as an $n$-tuple of $(\Z,\Gamma)$-harmonic functions, as in Definition~\ref{def:harmonic tropicalization}.
\end{defn}

Recall from~\secref{sec:subgraphs} that the pullback $\cA^{p,q}(\Sigma,\del\Sigma)\to\cA^{p,q}(U,\del U)$ under inclusion of a subgraph is called \defi{restriction}.

\begin{prop}\label{prop:local.pullbacks.graphs}
  Let $p,q\in\{0,1\}$, let $\omega\in\cA^{p,q}(\Sigma,\del\Sigma)$, and let $x\in\Sigma$.  Then there is a neighborhood $U\subset\Sigma$ of $x$ which is a subgraph with $\bar\Gamma$-rational vertices, a $(\Z,\Gamma)$-harmonic tropicalization $h\colon U\to\R^n$ of $(U,\del U)$, and a Lagerberg form $\eta\in\cA^{p,q}(\R^n)$, such that $h^*\eta = \omega|_U$.  Moreover, if $x$ is not an interior leaf vertex then we can choose $U$ and $h$ independent of $\omega$.
\end{prop}

We will use the following construction in the proof of Proposition~\ref{prop:local.pullbacks.graphs}.  We are very grateful to Robert Bryant for providing us with the argument.

\begin{lem}\label{lem:extend.smooth}
  Let $n\geq 2$, for $1\leq i\leq n$ let $v_i$ be the $i$th unit coordinate vector in $\R^n$, and let $v_0 = -\sum_{i=1}^n v_i = (-1,-1,\ldots,-1)$.  For $i=0,1,\ldots,n$ let $e_i$ denote the line segment from $0$ to $\ell_iv_i$, where $\ell_0,\ell_1,\ldots,\ell_n>0$.  Let $f_i\colon e_i\to\R$ be smooth functions for $i=0,1,\ldots,n$ (smooth on one side at the endpoints), and suppose that $f_0(0) = f_1(0) = \cdots = f_n(0)$, and that
  \[ \sum_{i=0}^n \frac{\d f_i}{\d v_i}(0) = 0, \]
  where $\d/\d v_i$ denotes the directional derivative.  Then there is a smooth function $f\colon\R^n\to\R$ such that $f|_{e_i} = f_i$.
\end{lem}

\begin{proof}
  First we construct a formal power series to serve as the Taylor expansion at the origin of such a function $f$.  To that end, we write the Taylor series for $f_i(tv_i)$ as $F_i(t) = \sum_{j\geq 1} \alpha_{i,j}t^j$.  Define $p_0 = f_0(0)$ and $p_1 = \sum_{i=1}^n\alpha_{i,1} x_i$, and note that
  \[ p_1(v_0) = p_1(-1,-1,\ldots,-1) = -\sum_{i=1}^n\frac{\d f_i}{\d v_i}(0) = \frac{\d f_0}{\d v_0}(0) = \alpha_{0,1}. \]
  For $j\geq 2$ choose a homogeneous polynomial $p_j$ in $n$ variables satisfying $p_j(v_i) = \alpha_{i,j}$ for $i=0,1,\ldots,n$; this is possible because $n,j\geq 2$.  Let $F = \sum_{n\geq 0} p_n$.  By construction, $F(tv_i) = F_i(t)$ for all $i=0,1,\ldots,n$.

  By a theorem of Borel~\cite[Theorem~1.5.4]{narasimh85:analysis_real_complex_manifold}, there exists a smooth function $g$ on $\R^n$ whose Taylor expansion at the origin is $F$.  Replacing $f_i$ by $f_i-g$, we may assume that all $f_i$ vanish to infinite order at $0$.  By a similar argument, one can extend $f_i$ to a smooth function on $\R v_i$.

  For $i\geq 1$ define $q_i(x_1,\ldots,x_n) = \sum_{j\neq i}x_j^2$, and let $q_0(x_1,\ldots,x_n) = \sum_{i=2}^n (x_i-x_1)^2$.  Then $q_i$ vanishes on $\R v_i$, and is positive elsewhere.  The function $Q_i = \prod_{j\neq i} q_j$ vanishes on each line $\R v_j$ for $j\neq i$, but only vanishes to finite order at the origin on $\R v_i$, and is positive elsewhere on $\R v_i$.  It follows that $f_i/Q_i$ extends to a smooth function on all of $\R v_i$, since $f_i$ vanishes to infinite order at $0$.  Therefore, $f_i/Q_i$ extends to a smooth function $g_i$ on all of $\R^n$ (compose with orthogonal projection onto the line).  Now consider the sum
  \[ f = Q_0 g_0 + Q_1 g_1 + \cdots + Q_n g_n. \]
  This is a smooth function on $\R^n$ that restricts to $f_i$ on $\R v_i$.
\end{proof}

Contained in the proof of Lemma~\ref{lem:extend.smooth} is the following useful fact:

\begin{lem}\label{lem:extend.smooth.1d}
  Let $f\colon[0,1]\to\R$ be a smooth function, smooth on one side at the endpoints.  Then $f$ is the restriction of a smooth function on $\R$.
\end{lem}

\subsubsection{Proof of Proposition~\ref{prop:local.pullbacks.graphs}, unweighting step}
Lemmas~\ref{lem:unweighting.forms} and~\ref{lem:pullback.lagerberg.harmonic} allow us to pass to the unweighting $\Sigma_0$ of $\Sigma$ in our construction, with the following caveat: if $h\colon\Sigma_0\to\R$ is a $(\Z,\Gamma)$-harmonic function, then the slopes of $h$ regarded as a function on $\Sigma$ are only rational numbers (the slope on $e$ is divided by $w(e)$).  Thus the following argument is needed.

Let $N$ be the least common multiple of the weights of the edges of $\Sigma$.
Suppose that $h_1,\ldots,h_n\colon\Sigma\to\R$ are harmonic functions with slopes in $\frac 1N\Z$, taking values in $\frac1N\Gamma$ on the vertices.  Let $h = (h_1,\ldots,h_n)\colon\Sigma\to\R^n$, let $\omega\in\cA^{p,q}(\R^n)$, and let $\eta=h^*\omega$.  The harmonic function $Nh\colon\Sigma\to\R^n$ (composition of $h$ with multiplication by $N$ on $\R^n$) is $(\Z,\Gamma)$-harmonic.  We have
\[ (Nh)^*\left[\frac 1N\right]^*\omega = h^*[N]^*\left[\frac 1N\right]^*\omega = h^*\omega = \eta, \]
so there also exists a $(\Z,\Gamma)$-harmonic tropicalization $Nh$ and a Lagerberg form $[1/N]^*\omega$ pulling back to $\eta$.  Therefore, we may pass to the unweighting of $\Sigma$ and replace $\Gamma$ by $\frac1N\Gamma$ to assume that all edge weights of $\Sigma$ are $1$.

\subsubsection{Proof of Proposition~\ref{prop:local.pullbacks.graphs} for $\cA^{0,0}$}
\label{sec:pullback.A00}
  Let $f\in\cA^{0,0}(\Sigma,\del\Sigma)$, and fix $x\in\Sigma$.  We treat several cases:
  \begin{enumerate}
  \item If $x$ is contained in the interior of an edge $e$, then there exists a neighborhood $U$ consisting of a segment contained in the interior of $e$ with $\bar\Gamma$-rational endpoints; then $\del U$ consists of the endpoints of $U$.  In this case, $h = t_e\inv|_U\colon U\to\R$ is a $(\Z,\Gamma)$-harmonic tropicalization, and $f|_U = g\circ h$, where $g\colon\R\to\R$ is a smooth function extending $f\circ t_e$ (Lemma~\ref{lem:extend.smooth.1d}).
  \item If $x$ is an interior vertex of valency $2$ then we may remove $x$ from the vertex set of $\Sigma$ without affecting $\cA^{0,0}(\Sigma,\del\Sigma)$, and proceed as in~(1).
  \item If $x$ is an interior leaf vertex then $f$ is constant in a neighborhood of $x$; choosing a small enough neighborhood $U$ with $\bar\Gamma$-rational endpoints, we have $f = g\circ h$ for $h\colon U\to\Gamma\subset\R$ and $g\colon\R\to\R$ both constant.
  \item Suppose that $x$ is a boundary vertex with outgoing edges $e_1,\ldots,e_r$.  Shrinking the edges if necessary (with edge lengths still in $\bar\Gamma$), we can choose $U = e_1\cup\cdots\cup e_r$ such that $\del U = \{x,e_1^+,\ldots,e_r^+\}$.  For $i=1,\ldots,r$ define $h_i\colon U\to\R$ to be $h_i = t_{e_i}\inv\colon e_i\to\R$ on $e_i$ and $0$ on $U\setminus e_i$.  Each $h_i$ is a $(\Z,\Gamma)$-harmonic function on $(U,\del U)$; set $h = (h_1,\ldots,h_r)\colon U\to\R^r$.  Let $g_i\colon\R\to\R$ be a smooth function extending $f\circ t_{e_i}$ (Lemma~\ref{lem:extend.smooth.1d}), and set $g(x_1,\ldots,x_r) = \sum_{i=1}^r g_i(x_i)$.  Then $f = g\circ h$.
  \item Finally, suppose that $x$ is an interior vertex of valency at least $3$, with outgoing edges $e_0,e_1,\ldots,e_r$.  Shrinking the edges if necessary (with edge lengths still in $\bar\Gamma$), we can choose $U = e_0\cup e_1\cup\cdots\cup e_r$ such that $\del U = \{e_0^+,e_1^+,\ldots,e_r^+\}$.  For $1\leq i\leq r$, define a $(\Z,\Gamma)$-harmonic function $h_i\colon U\to\R$ to be $t_{e_i}\inv$ on $e_i$, $-t_{e_0}\inv$ on $e_0$, and $0$ on $e_j$ for $j\notin\{0,i\}$.  Let $h = (h_1,\ldots,h_r)\colon U\to\R^r$.  For $i=1,\ldots,r$, this function maps $e_i$ to the interval $[0,\ell(e_i)]$ in the $i$th coordinate axis, and it maps $e_0$ to the line segment from $0$ to $-\ell(e_0)\,(1,1,\ldots,1)$.  See Figure~\ref{fig:image.of.h}.  Noting that $h$ is a homeomorphism from $U$ onto its image, we define $g\colon h(U)\to\R$ by $g = f\circ h\inv$.  This function is smooth on each edge $h(e_i)$, and Condition~(3) of Definition~\ref{def:smooth.on.graph} implies that the sum of the directional derivatives of $g$ at $0$ in the directions of the edges, is equal to $0$.  By Lemma~\ref{lem:extend.smooth}, we can extend $g$ to a smooth function on $\R^r$; then $f = g\circ h$.
  \end{enumerate}

  \begin{figure}[ht]\label{fig:image.of.h}
    \centering
    \begin{tikzpicture}[every node/.style={inner sep=1pt}]
      \point at (0, 0);
      \draw[thick] (0, 0) to["$h(e_1)$"] (2.2, 0);
      \draw[thick] (0, 0) to["$h(e_2)$"] (0, 2);
      \draw[thick] (0, 0) to["$h(e_0)$"] (-1.3, -1.3);
    \end{tikzpicture}
    \caption{The image of $h$.}
  \end{figure}
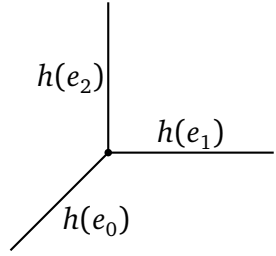

  Note that we only made reference to the particular form $\omega$ in case~(3); the same will be true for forms in $\cA^{p,q}(\Sigma,\del\Sigma)$ for all $p,q\in\{0,1\}$.

\subsubsection{Proof of Proposition~\ref{prop:local.pullbacks.graphs} for $\cA^{1,0}$ and $\cA^{0,1}$}
\label{sec:pullback.A10}
  Let $\omega = (f_e\,\d' t_e)\in\cA^{1,0}(\Sigma,\del\Sigma)$ and let $x\in\Sigma$.  Cases~(1)--(4) of~\secref{sec:pullback.A00} are handled much like the $\cA^{0,0}$ situation, so we assume that we are in case~(5).  We use the same notations as in case~(5) of~\secref{sec:pullback.A00}, and for brevity we set $f_i = f_{e_i}$.  Define $g_1\colon h(U)\to\R$ by
  \[
    g_1(x_1,\ldots,x_r) =
    \begin{cases}
      -f_0\circ h\inv(x_1,\ldots,x_r) + f_0(v) + f_1(v) & \text{on } h(e_0) \\
      f_1\circ h_1\inv(x_1) & \text{on } h(e_1) \\
      \bigl(\frac{\d f_0}{\d t_{e_0}}(v) - \frac{\d f_1}{\d t_{e_1}}(v)\bigr)\,x_2 + f_1(v) & \text{on } h(e_2) \\
      f_1(v) & \text{on } h(e_i),\; i > 2.
    \end{cases}
  \]
  Informally, this restricts to $-f_0$ on $h(e_0)$ up to an additive constant; it restricts to $f_1$ on $h(e_1)$; it is linear on $h(e_2)$ with slope designed to satisfy the balancing condition of Lemma~\ref{lem:extend.smooth}; and it is constant on other edges.  By Lemma~\ref{lem:extend.smooth}, $g_1$ extends to a smooth function on $\R^n$, which we also denote by $g_1$.

  For $i > 1$ we define $g_i\colon h(U)\to\R$ as follows:
  \[
    g_i(x_1,\ldots,x_r) =
    \begin{cases}
      -\frac{\d f_i}{\d t_{e_i}}(v)\,x_1 + f_i(v) & \text{on } h(e_1) \\
      f_i\circ h_i\inv(x_i) & \text{on } h(e_i) \\
      f_i(v) & \text{elsewhere.}
    \end{cases}
  \]
  As above, $g_i$ extends to a smooth function on $\R^n$, which we again denote by $g_i$.

  Define
  \[ \eta = \sum_{i=1}^r g_i\,\d'x_i \in \cA^{1,0}(\R^n). \]
  The pullback $h^*\d'x_i$ is equal to $\d't_{e_i}$ on $e_i$, is equal to $-\d't_{e_0}$ on $e_0$, and is zero on all other edges.  It follows that $h^*\eta$ is equal to $f_i\d't_{e_i}$ on $e_i$ for $i\geq 1$, and is equal to
  \[ \bigl( f_0 - f_0(v) - f_1(v) - \cdots - f_r(v) \bigr)\,\d t_{e_0} = f_0\,\d t_{e_0} \]
  on $e_0$.  Therefore, we have $h^*\eta = \omega$.

  The proof for $\cA^{0,1}$ is identical.

\subsubsection{Proof of Proposition~\ref{prop:local.pullbacks.graphs} for $\cA^{1,1}$}
\label{sec:pullback.A11}
  Let $\omega\in\cA^{1,1}(\Sigma,\del\Sigma)$ and let $x\in\Sigma$.  Cases~(1)--(4) of~\secref{sec:pullback.A00} are handled much like the $\cA^{0,0}$ situation, so we assume that we are in case~(5).  We use the same notations as in case~(5) of~\secref{sec:pullback.A00}.  By linearity, we may assume that $\omega = (f_{e_i}\d't_{e_i}\d''t_{e_i})$, where $f_{e_i} = 0$ unless $i=1$; we set $f = f_{e_1}$.  Define $g_1\colon h(U)\to\R$ by
  \[g_1(x_1,\ldots,x_r) =
    \begin{cases}
      f\circ h_1\inv(x_1) & \text{on } h(e_1) \\
      -\frac{\d f}{\d t_{e_1}}(v)\,x_2 + f(v) & \text{on } h(e_2) \\
      f(v) & \text{elsewhere}.
    \end{cases}
  \]
  Informally, this is the function that restricts to $f$ on $h(e_1)$, is linear on $h(e_2)$ with slope designed to satisfy the balancing condition of Lemma~\ref{lem:extend.smooth}, and is constant on other edges.  By Lemma~\ref{lem:extend.smooth}, $g_1$ extends to a smooth function on $\R^n$, which we also denote by $g_1$.  Let $g_2\colon\R^n\to\R$ be the constant function $-f(v)$, and define
  \[\eta = g_1\d'x_1\wedge\d''x_1 + g_2\d'x_1\wedge\d''x_2 \in\cA^{1,1}(\R^n). \]
  The pullback $h^*\d'x_1\wedge\d''x_1$ is equal to $\d't_{e_1}\d''t_{e_1}$ on $e_1$, is equal to $\d't_{e_0}\d''t_{e_0}$ on $e_0$, and is zero on all other edges; the pullback $h^*\d'x_1\wedge\d''x_2$ is equal to $\d't_{e_0}\d''t_{e_0}$ on $e_0$, and is zero on all other edges.  It follows that $h^*\eta$ is equal to $f\d't_{e_1}\d''t_{e_1}$ on $e_1$, is $(f(v)-f(v))\d't_{e_0}\d''t_{e_0} = 0$ on $e_0$, and is zero on all other edges as well.  Therefore we have $h^*\eta = \omega$.

  This concludes the proof of Proposition~\ref{prop:local.pullbacks.graphs}.\qed

\subsection{Integration of Pullbacks}\label{sec:integrate-pullback}
In this subsection, we prove that integration of Lagerberg forms and integration in the sense of~\secref{sec:integration} are compatible under pullbacks.  In order to do this, we need to define the image of a graph under a $\Z$-harmonic tropicalization $h\colon\Sigma\to\R^n$ as a tropical cycle, i.e., we need a notion of \emph{multiplicities}.

Let $h = (h_1,\ldots,h_n)\colon\Sigma\to\R^n$ be a $\Z$-harmonic tropicalization.
Note that $h(\Sigma)$ is a rational polytopal complex of dimension (at most) one.  Choose a subdivision of $h(\Sigma)$ and a subdivision of $\Sigma$ such that the inverse image of each edge of $h(\Sigma)$ is a union of edges of $\Sigma$, and the image of each edge of $\Sigma$ is an edge or a vertex of $h(\Sigma)$.

\begin{defn}\label{def:trop.graph.mult}
  The \defi{multiplicity} of an edge $e\subset\Sigma$ is
  \[ m_e = w(e)\gcd(a_1,\ldots,a_n), \]
  where $a_i$ is the slope of $h_i$ on $e$; we set $m_e=0$ if $h$ is constant on $e$.  We define the \emph{multiplicity} $m_\sigma$ of an edge $\sigma\subset h(\Sigma)$ as $\sum_{e\mapsto\sigma} m_e$.

This definition of multiplicity makes $h(\Sigma)$ into a weighted polytopal complex, which we call $\Trop_h(\Sigma)$, or $\Trop(\Sigma)$ if there is no confusion.
\end{defn}

\begin{rem}\label{rem:Trop.unweighting}
  Let $h\colon\Sigma\to\R^n$ be a $\Z$-harmonic tropicalization, and let $\nu\colon\Sigma_0\to\Sigma$ be the unweighting.  Let $e$ be an edge of $\Sigma$, and let $e_0$ be the associated edge in $\Sigma_0$.  Then $h_0 = h\circ\nu\colon\Sigma_0\to\R^n$ is a $\Z$-harmonic tropicalization, and $m_{e_0} = m_e$, so $\Trop(\Sigma) = \Trop(\Sigma_0)$ as weighted polytopal complexes.
\end{rem}

\begin{rem}\label{rem:trop.and.degree}
  More generally, let $\phi\colon\Sigma'\to\Sigma$ be a harmonic morphism of weighted metric graphs with boundary.  Suppose that $\phi$ has a well-defined positive degree $d(\phi)>0$ in the sense of Definition~\ref{def:harmonic.degree}, and that the expansion factor $d_{e'}(\phi)$ is an integer for each edge $e'\subset\Sigma'$.  Let $h\colon\Sigma\to\R^n$ be a $\Z$-harmonic tropicalization, and let $h' = h\circ\phi\colon\Sigma'\to\R^n$.  Then $h'$ is a $\Z$-harmonic tropicalization, and $\Trop_{h'}(\Sigma') = d(\phi)\Trop_h(\Sigma)$ as weighted polytopal complexes.
\end{rem}

The following Proposition is proved in the same way as~\cite[Theorem~5.14]{baker_payne_rabinoff16:tropical_curves}; alternatively, it follows from the proof of Proposition~\ref{prop:compat.integration}(2) below.  It implies that the weights on $h(\Sigma)$ are insensitive to further subdivision.

\begin{prop}\label{prop:trop.graph.balanced}
  With the notation in Definition~\ref{def:trop.graph.mult}, the weighted polytopal complex $\Trop(\Sigma)$ is balanced outside the image of $\del\Sigma$.
\end{prop}

Now we prove compatibility of the integration theories.  We use integration and boundary integrals for weighted polytopal complexes from~\cite{gubler16:forms_currents}.

\begin{prop}\label{prop:compat.integration}
  Let $h\colon\Sigma\to\R^n$ be a $\Z$-harmonic tropicalization and let $U\subset\R^n$ be an open neighborhood of $h(\Sigma)$.
  \begin{enumerate}
  \item For $\eta\in\cA^{1,1}(U)$ we have
    \[ \int_{\Trop(\Sigma)}\eta = \int_\Sigma h^*\eta. \]
  \item For $\eta\in\cA^{1,0}(U)$ we have
    \[ \int_{\del\Trop(\Sigma)}\eta = \int_{\del\Sigma} h^*\eta, \]
    and likewise for $\eta\in\cA^{0,1}(U)$.
  \end{enumerate}
\end{prop}

\begin{proof}
  By Lemma~\ref{lem:integral.unweighting} and Remark~\ref{rem:Trop.unweighting}, we may replace $\Sigma_0$ by its unweighting to assume that the edge weights are trivial.

  We begin with~(1).  We assume that we have subdivided $\Sigma$ and $h(\Sigma)$ as in Definition~\ref{def:trop.graph.mult}.  By linearity, we may assume $\eta = g\,\d' x_i\wedge\d'' x_j$, so $\omega = h^*\eta = (f_e\,\d't_e\d''t_e)$ for $f_e = \frac{\d h_i}{\d t_e}\frac{\d h_j}{\d t_e}\,g\circ h$.  Fix an edge $e$ of $\Sigma$, and let $a_i = \d h_i/\d e$.  If $h(e)$ is a point in $\R^n$ then $\int_\sigma\omega=0$; likewise, $\d h_i/\d t_e = 0$, so $\int_e\omega=0$.  Hence we may suppose that $\sigma = h(e)$ is a line segment in $\R^n$.  Let $\bA_\sigma$ be the affine span of $\sigma$, and let $\bL_\sigma$ be the underlying real vector space.  The integral structure on $\bA_\sigma$ is induced by $\Z^n\cap\bL_\sigma$, so we can identify $\bA_\sigma$ with $\R$ (as $\Z$-affine spaces) by $p_e(x) = h\circ t_e(x/m_e)$, where $m_e = \gcd(a_1,\ldots,a_n)$ is the edge multiplicity of Definition~\ref{def:trop.graph.mult}.  We compute
  \[\begin{split} \int_{\sigma}\eta
      &= \int_{[0,\ell(e)m_e]}g\circ p_e(x)\,\d'(x_i\circ p_e(x))\wedge\d''(x_j\circ p_e(x))\\
      &= \int_0^{\ell(e)m_e}g\circ h\circ t_e(x/m_e)\,\frac{a_ia_j}{m_e^2}\,\d x \\
      &= \frac 1{m_e}\int_0^{\ell(e)}a_ia_j\,g\circ h\circ t_e(x)\,\d x
      = \frac 1{m_e}\int_0^{\ell(e)}f_e\circ t_e(x)\,\d x.
    \end{split}\]
  Summing over all (unoriented) edges mapping to $\sigma$ gives
  \[ \sum_{e\mapsto\sigma}\int_0^{\ell(e)}f_e\circ t_e(x)\,\d x
    = \sum_{e\mapsto\sigma} m_e\int_\sigma\eta
    = m_\sigma\int_\sigma\eta. \]
Summing over all edges of $\Sigma$ then yields $\int_{\Trop(\Sigma)}\eta = \int_\Sigma\omega$.

  We will prove assertion~(2) for $\eta\in\cA^{1,0}(U)$, as the proof for $\eta\in\cA^{0,1}(U)$ is the same.  Again by linearity we may assume $\eta = g\,\d' x_i$, so $\omega = h^*\eta = (f_e\,\d't_e)$ for $f_e = \frac{\d h_i}{\d t_e}\,g\circ h$.  Let $e$ be an edge of $\Sigma$, and assume $\sigma = h(e)$ is a line segment, as above.  Let $a_i = \d h_i/\d t_e$, so $(b_1,\ldots,b_n) = \frac 1{m_e}(a_1,\ldots,a_n)\in\Z^n$ is the primitive lattice vector in the direction $h(e^+)-h(e^-)$.  Unwrapping the definition in~\cite[2.8]{gubler16:forms_currents}, we have $\int_{\del\sigma}\eta = b_i g(h(e^-)) - b_i g(h(e^+))$,%
  \footnote{For $\eta\in\cA^{0,1}(U)$ we must exchange $e^+$ with $e^-$.  Note that there is a sign mistake in~\cite{gubler16:forms_currents}; see the erratum in~\cite{gubler13:forms_currents_analytif_algebraic_variety}.}
  so
  \[ m_e\int_{\del\sigma}\eta = a_i\bigl( g\circ h(e^-) - g\circ h(e^+) \bigr)
    = f_e(e^-) - f_e(e^+).
  \]
  Summing over all edges and using the fact that $f_e=0$ if $h$ is constant on $e$ gives
  \[ \int_{\del\Trop(\Sigma)}\eta
    = \sum_\sigma\sum_{e\mapsto\sigma}m_e\int_{\del\sigma}\eta
    = \sum_{v\in V(\Sigma)}\sum_{v=e^-} f_e(v) = \int_{\del\Sigma}\omega, \]
  where the last equality holds by~\eqref{eq:bdy.int.2}.
\end{proof}


\section{Curves and Skeletons}
\label{Sec:curves-skeletons}

In this section we fix our notions regarding non-Archimedean analytic curves and their structure theory via skeletons.  Importantly, we prove several functoriality properties of skeletons that will be used in the sequel.  Our main reference is Thuillier's thesis~\cite[Section~2]{thuillier05:thesis}, where most of the theory is worked out.  See also~\cite{baker_payne_rabinoff13:analytic_curves,ducros14:structur_des_courbes_analytiq}.

\subsection{Notation}
For the rest of the paper we fix a field $k$ that is complete with respect to a nontrivial, non-Archimedean valuation; we call such a field a \defi{non-Archimedean field}.

\medskip\noindent
We will use the following notation concerning non-Archimedean fields. \\[1mm]
\null\begin{tabular}{rl}
  \hbox to 1.5cm{\hfil$k$} & A non-Archimedean field. \\
  $\val$ & $\colon k\to\R\cup\{\infty\}$, the valuation on $k$. \\
  $|\scdot|$ & $= \exp(-\val(\scdot))$, the associated absolute value. \\
  $k^\circ$ & The valuation ring of $k$. \\
  $k^{\circ\circ}$ & The maximal ideal in $k^\circ$. \\
  $\td k$ & $= k^\circ / k^{\circ\circ}$, the residue field of $k$. \\
  $\Gamma$ & $= \val(k^\times)$, the value group of $k$. \\
\end{tabular}

\medskip
In this paper, all analytic spaces are strictly $k$-analytic, good Berkovich spaces. 
This is due to the fact that we are dealing with analytic curves, which are always good~\cite[Proposition~3.3.4]{ducros14:structur_des_courbes_analytiq}.

\medskip\noindent
We will use the following notation concerning analytic spaces.  Let ${\mathscr A}$ be a (strictly) $k$-affinoid algebra and let $X$ be an analytic space.\\[1mm]
\null\begin{tabular}{rl}
  \hbox to 1.5cm{\hfil${\mathscr A}^\circ$} & The ring of power-bounded elements in ${\mathscr A}$. \\
  ${\mathscr A}^{\circ\circ}$ & The ring of topologically nilpotent elements in ${\mathscr A}^\circ$. \\
  $\td {\mathscr A}$ & $= {\mathscr A}^\circ/{\mathscr A}^{\circ\circ}$, the canonical reduction of ${\mathscr A}$. \\
  $|\scdot|_{\sup}$ & The supremum semi-norm on ${\mathscr A}$. \\
  $\sM({\mathscr A})$ & The Berkovich spectrum of ${\mathscr A}$. \\
  $\red_{\sM({\mathscr A})}$ & $\colon \sM({\mathscr A}) \to \Spec(\td {\mathscr A})$, the reduction map. \\
  $\sH(x)$ & The completed residue field at a point $x\in X$. \\
\end{tabular}

\medskip
We will make extensive use of admissible formal $k^\circ$-schemes.  We refer the reader to~\cite{bosch_lutkeboh93:formal_rigid_geometry_I} for the basics of the theory.

\medskip\noindent
We will use the following notation concerning admissible formal schemes.  Let $\fX$ be a quasi-compact admissible formal $k^\circ$-scheme.\\[1mm]
\null\begin{tabular}{rl}
  \hbox to 1.5cm{\hfil$\fX_s$} & The special fiber of $\fX$, a $\td k$-scheme of finite type. \\
  $\fX_\eta$ & The generic fiber of $\fX$, a compact analytic space. \\
  $\red_\fX$ & $\colon\fX_\eta\to\fX_s$, the reduction map. \\
  $\rS_0(\fX)$ & $= \{x\in\fX_\eta\colon\red_\fX(x)\text{ is a generic point of } \fX_s\}$. \\
\end{tabular}

\medskip
The generic fiber of $\Spf(A)$ is by definition the Berkovich spectrum of the $k$-affinoid algebra ${\mathscr A}=A\tensor_{k^\circ}k$; this construction globalizes to give the generic fiber of $\fX$ (see~\cite[Section~2]{gubler_rabinoff_werner:tropical_skeletons}).

\begin{defn}
  A \defi{formal model} of a $k$-analytic space $X$ is an admissible formal $k^\circ$-scheme $\fX$ equipped with an isomorphism $\fX_\eta\cong X$.
\end{defn}

Following Thuillier~\cite[D\'efinition~2.1.5]{thuillier05:thesis}, we call a quasi-compact admissible formal scheme $\fX$ \defi{maximal} if, for every formal affine open $\fU = \Spf(A)\subset\fX$, we have $A = {\mathscr A}^\circ$ where $\mathscr A \coloneqq (A\tensor_{k^\circ} k)$ is the $k$-affinoid algebra associated to the admissible $k$-algebra $A$.  In this case, by~\cite[Proposition~2.4.4]{berkovic90:analytic_geometry}, for every generic point $\td x\in\fX_s$, there exists a unique point $x\in\fX_\eta$ such that $\red_\fX(x) = \td x$.  In other words, $\red_\fX$ sets up a bijection between $\rS_0(\fX)$ and the set of generic points of $\fX_s$.

The following lemma is a version of~\cite[Proposition~2.1.1]{thuillier05:thesis}; the proof is extracted from~\cite[Proposition~3.13]{baker_payne_rabinoff16:tropical_curves}.

\begin{lem}\label{lem:formal.analytic.variety}
  Let $\fX$ be an admissible formal scheme with reduced special fiber.  Then $\fX$ is maximal and we have $|\mathscr A|_{\sup}=|k|$.  Moreover, for any $x\in\rS_0(\fX)$ we have $|\sH(x)| = |k|$.
\end{lem}

\begin{proof}
  Let $\fU = \Spf(A)$ be a formal affine open in $\fX$.  Since $A$ is of topologically finite type, there is a surjection $\alpha\colon T^\circ\surject A$ for some Tate algebra $T = k\angles{T_1,\ldots,T_n}$.  The $T^\circ$-ideal $T^{\circ\circ} + \ker(\alpha)$ is a reduced ideal because it is the kernel of the composition $T^\circ\surject A\surject A/k^{\circ\circ}A$, and $A/k^{\circ\circ}A = A\tensor_{k^\circ}\td k$ is reduced by hypothesis. It follows from~\cite[Proposition~6.4.3/3,4]{bosch_guntzer_remmert84:non_archimed_analysis} that $\fX$ is maximal and that $|\mathscr A|_{\sup}=|k|$. Applying the latter in a neighbourhood of the generic point $\red_\fX(x)$ in $\fX_s$, we get the last claim. 
\end{proof}

\subsection{Strictly Semistable Curves}\label{sec:semistability}
We will only consider \emph{compact} curves in this paper.

\begin{defn}\label{def:curve}
  A \defi{curve} is a compact, rig-smooth, strictly $k$-analytic space of pure dimension $1$.
\end{defn}

Recall that rig-smooth means that the sheaf of K\"ahler differentials is locally free of rank $1$ with respect to the Grothendieck topology. 
Note that curves are allowed to be disconnected.  We will consider formal models of our curves modeled on a formal annulus.

\begin{notn}
  For $a\in k^\circ\setminus\{0\}$, the \defi{formal annulus of modulus $|a|$} is
  \[ \fS(a) = \Spf\bigl(k^\circ\{T_0,T_1\}/(T_0T_1-a)\bigr). \]
  The generic fiber $\fS(a)_\eta$ is the \defi{annulus of modulus |a|}; it can be identified (via $t\mapsto T_0$ or $t\mapsto T_1$) with the affinoid domain
  \[ S(a) = \bigl\{ x\in\bA^{1,\an} = \Spec(k[t])^\an\colon |a|\leq |t(x)|\leq 1 \bigr\}. \]
\end{notn}

The special fiber $\fS(a)_s$ is isomorphic to $\bG_{\rm m,\td k}$ if $|a|=1$, and is otherwise isomorphic to the union of the coordinate axes in $\bA^2_{\td k}$.

\begin{defn} \label{def:strictly semistable over valuation ring}
  A \defi{strictly semistable $k^\circ$-curve} is a quasi-compact admissible formal $k^\circ$-scheme $\fX$ of pure dimension $1$ that admits a covering by formal open subsets $\fU$ admitting an \'etale morphism to a formal annulus $\fU\to\fS(a_\fU)$.
\end{defn}

\begin{rem}\

  \begin{enumerate}
  \item The generic fiber of a strictly semistable curve is a curve in the sense of Definition~\ref{def:curve}.

  \item A strictly semistable $k^\circ$-curve has reduced special fiber, hence is maximal.

  \item If $k'/k$ is a non-Archimedean field extension then $\fX' = \fX\hat\tensor_{k^\circ}k'^\circ$ is a strictly semistable $k'^\circ$-curve. Note that the special fiber of $\fX'$ is canonically isomorphic to $\fX_s\tensor_{\td k}\td k'$, \emph{not} to $\fX'\tensor_{k^\circ}\td k$.

  \item A $k^\circ$-curve $\fX$ is strictly semistable if and only if $\fX_\eta$ is rig-smooth and $\fX_s$ has at worst ordinary double point singularities and smooth irreducible components.  See~\cite[Remarque~2.2.9]{thuillier05:thesis}.

  \item The irreducible components of $\fX_s$ need not be geometrically connected, and the singularities need not be $\td k$-rational.

  \item A strictly semistable model of $X$ exists after changing base to a finite, separable field extension $k'/k$ by~\cite[6.4.3]{ducros14:structur_des_courbes_analytiq}.
  \end{enumerate}
\end{rem}

\subsection{Skeletons} \label{sec:skeletons}

We say that a strictly affinoid curve $U$ over $k$ is \emph{potentially isomorphic to the unit disc} if there is a finite separable extension $k'/k$ such that the base change of $U$ to $k'$ is isomorphic to a disjoint union of unit discs over $k'$. 

In the following, we consider a strictly semistable model $\fX$ of a curve $X$. We associate to $\fX$  a \defi{skeleton} $\Sigma_\fX\subset X$ along with a retraction $\tau_\fX\colon X\to\Sigma_\fX$: see~\cite[Th\'eor\`eme~2.2.10]{thuillier05:thesis},~\cite[4.3]{berkovic90:analytic_geometry},  and~\cite[Section~4]{baker_payne_rabinoff13:analytic_curves}.  The complement of the underlying set of $\Sigma_\fX$ can be characterized as the set of all points of $X$ admitting a neighborhood that is potentially isomorphic to the unit disc and is disjoint from $\rS_0(\fX)$~\cite[Corollaire~2.2.12]{thuillier05:thesis}.  The retraction map is a strong proper deformation retraction, hence a homotopy equivalence~\cite[Theorem~5.2]{berkovic99:locally_contractible_I}.    The skeleton naturally has the structure of a weighted metric graph with boundary, in the sense of Section~\ref{Sec:forms-graphs}, defined as follows.
\begin{itemize}
\item The vertex set of $\Sigma_\fX$ is equal to $\rS_0(\fX)$.
\item The boundary of $\Sigma_\fX$ is the set of vertices reducing to generic points of $\fX_s$ with non-proper closure.  This concides with the boundary $\del X$ in the sense of~\cite[3.1]{berkovic90:analytic_geometry}.  See~\cite[2.1.2]{thuillier05:thesis}.
\item For $\rho\in\R$ define $\eta_\rho\colon k[t^{\pm1}]\to\R\cup\{0\}$ by
  \begin{equation}\label{eq:eta_rho}
    \eta_\rho\left( \sum a_n t^n \right) = \max\bigl\{ |a_n|\rho^n \bigr\}.
  \end{equation}
  This defines an embedding $\R\inject\bG_{\rm m}^{\an}$ which restricts to a continuous function $[|a|,1]\to S(a)$ whose image is defined to be the skeleton $\Sigma_{\fS(a)}$ of $\fS(a)$.  If $|a|<1$ then $\Sigma_{\fS(a)}$ consists of a single edge $e$; we make $\Sigma_{\fS(a)}$ into a metric graph using the parameterization $t_e\colon[0,\val(a)]\to e$ defined by $t_e(r) = \eta_{\exp(-r)}$.  The retraction $\tau_{\fS(a)}$ is defined by $\tau_{\fS(a)}(x) = t_e(-\log|t(x)|) = \eta_{|t(x)|}$.
\item Each singular point $\td x\in\fX_s$ has a formal neighborhood $\fU\subset\fX$ admitting an \'etale morphism $p\colon\fU\to\fS(a)$ for some $a\in k^{\circ\circ}\setminus\{0\}$.  By shrinking $\fU$, we may assume that $\fU_s$ consists of two components meeting at $\td x$, then $p_\eta$ defines a homeomorphism $p_\eta\inv(\Sigma_{\fS(a)})\isom\Sigma_{\fS(a)}$. Such a $\fU$ is called a \emph{building block} of $\fX$. We associate to $\td x$  the edge $e_{\td x} = p_\eta\inv(\Sigma_{\fS(a)})\subset\Sigma_\fX$ connecting the two vertices of $\Sigma_\fX$ lying above the vertices of $\Sigma_{\fS(a)}$.  We parameterize $e_{\td x}$ by $t_{e_{\td x}} = p\inv\circ t_e$, where $e$ is the unique edge of $\Sigma_{\fS(a)}$; in this way, $\Sigma_\fX$ becomes a metric graph with $\ell(e_{\td x}) = \val(a)$.  The retraction $\tau_\fX$ is defined on $\fU_\eta$ by $\tau_\fX|_{\fU_\eta} = p\inv\circ\tau_{\fS(a)}\circ p_\eta$.
\item The weight of an edge $e_{\td x}$ is defined to be the residue field degree $[\kappa(\td x):\td k]$.
\end{itemize}
It follows from the above description that the graph underlying $\Sigma_\fX$ is the incidence graph of $\fX_s$, and that the length of each edge of $\Sigma_\fX$ is contained in $\Gamma$.

The following fact given in~\cite[Proposition~2.2.17]{thuillier05:thesis} shows that the parameterization of an edge does not depend on the choice of \'etale morphism to a formal annulus.  It is a very special case of~\cite[Theorem~5.3]{berkovic99:locally_contractible_I}. For convenience, we include a proof.

\begin{prop}\label{prop:length.well.defined}
  Let $\fU$ be a building block of $\fX$ at a singular point $\td x$, and consider the edge parameterization $t_e\colon[0,\val(a)]\to e$ defined above.   For  $u\in\Gamma(\fX_\eta,\sO)^\times$, there exist $r\in\Gamma$ and $n\in\Z$ such that
	\[ -\log|u(t_e(x))| = nx + r \]
	for all $x\in[0,\val(a)]$.  Conversely, for any $n\in\Z$ and $r\in\Gamma$, there exists $u\in\Gamma(\fX_\eta,\sO)^\times$ satisfying the above equation.
\end{prop}

\begin{proof}
  Let $x_0,x_1$ be the endpoints of $e$.  Since $|\sH(x_0)^\times| = |k^\times|$ by Lemma~\ref{lem:formal.analytic.variety}, we may multiply $u$ by a scalar in $k^\times$ to assume $u(x_0) = 1$.  We need to show that $-\log|u|$ is linear with integral slope on the interior of $e$.  Let $K$ be an algebraically closed non-Archimedean field containing $k$, let $\fX_K = \fX\hat\tensor_{k^\circ}K^\circ$ and $\fS(a)_K = \fS(a)\hat\tensor_{k^\circ}K^\circ$, and consider the commutative square
  \[ \xymatrix @R=.3in {
      {\fX_K} \ar[r]^(.45){p_K} \ar[d]_\pi &
      {\fS(a)_K} \ar[d] \\
      {\fX} \ar[r]_(.43)p & {\fS(a).}
    } \]
  Let $\td y$ be the singular point of $\fS(a)_s$, and let $\td y'$ be the singular point of $\fS(a)_{K,s}$.  Then $p_K\inv(\td y') = \pi_s\inv(\td x)$ is a finite set of singular points.  Choose $\td x'\in \pi_s\inv(\td x)$, and let $e'$ be the associated edge of $\Sigma_{\fX_K}$.  Then $\pi_\eta$ maps $e'$ isomorphically onto $e$, in that the  parameterizations of $e$ and $e'$ induced by the \'etale morphisms $p$ and $p_K$, respectively, are compatible with $\pi_\eta$.

  The morphism $\red_{\fX_K}\inv(\td x')\to\red_{\fS(a)_K}\inv(\td y')$ is an isomorphism by~\cite[Lemma~4.4]{berkovic99:locally_contractible_I}.  In particular, $U = \red_{\fX_K}\inv(\td x')$ is isomorphic to the open annulus of modulus $|a|$.  Thus $-\log|\pi_\eta^*f|$ has integral slope along $\Sigma_{\fX_K}\cap U$ by~\cite[Proposition~2.10]{baker_payne_rabinoff13:analytic_curves}, for instance.

  The final assertion is a consequence of~\cite[Lemme~2.2.1]{thuillier05:thesis}: one can choose $u$ to be the pullback of a unit on $\fS(a)_\eta$ under $p_\eta\colon\fX_\eta\to\fS(a)_\eta$.
\end{proof}

\begin{rem} \label{rem:stratum face correspondence}
	By construction, there is a bijective correspondence between edges of $\Gamma_\fX$ and singularities of the special fiber $\fX_s$. Explicitly, the edge $e$ with interior $\mathring{e}$ and the corresponding singularity $\td x$ are related by
	\begin{equation} \label{eq:stratum face for edges}
	\{\td x\}=\red_\fX(\tau_\fX^{-1}(\mathring{e})) \quad \text{and} \quad \mathring{e}= \tau_\fX(\red_\fX^{-1}(\{\td{x}\})).
	\end{equation}
    Similarly, the reduction $\red_\fX$ restricts to a bijection between the vertices $S_0(\fX)$ of $\Gamma_\fX$ and the generic points of the irreducible components of $\fX_s$. If $x \in S_0(\fX)$ and if $R$ is the corresponding irreducible component of $\fX_s$ with all singular points from $\fX_s$ removed, then we have
    \begin{equation} \label{eq:stratum face correspondence for vertices}
    R = \red_\fX(\tau_\fX^{-1}(\{x\})) \quad \text{and} \quad \{x\} = \tau_\fX(\red_\fX^{-1}(R)).
    \end{equation} 
    For a higher dimensional generalization, we refer to~\cite[Proposition~2.8]{vilsmeier21:monge_ampere}. 
\end{rem}

\begin{rem}\label{rem:non-rational-subgraph}
  Let $U\subset\Sigma_\fX$ be a subgraph with $\sqrt\Gamma$-rational vertices (Definition~\ref{def:gamma.ratl.pts}).  Then $V = \tau_\fX\inv(U)$ is a compact strictly $k$-analytic domain in $X$ with boundary $\del V = \del U$, hence is again a curve.  Note however that $U$ is not necessarily a skeleton of $V$, as the lengths of the edges of $U$ are only assumed to be contained in $\sqrt\Gamma$.  After a finite, separable extension of $k$, the curve $V$ does admit a strictly semistable model with associated skeleton $U$.
\end{rem}

\subsection{Functoriality With Respect to Galois Actions}
First we prove functoriality of the skeleton and the retraction with respect to a Galois action.

\begin{lem}\label{lem:skeleton.galois}
  Let $X$ be a $k$-curve, let $k'/k$ be a finite Galois extension with Galois group $G$, let $X' = X\tensor_k k'$ with structure morphism $\pi\colon X'\to X$, and let $\fX'$ be a strictly semistable model of $X'$.  Suppose that $\rS_0(\fX') = \pi\inv(\pi(\rS_0(\fX')))$.  Then $G$ stabilizes $\Sigma_{\fX'}$, and the resulting action on $\Sigma_{\fX'}$ is by harmonic maps preserving lengths, weights, and $\del\Sigma_{\fX'}$.  Furthermore, we have $\tau_{\fX'}\circ\sigma = \sigma\circ\tau_{\fX'}$ for all $\sigma\in G$.
\end{lem}

\begin{proof}
  The $G$-action on $X'$ extends to a $G$-action on $\fX'$ by~\cite[Proposition~2.1.15]{thuillier05:thesis} (regarding $\fX'$ as a maximal admissible formal $k^\circ$-scheme). Let $p\colon\fU'\to\fS(a')$ be the  \'etale morphism for a building block at a singularity $y$ of $\fX_s'$ as in~\secref{sec:skeletons}.  We obtain an \'etale morphism $p_\sigma\colon\sigma\inv(\fU')\to\fS(\sigma(a')) = \fS(a')\tensor_{k',\sigma}k'$ which fits into the Cartesian square
  \[ \xymatrix @R=.3in {
      {\sigma\inv(\fU')} \ar[r] \ar[d]_{p_\sigma} & {\fU'} \ar[d]^p \\
      {\fS(\sigma(a'))} \ar[r]_(.55){f} & {\fS(a')}
    } \]
  in which the horizontal arrows lift the automorphism $\sigma\colon k'^\circ\isom k'^\circ$.  It is immediate from the definition of the skeleton of a formal annulus that  $f\colon\fS(\sigma(a'))\to\fS(a')$ sends $\Sigma_{\fS(\sigma(a'))}$ isomorphically onto $\Sigma_{\fS(a')}$, from which it follows that $\sigma(\Sigma_{\sigma\inv(\fU')}) = \Sigma_{\fU'}$.  This proves that $\sigma$ stabilizes $\Sigma_{\fX'}$ and acts by a graph automorphism preserving lengths.  Weights are preserved because $\sigma$ does not change the residue degrees of singular points of $\fX'_s$ over $\td k$, and $\sigma$ stabilizes $\del X' = \del\Sigma_{\fX'}$ by~\cite[Proposition~3.1.3(ii)]{berkovic90:analytic_geometry} as applied to $X'\to X\to\sM(k)$, noting that $X'\to X$ is finite and hence boundaryless.  It follows from all these considerations that $G$ acts by harmonic maps.  The retraction $\tau_{\fX'}$ respects the \'etale morphisms $p$ and $p_\sigma$ by~\cite[Proposition~2.2.16]{thuillier05:thesis}, and one sees easily that $f\colon\fS(\sigma(a'))\to\fS(a')$ commutes with the retraction maps on both sides, which proves that $\tau_{\fX'}\circ\sigma = \sigma\circ\tau_{\fX'}$.
\end{proof}

\begin{rem}\label{rem:galois.stable.model}
Let $X$ be a $k$-curve and let $S$ be a finite subset of type $2$ points of $X$. Then there is a finite Galois extension $k'/k$ and a strictly semistable model $\fX'$ of $X' \coloneqq X \otimes_k K'$ such that the structure morphism $\pi:X' \to X$ satisfies 
$$\pi^{-1}(S)\subset \rS_0(\fX') = \pi\inv(\pi(\rS_0(\fX'))).$$
To see this, we follow the book of Ducros on analytic curves~\cite{ducros14:structur_des_courbes_analytiq}. It is shown in~\cite[Th\'eor\`eme~5.1.14]{ducros14:structur_des_courbes_analytiq} that $X$ has a finite triangulation such that the vertices  form a finite set $T$ of type 2 points containing $S$. By definition, this means that the connected components $C$ of $X \setminus T$ are such that the base change of $C$ to the completion of an algebraic closure  of $k$ is either a disk or an annulus. It is shown in~\cite[6.4]{ducros14:structur_des_courbes_analytiq} that there is a finite Galois extension $k'/k$ such that triangulation $T' \coloneqq \pi^{-1}(T)$ on $X'$ induces a semistable model $\fX'$ of $X'$ with $S_0(\fX')=T'$. Using a subdivision of the triangulation to omit loop edges, we may assume that $\fX'$ is strictly semistable. Since $\pi$ is surjective, this proves the claim.
\end{rem}

\subsection{Functoriality With Respect to Extension of Scalars}
Our next functoriality property is similar to~\cite[Proposition~2.2.21]{thuillier05:thesis}, although it requires a separate argument.

Recall that we defined harmonic maps of graphs in  Definition~\ref{def:harmonic.morphism}.

\begin{prop}\label{prop:skeleton.basechange}
  Let $k'/k$ be a non-Archimedean extension, let $\fX$ be a strictly semistable model of a curve $X$, and let $\fX' = \fX\hat\tensor_{k^\circ}k'^\circ$ and $X' = \fX'_\eta = X\hat\tensor_k k'$.  Let $\pi\colon\fX'\to\fX$ be the structure morphism.
  \begin{enumerate}
  \item The morphism $\pi_\eta\colon X'\to X$ maps $\Sigma_{\fX'}$ surjectively onto $\Sigma_\fX$, and the restriction of $\pi_\eta$ to $\Sigma_{\fX'}$ is a harmonic map of weighted metric graphs with boundary $\Sigma_\pi\colon\Sigma_{\fX'}\to\Sigma_\fX$.  We have $\Sigma_\pi\inv(\del\Sigma_\fX) = \del\Sigma_{\fX'}$, we have $d_{e'}(\Sigma_\pi) = 1$ for all edges $e'\subset\Sigma_{\fX'}$, and $d(\Sigma_\pi) = 1$ if $\Sigma_\fX$ has an edge.
  \item We have $\Sigma_\pi\circ\tau_{\fX'} = \tau_\fX\circ\pi_\eta$.
  \item If all irreducible components of $\fX_s$ are geometrically irreducible and all singular points of $\fX_s$ are $\td k$-rational, then $\Sigma_\pi\colon\Sigma_{\fX'}\to\Sigma_\fX$ is an isomorphism.
  \item If $\td k' = \td k$, then $\Sigma_\pi\colon\Sigma_{\fX'}\to\Sigma_\fX$ is an isomorphism.
  \item If $k'/k$ is a finite Galois extension with Galois group $G$, then $G$ acts on $\Sigma_{\fX'}$ 
  and $\Sigma_\pi$ identifies $\Sigma_\fX$ with the quotient graph $\Sigma_{\fX'}/G$ in the sense of Definition~\ref{def:quotient.graph}.
  \end{enumerate}
  \end{prop}

\begin{proof}
  It follows from the definition of the skeleton that $\pi_\eta(\Sigma_{\fX'})\subset \Sigma_\fX$. For brevity we write $\Sigma = \Sigma_\fX$, $\Sigma' = \Sigma_{\fX'}$, and $\phi = \Sigma_\pi\colon\Sigma'\to\Sigma$.  The morphism on special fibers $\pi_s\colon\fX'_s\to\fX_s$ identifies $\fX'_s$ with $\fX_s\tensor_{\td k}\td k'$.  It follows that the inverse image of an irreducible component of $\fX_s$ is a disjoint union of smooth curves in $\fX_s'$, that the inverse image of the smooth locus of $\fX_s$ is smooth, and that the inverse image of a singular point of $\fX_s$ is a finite set of singular points of $\fX_s'$.  Let $\td x'\in\fX_s'$ be a singular point with corresponding edge $e'\subset\Sigma'$, let $\td x = \pi_s(\td x')$, and let $e\subset\Sigma$ be the corresponding edge in $\Sigma$.  
  One shows as in the proof of Proposition~\ref{prop:length.well.defined} that $e'$ maps onto $e$ with expansion factor $d_{e'}(\phi) = 1$, see also~\cite[Proposition~2.2.21]{thuillier05:thesis}.  
  Recall that the points of $\del X = \del\Sigma$ are the vertices that reduce to generic points of non-proper irreducible components of $\fX_s$, and likewise for $\del\Sigma'$.  Since properness can be checked after extension of the ground field, it follows that $\phi\inv(\del\Sigma) = \del\Sigma'$.  In the situation of assertions~(3) and~(4), these considerations show that $\phi$ is an isomorphism, and in general they show that $\phi$ is a surjective map.

  We claim that $\phi$ is harmonic.  The proof involves some basic algebraic geometry of curves over non-algebraically closed fields, and is a mild generalization of the argument in~\cite[Proposition~2.2.21]{thuillier05:thesis}.  Let $C$ be a smooth, proper, connected $\td k$-curve, and let $C' = C\tensor_{\td k}\td k'$.  Since $C$ is geometrically reduced and proper over $\td k$, the field of global functions $k_1 = \Gamma(C,\sO_C)$ is a finite, separable extension of $\td k$, so $k_1\tensor_{\td k}\td k'$ is an \'etale $\td k'$-algebra which splits into a product of finite, separable extensions $k_i'$ of $\td k'$.  The curve $C$ is geometrically connected as a $k_1$-curve, so $C'$ is the disjoint union of the curves $C_i' = C\tensor_{k_1}k_i'$, and $k_i'$ is the field of global functions of $C_i'$ since $C_i'$ is geometrically connected as a $k_i'$-curve.  It follows that
  \begin{equation}\label{eq:total.degree.special.fiber}
    \sum_i [k_i':\td k'] = [k_1:\td k].
  \end{equation}
  Let $\td x\in C$ be a closed point.  The fiber of $C_i'\to C$ over $\td x$ is the spectrum of $\kappa(\td x)\tensor_{k_1}k_i'$, which is a $k_i'$-algebra of dimension $[\kappa(\td x):k_1]$.  Hence
  \begin{equation}\label{eq:local.degree.special.fiber}
    \dim_{\td k'}\bigl(\kappa(\td x)\tensor_{k_1}k_i'\bigr) = [\kappa(\td x):k_1][k_i':\td k']
  \end{equation}
  These are the facts we will use below.

  Let $v'\in\Sigma'$ be an interior vertex, let $v = \phi(v')$, and let $C'$ and $C$ be the corresponding (proper) irreducible components of $\fX'_s$ and $\fX_s$, respectively.  Let $k_1$ (resp.\ $k_1'$) be the field of global functions of $C$ (resp.\ $C'$).  Let $\td x\in C$ be a singular point of $\fX_s$.  Note that $\kappa(\td x)$ is a finite, separable extension of $\td k$, since $\fX_s$ admits an \'etale morphism to $\Spec(\td k[T_1,T_2]/(T_1T_2))$ in a neighborhood of $\td x$, with $\td x$ mapping to the singular point.  Hence $\kappa(\td x)$ is separable over $k_1$, so $\kappa(\td x)\tensor_{k_1}k_1'$ is an \'etale $k_1'$-algebra.  Equation~\eqref{eq:local.degree.special.fiber} then shows that
  \[ \sum_{\substack{\td x'\in C'\\\td x'\mapsto\td x}}[\kappa(\td x'):\td k'] = [\kappa(\td x):k_1]\,[k_1':\td k']. \]
  Let $e\subset\Sigma$ be the edge corresponding to $\td x$.  Recalling that all expansion factors are equal to one, the above equality yields
  \begin{equation}\label{eq:dv.phi.extend.scalars}
 d_{v'}(\phi) =  \sum_{\substack{e'^- = v'\\e'\mapsto e}}\frac{w(e')}{w(e)}d_{e'}(\phi)
    = \sum_{\substack{\td x'\in C'\\\td x'\mapsto\td x}}\frac{[\kappa(\td x'):\td k']}{[\kappa(\td x):\td k]}
    = \frac{[\kappa(\td x):k_1]}{[\kappa(\td x):\td k]}[k_1':\td k']
    = \frac{[k_1':\td k']}{[k_1:\td k]}.
  \end{equation}
  This quantity is independent of $\td x$ (and $e$), which proves harmonicity.  Summing over all $v'$ mapping to $v$ gives $d_e(\phi) = 1$ by~\eqref{eq:total.degree.special.fiber}.  This finishes the proof of~(1).

  Now we move on to assertion~(2).  The assertion is local on $\fX$, so we assume that there exists an \'etale morphism $p\colon\fX\to\fS(a)$ inducing an isomorphism on skeletons.  Let $\fS(a)' = \fS(a)\hat\tensor_{k^\circ}k'^\circ$, let $\pi'\colon\fS(a)'\to\fS(a)$ be the structure morphism, and let $p'\colon\fX'\to\fS(a)'$ be the induced \'etale morphism.  Then $p\circ\pi = \pi'\circ p'$ and $\pi'_\eta\circ\tau_{\fS(a)'} = \tau_{\fS(a)}\circ\pi'_\eta$ by construction, and $p'_\eta\circ\tau_{\fX'} = \tau_{\fS(a)'}\circ p'_\eta$ and $p_\eta\circ\tau_\fX = \tau_{\fS(a)}\circ p_\eta$ by~\cite[Proposition~2.2.16]{thuillier05:thesis}.  Now we calculate
  \[\begin{split}
      p_\eta\circ\pi_\eta\circ\tau_{\fX'}
      &= \pi'_\eta\circ p'_\eta\circ\tau_{\fX'}
      = \pi'_\eta\circ\tau_{\fS(a)'}\circ p'_\eta \\
      &= \tau_{\fS(a)}\circ\pi_\eta'\circ p'_\eta
      = \tau_{\fS(a)}\circ p_\eta\circ\pi_\eta
      = p_\eta\circ\tau_\fX\circ\pi_\eta.
    \end{split}\]
  Since $p_\eta$ is an isomorphism on skeletons and since $\pi_\eta(\Sigma')=\Sigma$, this shows that $\pi_\eta\circ\tau_{\fX'} = \tau_\fX\circ\pi_\eta$, and by definition we have $\pi_\eta\circ\tau_{\fX'} = \phi\circ\tau_{\fX'}$.

  It remains to prove~(5). It follows from Remark~\ref{rem:stratum face correspondence} and the considerations at the beginning that $\pi_\eta$ maps $S_0(\fX')$ onto $S_0(\fX)$ and that $S_0(\fX')=\pi_\eta\inv(S_0(\fX))$. 
  By Lemma~\ref{lem:skeleton.galois}, the Galois group $G$ acts by harmonic maps on $\Sigma'$. 
    By~\cite[Proposition~1.3.5(i)]{berkovic90:analytic_geometry}, the Galois group acts transitively on the fibers of $\pi_\eta\colon X'\to X$, so since $\Sigma'$ surjects onto $\Sigma$, this shows $\pi_\eta\inv(\Sigma) = \Sigma'$, and hence $G$ acts transitively on the fibers of $\phi$ as well.  We have already shown that $\phi\inv(\del\Sigma) = \del\Sigma'$, so we are done.
\end{proof}

\subsection{Functoriality With Respect to Morphisms of Curves}
In this section we discuss functoriality of skeletons with respect to a morphism of strictly semistable $k^\circ$-curves $f\colon\fX'\to\fX$, expanding on~\cite[Proposition~2.2.27]{thuillier05:thesis}.  We begin with a useful set-theoretic functoriality property.

\begin{prop}\label{prop:inverse.image.skeleton}
  Let $f\colon\fX'\to\fX$ be a morphism of strictly semistable $k^\circ$-curves.  Then $f_\eta\inv(\Sigma_{\fX})\subset\Sigma_{\fX'}$.
\end{prop}

The proof consists mostly of the following elementary lemma.

\begin{lem}\label{lem:map.from.pot.unit.disc}
  Let $U$ be a $k$-affinoid space that is potentially isomorphic to the unit disc, let $a\in k^\circ\setminus\{0\}$, and let $f\colon U\to S(a)$ be a morphism.  Then $f(U)\cap\Sigma_{\fS(a)} = \emptyset$.
\end{lem}

\begin{proof}
  To say that $U$ is potentially isomorphic to the unit disc means that there exists a finite, separable extension $k'/k$ such that $U\tensor_k k'$ is isomorphic to a disjoint union of copies of the unit disc $B = \sM(k'\angles T)$.  Let $S(a)_{k'} = S(a)_\eta\tensor_k k'$, and let $\pi\colon S(a)_{k'}\to S(a)$ denote the structure morphism.  We have a commutative square
  \[ \xymatrix @R=.3in{
      {\llap{$\Djunion B\cong$ } U\tensor_k k'} \ar[r]^{f_{k'}} \ar@{->>}[d] & {S(a)_{k'}} \ar[d]^\pi \\
      {U} \ar[r]_(.4)f & {S(a)_\eta,}
    } \]
  so it suffices to show that if $g\colon B\to S(a)_{k'}$ is a morphism, then $\pi\circ g(B)\cap\Sigma_{\fS(a)} = \emptyset$.

  We identify $S(a)$ with the affinoid subdomain $\{x\colon|a|\leq|t(x)|\leq 1\}$ of the affine line $\A^{1,\an} = \Spec(k[t])^\an$, and likewise with $S(a)_{k'}$.  Let $F = g^*(t) = (\pi\circ g)^*(t)\in k'\angles{T}$.  Since $t$ is a unit on $S(a)_{k'}$, we have $F\in k'\angles{T}^\times$, so we can write $F = \alpha(1+\epsilon)$ for $\alpha\in k'^\times$ and $\epsilon\in k'\angles T$ with $|\epsilon|_{\sup} < 1$.  Set $\rho = |\alpha|$.  We have $|t(\pi\circ g(x))| = |t(g(x))| = |F(x)| = |\alpha| = \rho$ for all $x\in B$, so $\pi\circ g(B)\subset\tau_{\fS(a)}\inv(\eta_\rho)$, where $\eta_\rho(\sum a_it^i) = \max\{|a_i|\,\rho^i\}$ is the point on the skeleton~\eqref{eq:eta_rho}.  Hence it suffices to show that $\eta_\rho\notin\pi\circ g(B)$.

  Let $p\in k[t]$ be the (monic) minimal polynomial of $\alpha$, and let $n = \deg(p)$.  For $x\in B$, we have
  \[ \bigl|p(\pi\circ g(x))\bigr| = \bigl|p(F(x))\bigr|
    = \bigl|p(\alpha + \alpha \epsilon(x))\bigr|
    = \bigl|p(\alpha+\alpha\epsilon(x))-p(\alpha)\bigr|. \]
  Writing $p(t) = b_0 + b_1t + \cdots + b_{n-1}t^{n-1} + t^n$, we have
  \[ p(\alpha+\alpha\epsilon) - p(\alpha) = \sum_{i=1}^n b_i\alpha^i\sum_{j=1}^i\binom ij\epsilon^j. \]
  Since $|\epsilon(x)| < 1$, each outer summand has absolute value less than $|b_i\alpha^i| = |b_i|\rho^i$ when evaluated at $x$, so $|p(\pi\circ g(x))| < \max\{|b_i|\rho^i\colon i=0,\ldots,n\} = \eta_\rho(p)$.  Hence $\pi\circ g(x)\neq\eta_\rho$, so $\eta_\rho\notin \pi\circ g(B)$.
\end{proof}

\begin{proof}[Proof of Proposition~\ref{prop:inverse.image.skeleton}]
  Let $\fU\subset\fX$ be a formal affine open admitting an \'etale morphism $p\colon\fU\to\fS(a) = \Spf(k^\circ\{T_0,T_1\}/(T_0T_1-a))$ for some $a\in k^\circ\setminus\{0\}$ and  $\fU' \coloneqq  f\inv(\fU)$.  Then $\Sigma_\fX\cap\fU_\eta = \Sigma_\fU = p_\eta\inv(\Sigma_{\fS(a)})$ and $\Sigma_{\fX'}\cap\fU'_\eta = \Sigma_{\fU'}$ by~\cite[Th\'eor\`eme~2.2.10(ii), Lemme~2.2.14]{thuillier05:thesis}, so we may replace $\fX$ by $\fS(a)$ and $\fX'$ by $\fU'$ to assume $\fX = \fS(a)$ is a formal annulus.

  Let $x'\in \fX'_\eta$, and assume that $x'\notin\Sigma_{\fX'}$.  We must show that $x = f_\eta(x')\notin\Sigma_{\fS(a)}$.  This is immediate from Lemma~\ref{lem:map.from.pot.unit.disc} since $x'$ admits an affinoid neighborhood that is potentially isomorphic to the unit disc.
\end{proof}

We will use the following elementary lemma.

\begin{lem}\label{lem:totally.ramified}
	Let $r\in\sqrt\Gamma$.  There exists a finite, separable extension $k'/k$ such that $r\in\val(k'^\times)$ and $\td k' = \td k$.
\end{lem}

\begin{proof}
	Let $n$ be the smallest positive integer such that $nr\in\Gamma$.  If $n=1$ then there is nothing to do.  Otherwise, choose $\lambda\in k^\times$ with $\val(\lambda) = nr$, and consider the polynomial $f(t) = t^n - \lambda\in k[t]$.  If $\chr(k) = p > 0$ and $p\mid n$ then $f$ is not separable; in this case, we use $f(t) = t^n + \lambda t - \lambda$, which is separable.  In either case, by a Newton polygon argument we see that all roots of $f$ (in an algebraic closure of $k$) have valuation $r$, so any product of a proper subset of the roots does not have valuation contained in $\Gamma$, and hence $f$ is irreducible over $k$.  Let $k' = k[t]/(f)$.  Then $[k':k] = n$ and $[|k'^\times|:|k^\times|] \geq n$, so $\td k'=\td k$ by~\cite[Proposition~3.6.2/5]{bosch_guntzer_remmert84:non_archimed_analysis}.
\end{proof}

\begin{lem} \label{lem:subdivison of skeletons and morphism}
Let $f\colon\fY\to\fX$ be a morphism of strictly semistable $k^\circ$-curves.  Define $\Sigma_f\colon\Sigma_{\fY}\to\Sigma_\fX$ by $\Sigma_f = \tau_\fX\circ f_\eta|_{\Sigma_{\fY}}$.  There is a subdivision $\Sigma_\fX'$ of $\Sigma_\fX$ with $\Gamma$-rational vertices and a subdivision $\Sigma_\fY'$ of $\Sigma_\fY$ with $\sqrt{\Gamma}$-rational vertices such that  $\Sigma_f$ induces a piecewise linear map $\Sigma_\fY' \to \Sigma_\fX$ of weighted metric graphs with boundary and such that
$$\Sigma_f\circ\tau_{\fY} = \tau_\fX\circ f_\eta.$$
Any edge of $\Sigma_\fY'$ is mapped either linearly onto an edge of $\Sigma_\fX'$ with integral expansion factor or to a vertex of $\Sigma_\fX'$.
\end{lem}

\begin{proof} 
By restriction to connected components, we may assume that $\fX_\eta$ and $\fY_\eta$ are connected. Then $f_\eta$ is either constant or has finite fibers~\cite[3.5.8]{ducros14:structur_des_courbes_analytiq}. If $f_\eta$ is constant, then the claim is obvious and so we may assume that $f_\eta$ has finite fibers. Then the claim follows from~\cite[Proposition~2.2.27]{thuillier05:thesis} and its proof except that we have to check the boundary condition (3) in Definition~\ref{def:morphism}. 	Let $w$ be a vertex of $\Sigma_\fY$, let $v = \Sigma_f(w)$, and suppose that $v\in\del\Sigma_\fX=\partial X$.  Let $C$ (resp.\ $D$) be the irreducible component of $\fX_s$ (resp.\ $\fY_s$) corresponding to $v$ (resp.\ $w$).  We assume that $\varphi \coloneqq \Sigma_f$ is not constant in a neighborhood of $w$. Then there is an edge $e'$ of $\Sigma_\fY'$ with vertex $w$ mapping linearly onto an edge $e$ of $\Sigma_\fX'$ with vertex $v$. For an interior point $p'$ of $e'$, let $p \coloneqq f_\eta(p')$. By construction, we have $\varphi(p')=\tau_\fX(p)$ is in the interior of $e$ and hence $\red_\fX(p)=f_s(\red_\fY(p'))$ is a singular point of $\fX_s$ contained in $C$ by Remark~\ref{rem:stratum face correspondence}. Since $v=\varphi(w)=\tau_\fX(f_\eta(w))$ is a vertex of $\Sigma_\fX$, it follows from \eqref{eq:stratum face correspondence for vertices} that $\red_\fX(f_\eta(w))$ is a smooth point of $\fX_s$ contained in $C$. Since $\red_\fX(f_\eta(w))=f_s(\red_\fY(w))$ and $\red_\fY(w)$ is the generic point of $D$, we conclude that $f_s|_D$ is non-constant and hence induces a dominant  morphism $D\to C$.  

Up to now, we only have used that $v$ is a vertex of $\Sigma_\fX$, but now we use $v \in \partial \Sigma_\fX$.  
Then $C$ is affine, and dominance of the above map shows that the  same must be true of $D$, so $w\in\del Y =\del\Sigma_\fY$.  This finishes the proof that $\Sigma_f$ is a piecewise linear map of weighted metric graphs with boundary.
\end{proof}

\begin{lem} \label{lem:subdivison of skeletons and blow up}
	Let $f\colon\fY\to\fX$ be a morphism of strictly semistable $k^\circ$-curves and let $\Sigma_\fY'$ (resp.~$\Sigma_\fY'$) be a subdivision of the skeleton $\Sigma_\fX$ (resp.~$\Sigma_\fY$) as in Lemma~\ref{lem:subdivison of skeletons and morphism}. Then there is a finite separable extension $k'/k$ such that $\td k'=\td k$ satisfying the following properties.
	\begin{enumerate}
		\item The skeletons of $\fX$ (resp.~$\fY$) and $\fX\otimes_{k^\circ} k'^\circ$ (resp.~$\fY\otimes_{k^\circ} k'^\circ$) are equal as weighted metric graphs with boundary.
		\item There are admissible formal blowups $\pi:\fX' \to \fX$ (resp.~ $\pi':\fY' \to \fY\otimes_{k^\circ} k'^\circ$) of strictly semistable curves such that $\Sigma_{\fX'\otimes_{k^\circ} k'^\circ}=\Sigma_{\fX'}=\Sigma_\fX'$ (resp.~$\Sigma_{\fY'}=\Sigma_\fY'$) such that
		\[ \xymatrix @R=.3in {
			{\fY'} \ar[r]^{\pi'} \ar[d]_{f'} &
			{\fY\otimes_{k^\circ} k'^\circ} \ar[d]_{f_{k'^\circ}} \\
			{\fX'\otimes_{k^\circ} k'^\circ} \ar[r]_{\pi_{k'^\circ}} & {\fX\otimes_{k^\circ} k'^\circ}
		} \]
	    is  commutative for a unique morphism $f':\fY' \to \fX'\otimes_{k^\circ} k'^\circ$ over $k'^\circ$.
\end{enumerate}
\end{lem}

\begin{proof}
Since the vertices of the subdivision $\Sigma_\fX'$ are $\Gamma$-rational, there is an admissible formal blowup $\pi:\fX' \to \fX$ such that $\fX'$ is a strictly semistable curve over $k^\circ$ with $\Sigma_{\fX'}=\Sigma_\fX'$, see the proof of Lemma 2.2.22 in~\cite{thuillier05:thesis}. By Lemma~\ref{lem:totally.ramified}, there is a finite separable extension $k'/k$ with $\td k'=\td k$ such that all vertices of the subdivision $\Sigma_\fY'$ are $\Gamma'$-rational for the value group $\Gamma'$ of $k'$. By Proposition~\ref{prop:skeleton.basechange}(4), skeletons do not change after base change to $k'^\circ$. Applying the above to $\fY_{k'^\circ}$, there is an admissible formal blowup $\pi':\fY' \to \fY \otimes_{k^\circ} k'^\circ$ such that $\fY'$ is  a strictly semistable curve over $k^\circ$ with $\Sigma_{\fY'}=\Sigma_\fY'$.

We have to show that the base change $f_{\eta,k'}$ of $f_\eta$ to $k'$ is the generic fiber of a morphism $f':\fY' \to \fX' \otimes_{k^\circ} k'^\circ$. Note that such a morphism $f'$ is always unique. For this, we may assume as in the proof of Lemma~\ref{lem:subdivison of skeletons and morphism} that $\fX_\eta$ and $\fY_\eta$ are connected. If $f_{\eta,k'}$ is constant, then it extends to all formal models of $\fX_\eta \otimes_k k'$ and hence we may assume that $f_{\eta,k'}$ has finite fibers using again~\cite[3.5.8]{ducros14:structur_des_courbes_analytiq}. Then it follows from~\cite[Proposition~2.1.15]{thuillier05:thesis} that the morphism $f'$ exists if and only if $f_{\eta,k'}^{-1}(S_0(\fX'))\subset S_0(\fY')$. 

Let $v$ be a vertex of $\Sigma_{\fX'\otimes_{k^\circ} k'^\circ}=\Sigma_{\fX'}=\Sigma_\fX'$ and let $y \in \fY_\eta'$ with $f_{\eta,k'}(y)=v$. Note that $v$ is in the skeleton of $\fX\otimes_{k^\circ} k'^\circ$ and hence Proposition~\ref{prop:inverse.image.skeleton} shows that $y$ is in the skeleton of $\fY\otimes_{k^\circ} k'^\circ$ which agrees also with the skeleton of $\fY$. Since $\Sigma_f(y)=  v$ is a vertex, we conclude that $y$ is a vertex of $\Sigma_\fY'$ by the choice of this subdivision and hence $y \in S_0(\fY')$. Using the above, this shows that the morphism $f'$ exists. 
\end{proof}

\begin{rem} \label{rem:vertices and subdivison}
In the situation of Lemma~\ref{lem:subdivison of skeletons and blow up}, let $e'$ be an edge of $\Sigma_{\fY'}=\Sigma_\fY'$ mapping linearly onto an edge $e$ of $\Sigma_\fX'$ by the map $\Sigma_f$. For any vertex $v'$ of $e'$, we have that $f_\eta(v')$ is the vertex $v\coloneqq \Sigma_f(v')$ of $\Sigma_\fX'$. 

Let $D$ be the  irreducible component of $\fY'$ with generic point $\red_{\fY'}(v')$. Then $f'$ induces a morphism $D \to C$ for the irreducible component $C$ of the special fiber of $\fX \otimes_{k^\circ} k'^\circ$ corresponding to the vertex $v\coloneqq \Sigma_f(v')$ of $\Sigma_\fX'$. The proof of Lemma~\ref{lem:subdivison of skeletons and morphism} shows that the morphism $D \to C$ is dominant. In particular, the generic point $\red_{\fY'}(v')$ of $D$ is mapped to the generic point $\red_{\fX'}(v)$ of $C$. On the other hand, we have $$\red_{\fX'}(v) = \red_{\fX'}(\tau_{\fX'}(f_\eta(v')))=\red_{\fX'}(f_\eta\tau_{\fY'}(v'))=\red_{\fX'}(f_\eta(v'))$$
The generic point of $C$ has a unique preimage with respect to $\red_{\fX'}$ and hence $v=f_\eta(v')$.
\end{rem}

\begin{rem} \label{rem:expansion factor}
In the above setting with $\Sigma_f(e')=e$, let $p,p'$ be the singular points of the special fibers of $\fX'\otimes_{k^\circ} k'^\circ,\fY'$ corresponding to $e$ and $e'$, respectively. Then $p$ (resp.~$p'$) is contained in the irreducible component $C$ (resp.~$D$) of the special fiber corresponding to $v$ (resp.~$v'$).
Remark~\ref{rem:vertices and subdivison} shows  that $f'$ induces a dominant map $D \to C$. It follows from Remark~\ref{rem:stratum face correspondence} that $p'$ maps to $p$. From the proof of~\cite[Proposition~2.2.27(iv)]{thuillier05:thesis}, we get that the expansion factor for the linear map $\Sigma_f:e' \to e$ is equal to the ramification index of the discrete valuation ring $\mathcal O_{D,p'}$ over the discrete valuation ring $\mathcal O_{C,p}$. 
\end{rem}

The main result of this subsection is the following proposition.  For metric graphs with trivial edge weights and empty boundary, it can be found in~\cite[Section~4]{abbr14:lifting_harmonic_morphism_I}; it also extends~\cite[Proposition~2.2.27]{thuillier05:thesis}. 

\begin{prop}\label{prop:skeleton.morphism}
  Let $f\colon\fX'\to\fX$ be a morphism of strictly semistable $k^\circ$-curves.  Define $\Sigma_f\colon\Sigma_{\fX'}\to\Sigma_\fX$ by $\Sigma_f = \tau_\fX\circ f_\eta|_{\Sigma_{\fX'}}$.  Then $\Sigma_f$ is a harmonic map of weighted metric graphs with boundary.
  If $f_\eta$ is finite, $\fX_\eta$ is connected, and $\Sigma_\fX$ has an edge, then $\Sigma_f$ has a well-defined degree $d(\Sigma_f) = \deg(f_\eta)$ in the sense of Definition~\ref{def:harmonic.degree}.
\end{prop}

If $f_\eta$ is finite, then as a map between analytic curves, it is flat and hence  $(f_\eta)_*\sO_{\fX'_\eta}$ is locally free. If $\fX_\eta$ is connected, then $(f_\eta)_*\sO_{\fX'_\eta}$ is of finite (constant) rank; the degree $\deg(f_\eta)$ is by definition the rank of $(f_\eta)_*\sO_{\fX'_\eta}$. These facts can be deduced algebraically by passing to strictly affinoid domains and their associated (noetherian regular) schemes, see~\cite[3.2]{ducros14:structur_des_courbes_analytiq}.

\begin{proof}
	For brevity we write $X = \fX_\eta$, $X' = \fX'_\eta$, $\Sigma = \Sigma_\fX$, $\Sigma' = \Sigma_{\fX'}$, and $\phi = \Sigma_f\colon\Sigma'\to\Sigma$.  Restricting to connected components, we may assume that $X$ and $X'$ are connected.  If $f$ is constant then there is nothing to check, so we assume that $f$ is non-constant, which implies that $f$ has finite fibers. 
	We have seen in Remark~\ref{rem:vertices and subdivison} that $\varphi$ is a piecewise linear map of weighted metric graphs with boundary after a suitable subdivision of skeletons. It follows from~\cite[Proposition~2.2.27(iv)]{thuillier05:thesis} that the pull-back of harmonic functions on $\Sigma$ with respect to $\varphi$  are harmonic functions on $\Sigma'$ and hence $\varphi$ is a harmonic map by Proposition~\ref{prop:characterization of harmonic morphisms}.

	Finally, suppose that $f_\eta$ is finite, $X=\fX_\eta$ is connected, and $\Sigma=\Sigma_\fX$ has an edge. Passing to a finite separable extension $k'/k$ which does not change the skeletons and then to admissible formal blowups which only lead to subdivisions of the skeletons as in Lemma~\ref{lem:subdivison of skeletons and blow up}, we may assume that $\Sigma_f$ maps every edge of $\Sigma'$ either to a vertex or linearly onto an edge of $\Sigma$.  
 Since $f_\eta$ is finite, we see that $f_s\colon\fX_s'\to\fX_s$ is proper by~\cite[Corollary~4.4]{temkin00:local_properties}.  Let $e$ be an edge of $\Sigma_\fX$ with corresponding singular point $\td x\in\fX_s$, and let $v$ be an endpoint of $e$ with corresponding irreducible component $C\subset\fX_s$. 
 By Remark~\ref{rem:stratum face correspondence},
  the fiber $f_s\inv(\td x)$ consists of finitely many singular points $\td x'$ corresponding to the finitely many edges $e'$ in $\varphi^{-1}(e)$ and so there is a formal affine neighborhood $\fU$ of $\td x_s$ such that $f_s\inv(\fU_s)\to\fU_s$ has finite fibers.  Since $f_s$ is proper, the morphism $f_s\inv(\fU_s)\to\fU_s$ is finite, and by~\cite[Proposition~3.25]{baker_payne_rabinoff13:analytic_curves}, the finite morphism $f_C\colon f_s\inv(C\cap\fU_s)\to C\cap\fU_s$ has degree $\deg(f_C) = \deg(f_\eta|_{U_\eta}) = \deg(f_\eta)$.  We have seen in Remark~\ref{rem:expansion factor} that the expansion factor $d_{e'}(\varphi)$ agrees with the ramification index $\varepsilon_{\td x'}(f_C)$ in $\td x'$. The fundamental identity relating the degree of a finite map with the ramification indices and residue degrees shows
	\[ d_e(\phi) = \sum_{e'\mapsto e}\frac{w(e')}{w(e)}d_{e'}(\phi)
	= \sum_{\td x'\mapsto\td x}[\kappa(\td x'):\kappa(\td x)]\varepsilon_{\td x'}(f_C)
	= \deg(f_C) = \deg(f_\eta),
	\]
	which proves that $\phi$ has well-defined degree $d(\phi) = \deg(f_\eta)$.
	\end{proof}

\subsection{Functoriality With Respect to Modifications}
A \defi{morphism of formal models} of a curve $X$ is a morphism of formal schemes $\fX'\to\fX$ respecting the identifications $\fX'_\eta\cong X\cong\fX_\eta$.  The following proposition is a mild strengthening of~\cite[Proposition~2.2.26]{thuillier05:thesis}.

\begin{prop}\label{prop:skeleton.blowup}
  Let $f\colon\fX'\to\fX$ be a morphism of strictly semistable models of a curve $X$. Then we have $\tau_\fX\circ\tau_{\fX'} = \tau_\fX$ and 
  the retraction $\tau_{\fX}\colon\Sigma_{\fX'}\to\Sigma_\fX$ is a modification of $\Sigma_\fX$ in the sense of Definition~\ref{def:modification}.
\end{prop}

\begin{proof}
	  The identity $\tau_\fX\circ\tau_{\fX'} = \tau_\fX$ is a special case of Lemma~\ref{lem:subdivison of skeletons and morphism} showing also that the restriction of $\tau_{\fX'}$ to the skeleton $\Sigma_{\fX'}$ is a piecewise linear map of weighted metric graphs with boundary. Using that $f_\eta={\rm id}_X$, we deduce from Remark~\ref{rem:vertices and subdivison} that $\tau_\fX$ contracts all edges of $\Sigma_{\fX'}\setminus \Sigma_{\fX}$. In particular, for every $x\in \Sigma_{\fX'\setminus \fX}$, the connected component $T$ of $\Sigma_{\fX'\setminus \fX}$ containing $x$ is contracted to $\tau_\fX(x)$. 
	 Since the skeletons $\Sigma_\fX$ and $\Sigma_{\fX'}$ are homotopy equivalent to $X$, they have the same first Betti number and hence $T$ is a tree. We conclude that $\tau_\fX$ induces a modification $\Sigma_{\fX'}\to \Sigma_\fX$ using  $\Sigma_\fX=\partial X = \Sigma_{\fX'}$.
	 \end{proof}


\section{Weakly Smooth Forms on Curves}
\label{Sec:forms-curves}

In this section we use the previous sections to characterize the weakly smooth forms on a Berkovich curve, and to compute its Dolbeault cohomology groups.  Essentially, a weakly smooth form on a curve is the pullback of a weakly smooth form on a skeleton under the retraction map.

\subsection{Pullback of Forms on Skeletons}\label{sec:pullback-forms-skel}
Let $X$ be a curve with strictly semistable model $\fX$.  For $p,q\in\{0,1\}$ we will define ``pullback'' maps
\[ \tau_\fX^*\colon \cA^{p,q}(\Sigma_\fX,\del\Sigma_\fX)\To\cA^{p,q}(X) \]
that will induce isomorphisms on Dolbeault cohomology groups.  The construction is as follows.  Let $\omega\in\cA^{p,q}(\Sigma_\fX,\del\Sigma_\fX)$, and let $\cU$ be the set of subgraphs $U\subset\Sigma_\fX$ with $\bar\Gamma$-rational vertices (Definition~\ref{def:gamma.ratl.pts}, Notation~\ref{notn:barGamma}), with the property that there exists a $(\Z,\Gamma)$-harmonic tropicalization $h_U\colon U\to\R^{n_U}$ of $(U,\del U)$ (Definition~\ref{def:rational.subgraph}) and a Lagerberg form $\alpha_U\in\cA^{p,q}(\R^{n_U})$ such that $h_U^*\alpha_U = \omega|_U$.  Every point of $\Sigma_\fX$ has a neighborhood of this form by Proposition~\ref{prop:local.pullbacks.graphs}.

The key fact connecting the algebraic geometry and combinatorics in this situation is the following criterion for a function on a curve to be $(\Z,\Gamma)$-harmonic; it can be found in~\otherfile{Proposition~\ref*{I-simple neighbourhoods and harmonic functions on skeletons}}. We refer to~\cite{gubler_rabinoff_jell:harmonic_trop} for the definition of harmonic tropicalizations and weakly smooth forms on good strictly analytic spaces.

\begin{prop}\label{prop:compose.with.retraction}
  The map $g_U = h_U\circ\tau_\fX\colon\tau_\fX\inv(U)\to\R^{n_U}$ is a 
  harmonic tropicalization.
\end{prop}

It follows that $(g_U, \alpha_U)$ is a weakly smooth preform of type $(p,q)$ on $\tau_\fX\inv(U)$.

\begin{lem}\label{lem:preforms-glue}
  The preforms $(g_U,\alpha_U)$ glue to a weakly smooth differential form on $X$.
\end{lem}

\begin{proof}
  We will first show that for $U$ and $U'$ as above, then $(g_U,\alpha_U) = (g_{U'},\alpha_{U'})$ on $\tau_\fX^{-1}(U \cap U')$.  This amounts to showing that if $U\subset\Sigma_\fX$ is a subgraph with $\bar\Gamma$-rational endpoints that admits two $(\Z,\Gamma)$-harmonic tropicalizations $h\colon U\to\R^n$ and $h'\colon U\to\R^{n'}$, and if $\alpha\in\cA^{p,q}(\R^n)$ and $\alpha'\in\cA^{p,q}(\R^{n'})$ are Lagerberg forms such that $h^*\alpha = h'^*\alpha'$, then $\pi^*\alpha = \pi'^*\alpha'$ on $(h\times h')(U)$, where $\pi,\pi'\colon\R^n\times\R^{n'}$ are the projections onto the two factors.  This is true by functoriality and injectivity of $(h\times h')^*$ on $\cA^{p,q}((h\times h')(U))$: see Lemma~\ref{lem:pullbacks.are.smooth}. 
  
  For every $x \in X$, there is a subgraph $U \in \cU$ of $\Gamma_\fX$ such that $\tau_\fX(x)$ has $U$ as a neighbourhood in $\Sigma_\fX$. Then $\tau_\fX^{-1}(U)$ is a compact strictly analytic domain in $X$ which is a neighbourhood of $x$ in $X$ and the above shows that the preforms  $(g_U,\alpha_U)$ glue to a weakly smooth differential form on $X$.
\end{proof}

\begin{defn}\label{def:skeleton-pullback}
  We denote by $\tau_\fX^*\omega$ the weakly smooth form that restricts to each $(g_U,\alpha_U)$.  We call $\tau_\fX^*\omega$ the \defi{pullback from the skeleton}.
\end{defn}

Note that we can define $\tau_\fX^*\omega$ by forms $(g_U,\alpha_U)_{U\in\cU'}$ for any $\cU'\subset\cU$ such that $\Sigma_\fX$ is covered by the interiors of the subgraphs in $\cU'$.

The pullback from the skeleton satisfies the expected compatibility properties.

\begin{lem}\label{lem:pullback-form-properties}
  The pullback $\tau_\fX^*\colon \cA^{p,q}(\Sigma_\fX,\del\Sigma_\fX)\to\cA^{p,q}(X)$ is a homomorphism of differential bigraded algebras: that is,
  \begin{enumerate}
  \item $\tau_\fX^*\d'\omega = \d'\tau_\fX^*\omega$ and $\tau_\fX^*\d''\omega = \d''\tau_\fX^*\omega$ for any smooth form $\omega$, and
  \item $\tau_\fX^*(\omega\wedge\eta) = \tau_\fX^*\omega\wedge\tau_\fX^*\eta$ for any smooth forms $\omega$ and $\eta$.
  \end{enumerate}
\end{lem}

\begin{proof}
  These are automatic: for example, $\d'\tau_\fX^*\omega$ is defined by the preforms $(g_U,\d'\alpha_U)$, and $h_U^*\d'\alpha_U = \d'h_U^*\alpha_U = \d'\omega|_U$.
\end{proof}

\begin{lem}\label{lem:skeleton-pullback-inj}
  The pullback from the skeleton $\tau_\fX^*\colon \cA^{p,q}(\Sigma_\fX,\del\Sigma_\fX)\to\cA^{p,q}(X)$ is injective.
\end{lem}

\begin{proof}
  Choose $\omega\in\cA^{p,q}(\Sigma,\del\Sigma)$ with $\tau_\fX^*\omega = 0$.
Suppose that $\tau_\fX^*\omega$ is defined by a collection $(g_U,\alpha_U)$ as in~\secref{sec:pullback-forms-skel}.  Since $g_U^*\colon\cA^{p,q}(h_U(U))\to\cA^{p,q}(\tau_\fX\inv U)$ is injective by~\otherfile{Proposition~\ref*{I-harmonic lift to weakly smooth forms}}, the form $\alpha_U$ restricts to zero on $h_U(U)$.  Since the map $h_U^*\colon\cA^{p,q}(\R^{n_U})\to\cA^{p,q}(U,\del U)$ factors through $\cA^{p,q}(h_U(U))$ by Lemma~\ref{lem:pullbacks.are.smooth}, this shows that $\omega = 0$.
\end{proof}

The next lemma relates pullback from the skeleton, pullback of Lagerberg forms to graphs, and pullback of Lagerberg forms to curves.  Let $U\subset\Sigma_\fX$ be a subgraph with $\bar\Gamma$-rational endpoints, and consider the retraction $\tau_\fX\colon\tau_\fX\inv(U)\to U$.  The strictly analytic domain $\tau_\fX\inv(U)$ is again a curve, but as noted in Remark~\ref{rem:non-rational-subgraph}, it need not be the case that $U$ is a skeleton of $\tau_\fX\inv(U)$ under our definitions, and hence we have not defined a pullback map $\tau_\fX^*\colon\cA^{p,q}(U,\del U)\to\cA^{p,q}(\tau_{\fX}\inv(U))$.  However, our definition extends to this situation in the obvious way, and we use this in the statement of the lemma.

\begin{lem}\label{lem:pullback.lagerberg}
  Let $X$ be a curve with strictly semistable model $\fX$, let $U\subset\Sigma_\fX$ be a subgraph with $\bar\Gamma$-rational endpoints, let $h\colon U\to\R^n$ be a  $(\Z,\Gamma)$-harmonic tropicalization of $(U,\del U)$, and let $\alpha\in\cA^{p,q}(\R^n)$.  Then $h\circ\tau_\fX\colon\tau_\fX\inv(U)\to\R^n$ is a $(\Z,\Gamma)$-harmonic tropicalization by Proposition~\ref{prop:compose.with.retraction}, and
  \[ (h\circ\tau_\fX)^*\alpha = \tau_\fX^*h^*\alpha. \]
\end{lem}

\begin{proof} This is obvious from the definition of $\tau_\fX^*$.
	\end{proof}

\subsection{Functoriality of Pullbacks}
Now we show that pullbacks are compatible with morphisms of semistable curves and extension of scalars.  We begin with morphisms of semistable curves.

\begin{lem}\label{lem:pullback.morphism}
  Let $f\colon\fX'\to\fX$ be a morphism of strictly semistable $k^\circ$-curves, let $X = \fX_\eta$ and $X' = \fX'_\eta$, and let $\phi = \Sigma_f\colon\Sigma_{\fX'}\to\Sigma_\fX$ be the induced harmonic map from Proposition~\ref{prop:skeleton.morphism}.  Then the following square is commutative:
  \[ \xymatrix{
      {\cA^{p,q}(\Sigma_\fX,\del\Sigma_\fX)} \ar[r]^{\phi^*} \ar[d]_{\tau_\fX^*} &
      {\cA^{p,q}(\Sigma_{\fX'},\del\Sigma_{\fX'})} \ar[d]^{\tau_{\fX'}^*} \\
      {\cA^{p,q}(X)} \ar[r]_{f_\eta^*} & {\cA^{p,q}(X').}
    } \]
\end{lem}

\begin{proof}
  Let $\omega\in\cA^{p,q}(\Sigma_\fX,\del\Sigma_\fX)$.  We need to show that $f_\eta^*\tau_\fX^*\omega = \tau_{\fX'}^*\phi^*\omega$, which is a local question on $X'$.  Suppose that $\tau_\fX^*\omega$ is defined by a collection $(g_U,\alpha_U)_{U\in\cU}$ as in~\secref{sec:pullback-forms-skel}.  Let $U'\subset\Sigma_{\fX'}$ be a subgraph with $\bar\Gamma$-rational endpoints such that $\phi(U')$ is contained in some $U\in\cU$, and note that $\Sigma_{\fX'}$ is covered by the interiors of such $U'$.  Then $\tau_\fX^*\omega|_{\tau_{\fX}\inv(U)} = \tau_\fX^*h_U^*\alpha_U = (h_U\circ\tau_\fX)^*\alpha_U$ by Lemma~\ref{lem:pullback.lagerberg}, so
  \[\begin{split}
      f_\eta^*\tau_\fX^*\omega|_{\tau_{\fX'}\inv(U')}
      &= f_\eta^*(h_U\circ\tau_\fX)^*\alpha_U|_{\tau_{\fX'}\inv(U')} \\
      &= (h_U\circ\tau_\fX\circ f_\eta)^*\alpha_U|_{\tau_{\fX'}\inv(U')}
      = (h_U\circ\phi\circ\tau_{\fX'})^*\alpha_U|_{\tau_{\fX'}\inv(U')},
  \end{split}\]
  where the last equality uses Lemma~\ref{lem:subdivison of skeletons and morphism}.  Since $h_U\circ\phi|_{U'}\colon U'\to\R^{n_U}$ is a $(\Z,\Gamma)$-harmonic tropicalization, this last form is equal to $\tau_{\fX'}^*(h_U\circ\phi|_{U'})^*\alpha_U$, again by Lemma~\ref{lem:pullback.lagerberg}.  Finally, $(h_U\circ\phi|_{U'})^*\alpha_U= \phi^*h_U^*\alpha_U|_{U'} = \phi^*\omega|_{U'}$, where the first equality is Lemma~\ref{lem:pullback.lagerberg.harmonic}.
\end{proof}

Now we prove functoriality with respect to extension of scalars and Galois actions.

\begin{lem}\label{lem:pullback.basechange}
  Let $k'/k$ be a non-Archimedean extension, let $\fX$ be a strictly semistable model of a curve $X$, and let $\fX' = \fX\hat\tensor_{k^\circ}k'^\circ$ and $X' = \fX'_\eta = X\hat\tensor_k k'$.  Let $\pi\colon\fX'\to\fX$ be the structure morphism, and let $\phi = \Sigma_\pi\colon\Sigma_{\fX'}\to\Sigma_\fX$ be the restriction of $\pi_\eta$ to $\Sigma_{\fX'}$, as in Proposition~\ref{prop:skeleton.basechange}.  Then the following square is commutative:
  \[ \xymatrix{
      {\cA^{p,q}(\Sigma_\fX,\del\Sigma_\fX)} \ar[r]^{\phi^*} \ar[d]_{\tau_\fX^*} &
      {\cA^{p,q}(\Sigma_{\fX'},\del\Sigma_{\fX'})} \ar[d]^{\tau_{\fX'}^*} \\
      {\cA^{p,q}(X)} \ar[r]_{\pi_\eta^*} & {\cA^{p,q}(X').}
    } \]
\end{lem}

\begin{proof}
  This is proved in exactly the same way as Lemma~\ref{lem:pullback.morphism}, using the identity $\phi\circ\tau_{\fX'} = \tau_\fX\circ\pi_\eta$ of Proposition~\ref{prop:skeleton.basechange}(2) instead of the identity $\phi\circ\tau_{\fX'} = \tau_\fX\circ f_\eta$.
\end{proof}

\begin{lem}\label{lem:pullback.galois}
  Let $X$ be a $k$-curve, let $k'/k$ be a finite Galois extension with Galois group $G$, let $X' = X\tensor_k k'$ with structure morphism $\pi\colon X'\to X$, and let $\fX'$ be a strictly semistable model of $X'$.  Suppose that $\rS_0(\fX') = \pi\inv(\pi(\rS_0(\fX')))$, so that $G$ acts on $\Sigma_{\fX'}$ by harmonic maps by Lemma~\ref{lem:skeleton.galois}.  Then the pullback
  \[ \tau_{\fX'}^*\colon\cA^{p,q}(\Sigma_{\fX'},\del\Sigma_{\fX'}) \To \cA^{p,q}(X') \]
  is compatible with the action of $G$ on both sides.
\end{lem}

\begin{proof}
  This is also proved in the same way as Lemma~\ref{lem:pullback.morphism}, this time using the compatibility of the Galois action with the retraction $\tau_{\fX'}$ from Lemma~\ref{lem:skeleton.galois}.
\end{proof}

\subsection{Integration of Pullbacks}
In this subsection we prove that pullback from the skeleton respects integration of forms.  We will need the following compatibility of tropical multiplicities, which is a variant of~\cite[Theorem~5.8]{baker_payne_rabinoff16:tropical_curves}.

\begin{lem}\label{lem:compat.tropical.mults}
  Let $X$ be a curve with a strictly semistable model $\fX$, let $U\subset\Sigma$ be a subgraph with $\bar\Gamma$-rational vertices, let $h\colon U\to\R^n$ be a $(\Z,\Gamma)$-harmonic tropicalization, let $V = \tau_\fX\inv(U)$, and let $g = h\circ\tau_\fX\colon V\to\R^n$.  Then $\Trop_h(U) = \Trop_g(V)$ as weighted polytopal complexes.
\end{lem}

\begin{proof}
  Let $k'$ be an algebraically closed non-Archimedean extension field, let $\fX' = \fX\hat\tensor_{k^\circ}k'^\circ$, let $\pi\colon\fX'\to\fX$ be the structure morphism, let $\phi = \Sigma_\pi\colon\Sigma_{\fX'}\to\Sigma_\fX$ be the harmonic map of Proposition~\ref{prop:skeleton.basechange}, let $h' = h\circ\phi\colon\Sigma_{\fX'}\to\R^n$, and let $U' = \phi\inv(U)$.  By Remark~\ref{rem:trop.and.degree} and Proposition~\ref{prop:skeleton.basechange}(1), we have $\Trop_h(U) = \Trop_{h'}(U')$ as weighted polytopal complexes.  Let $V' = V\hat\tensor_k k' = \tau_{\fX'}\inv(U')$ and let $g' = g\circ\pi\colon V'\to\R^n$.  Then $\Trop_g(V) = \Trop_{g'}(V')$ as weighted polytopal complexes by~\otherfile{Proposition~\ref*{I-base change and tropical multiplicities}}.  Thus we may extend scalars to assume $k$ is algebraically closed.  Note that all edge weights are now trivial.

  For brevity we set $\Sigma = \Sigma_\fX$ and $\tau=\tau_\fX$.  The sets underlying $\Trop_h(U)$ and $\Trop_g(V)$ coincide by surjectivity of retraction.  In this paragraph, all subdivisions are understood to be constructed by adding $\Gamma$-rational points of $\Sigma$; the corresponding subdivided skeleton is then the skeleton associated to a modification of $\fX$.  Choose subdivisions of $U$ and of $\Trop(U)$ as in~\secref{sec:integrate-pullback}.  We may assume that the edges of $\Trop_h(U)$ and of $\Trop_g(V)$ agree.  Let $\sigma\subset\Trop_h(U)$ be an edge, let $m_\sigma$ denote the multiplicity of $\sigma$ in $\Trop_h(U)$, and let $m_\sigma'$ denote the multiplicity of $\sigma$ in $\Trop_g(V)$.  We must show that $m_\sigma=m_\sigma'$.  After further subdivision, we can assume that $h\inv(\sigma)$ is a disjoint union of edges of $U$; replacing $U$ with one of these edges, we can assume $U$ consists of a single edge $e$.  Let $e^\circ$ be the interior of $e$.  Then $\tau\inv(e^\circ)$ is isomorphic to an open annulus $S(a)_+$ for $a\in k^{\circ\circ}\setminus\{0\}$ by~\cite[Proposition~2.3]{bosch_lutkeboh85:stable_reduction_I}.

  Let $h_1,\ldots,h_n$ be the coordinates of $h\colon e\to\R^n$, and let $a_i$ be the slope of $h_i$ on $e$, so $m_\sigma = \gcd(a_1,\ldots,a_n)$.  Let $g_i = h_i\circ\tau$.  By Proposition~\ref{prop:length.well.defined}, there exist units $u_1,\ldots,u_n$ on $V$ such that $g_i = -\log|u_i|$.  In other words, $h$ is the moment map associated to the morphism $f\colon V\to\bG_{\rm m}^n$ determined by $u_1,\ldots,u_n$.  Fix $i$ such that $h_i$ is not constant, and let $f_i\colon V\to\bG_{\rm m}$ be the $i$th coordinate of $f$.  By~\cite[Proposition~2.5]{baker_payne_rabinoff13:analytic_curves}, the restriction of $f_i$ to $\tau\inv(e^\circ)\cong S(a)_+$ is finite of degree $|a_i|$ onto its image.

  Let $L_\sigma$ be the linear subspace of $\R^n$ spanned by $\{x-y\colon x,y\in\sigma\}$, and let $N_\sigma = \Z^n\cap L_\sigma$. By~\otherfile{Proposition~\ref*{I-item:trop.mult.projection}},  a primitive lattice vector on the line $L_\sigma$ is $(a_1,\ldots,a_n)/m_\sigma$.  The projection onto the $i$th factor $p_i\colon\Z^n\to\Z$ takes $L_\sigma$ to the sublattice generated by $a_i/m_\sigma$.  Hence by the definition of $m_\sigma'$, we have
  \[ m_\sigma' = \frac 1{[\Z\colon p_i(N_\sigma)]}\,\deg(f_i|_{\tau\inv(e^\circ)})
    = \frac{m_\sigma}{|a_i|}\,|a_i| = m_\sigma
  \]
  proving the claim.
\end{proof}

\begin{prop}\label{prop:integration.pullback}
  Let $X$ be a curve with a strictly semistable model $\fX$.  Then
  \[ \int_X\tau_\fX^*\omega = \int_{\Sigma_\fX}\omega \qquad
    \int_{\del X}\tau_\fX^*\eta' = \int_{\del\Sigma_\fX}\eta' \qquad
    \int_{\del X}\tau_\fX^*\eta'' = \int_{\del\Sigma_\fX}\eta''
  \]
  for all $\omega\in\cA^{1,1}(\Sigma_\fX,\del\Sigma_\fX),~\eta'\in\cA^{1,0}(\Sigma_\fX,\del\Sigma_\fX),$ and $\eta''\in\cA^{0,1}(\Sigma_\fX,\del\Sigma_\fX)$.
\end{prop}

\begin{proof}
  For brevity we set $\Sigma=\Sigma_\fX$ and $\tau=\tau_\fX$.  Let $\omega\in\cA^{1,1}(\Sigma,\del\Sigma)$, and suppose that $\tau^*\omega$ is defined by a collection $(g_U,\alpha_U)_{U\in\cU}$ as in~\secref{sec:pullback-forms-skel}.  By Proposition~\ref{prop:local.pullbacks.graphs} and its proof, we can choose a finite subset $\cU'\subset\cU$ whose topological interiors cover $\Sigma$ and such that for each $U\in\cU'$, either $U$ is a neighborhood of an interior leaf vertex on which $\omega=0$, or $U$ can be chosen independent of $\omega$.  We have seen in Remark~\ref{rem:partition of unity} that there exists a partition of unity $\{\lambda_U\}_{U\in\cU}$ consisting of smooth functions on $\Sigma$, such that $\supp(\lambda_U)$ is compact and contained in the topological interior of $U$.   Then $\omega = \sum_{U\in\cU}\lambda_U\omega$, so by linearity of $\int_X\tau_\fX^*$ and $\int_\Sigma$, it suffices to show that $\int_X\tau^*(\lambda_U\omega) = \int_\Sigma\lambda_U\omega$ for each subgraph $U$.  If $\omega=0$ on $U$ then there is nothing to prove, so we assume that $U$ was chosen independently of $\omega$.  Then $\lambda_U = h_U^*f_U$ for a smooth function $f_U$ on $\R^{n_U}$ by Proposition~\ref{prop:local.pullbacks.graphs}, so $\lambda_U\omega|_U = h_U^*(f_U\alpha_U)$.  Setting $\alpha = f_U\alpha_U$ and replacing $\omega$ by $\lambda_U\omega$, we may assume that $\omega|_U = h^*\alpha$ for $h = h_U\colon U\to\R^n$ and $\alpha\in\cA^{1,1}(\R^n)$, and that $\omega$ is supported on $U$.  Setting $V = \tau\inv U$ and $g=h\circ\tau\colon V\to\R^n$, we have
  \[ \int_X\tau^*\omega = \int_V\tau^*h^*\alpha = \int_V(h\circ\tau)^*\alpha = \int_V g^*\alpha = \int_{\Trop_g(V)}\alpha, \]
  where the second equality is Lemma~\ref{lem:pullback.lagerberg}, and the last equality is the definition of the integral from~\otherfile{Proposition~\ref*{I-integration of piecewise smooth forms}}.  On the other hand, we have \[ \int_\Sigma\omega = \int_U h^*\alpha = \int_{\Trop_h(U)}\alpha \]
  by Proposition~\ref{prop:compat.integration}(1), so the equality $\int_X\tau^*\omega = \int_\Sigma\omega$ follows from Lemma~\ref{lem:compat.tropical.mults}.

  The corresponding proofs for $(1,0)$- and $(0,1)$-forms proceed in exactly the same way, using Proposition~\ref{prop:compat.integration}(2).
\end{proof}

\subsection{Weakly Smooth Forms and Dolbeault Cohomology of Curves}
In this subsection we are able to compute all weakly smooth forms and Dolbeault cohomology groups of a curve.  More precisely, we prove that the pullback from a suitable skeleton induces an isomorphism of Dolbeault cohomology groups for a curve with semistable reduction, and we use Galois actions to compute the Dolbeault cohomology groups of a general curve.  Finally, we prove that Poincar\'e duality holds for proper curves.

Recall from the beginning of Section~\ref{Sec:forms-curves} that a \defi{curve} is a compact, rig-smooth strictly $k$-analytic space of pure dimension $1$.

\begin{thm}\label{thm:pullback-form-skeleton}
  Let $X$ be a curve, let $p,q\in\{0,1\}$, and let $\eta\in\cA^{p,q}(X)$ be a weakly smooth form.  There exists a finite Galois extension $k'/k$, a strictly semistable model $\fX'$ of $X' = X\tensor_k{k'}$, and a smooth form $\omega\in\cA^{p,q}(\Sigma_{\fX'},\del\Sigma_{\fX'})$ such that $\eta'=\tau_{\fX'}^*\omega$, where $\eta'$ is the pullback of $\eta$ to $X'$.
\end{thm}

If $k'/k$ is Galois then $\cA^{p,q}(X_{k'})^{\Gal(k'/k)} = \cA^{p,q}(X)$ by~\otherfile{Proposition~\ref*{I-Galois invariant forms}}, so Theorem~\ref{thm:pullback-form-skeleton} completely characterizes the smooth forms on a curve.

\begin{proof}[Proof of Theorem~\ref{thm:pullback-form-skeleton}]
  Fix $\eta\in\cA^{p,q}(X)$.  By definition, there exists a collection $\cV$ of strictly $k$-affinoid domains $V\subset X$ whose topological interiors cover $X$, and for each $V$ a weakly smooth preform $\eta_V = (g_V, \alpha_V)$ on $V$, which agree on overlaps.  By~\otherfile{Proposition~\ref*{I-comparison harmonic in general}}, after extending the ground field, we may assume that each $V$ admits a strictly semistable model $\fV$ such that $g_V$ factors as $g_V = h_\fV\circ\tau_\fV$, where $h_\fV\colon\Sigma_{\fV}\to\R^{n_V}$ is $(\Z,\Gamma)$-harmonic.  By~\cite[Theorem~5.1.14(iv), 6.4.3]{ducros14:structur_des_courbes_analytiq}, after potentially extending the ground field and passing to a modification of $\fV$, we may assume that the formal models $\fV$ glue to give a strictly semistable model $\fX$ of $X$.  (Start with a triangulation containing the inverse image under reduction of all of the generic points of each $\fV_s$.)

We define $\omega_\fV\in\cA^{p,q}(\Sigma_\fV,\del\Sigma_\fV)$ to be $h_\fV^*\alpha_V$.  Since the $\eta_V$ agree on overlaps, we see from  Lemma~\ref{lem:skeleton-pullback-inj}  that the $\omega_\fV$ glue to give a smooth form $\omega\in\cA^{p,q}(\Sigma_\fX,\del\Sigma_\fX)$.  By construction we have $\eta_V = \tau_\fV^*\omega_\fV$ for all $V$, so $\eta = \tau_\fX^*\omega$.
\end{proof}

\begin{thm}\label{thm:pullback-skeleton-isom}
  Let $X$ be a curve with a strictly semistable model $\fX$.  Suppose that $\fX_s$ contains no connected component that is both smooth and proper, or equivalently, that $\Sigma_\fX$ has no isolated interior vertices.  Then the pullback from the skeleton induces isomorphisms
  \[ \tau_\fX^*\colon H^{p,q}(\Sigma_\fX,\del\Sigma_\fX) \isom H^{p,q}(X) \]
  for all $p,q\in\{0,1\}$.
\end{thm}

\begin{proof}
  By Lemma~\ref{lem:pullback-form-properties}, the map $\tau_\fX^*$ is well-defined on Dolbeault cohomology groups. 
  For a $\d''$-closed form $\omega\in\cA^{p,q}(\scdot)$, let $[\omega]$ denotes its cohomology class in $H^{p,q}(\scdot)$ and set $\Sigma = \Sigma_\fX$.

  First we suppose that all irreducible components of $\fX_s$ are geometrically irreducible and that all singular points of $\fX_s$ are $\td k$-rational.  Under these hypotheses, if $k'/k$ is a finite, separable extension then the induced morphism $\Sigma_{\fX\tensor_{k^\circ}k'^\circ}\to\Sigma$ is an isomorphism by Proposition~\ref{prop:skeleton.basechange}(3).

  Let us show that $\tau_\fX^*$ is injective.  Let $\omega\in\cA^{p,q}(\Sigma,\del\Sigma)$ be a $\d''$-closed form, and suppose that $\tau_\fX^*[\omega] = 0$, so $\tau_\fX^*\omega = \d''\eta$ for some $\eta\in\cA^{p,q-1}(X)$.  By Theorem~\ref{thm:pullback-form-skeleton}, there exists a finite, separable extension $k'/k$, a strictly semistable model $\fX'$ of $X' = X\tensor_k{k'}$, and a form $\omega'\in\cA^{p,q-1}(\Sigma_{\fX'},\del\Sigma_{\fX'})$ such that $\pi^*\eta = \tau_{\fX'}^*\omega'$, where $\pi\colon X'\to X$ is the structure morphism.  By~\cite[Proposition~2.2.27]{thuillier05:thesis}, after a further finite, separable field extension we may replace $\fX'$ with by a suitable admissible formal blowup to assume that there is a morphism of strictly semistable models $f\colon\fX'\to\fX\tensor_{k^\circ}{k'^\circ}$.  The composition $\pi\circ f\colon\fX'\to\fX\tensor_{k^\circ}{k'^\circ}\to\fX$ induces a harmonic map of skeletons $\phi\colon\Sigma_{\fX'}\to\Sigma_\fX$ by Propositions~\ref{prop:skeleton.blowup} and~\ref{prop:skeleton.basechange} and Remark~\ref{rem:functoriality of pull-back}.  Note that $\phi^*\colon H^{p,q}(\Sigma_\fX,\del\Sigma_\fX)\to H^{p,q}(\Sigma_{\fX'},\del\Sigma_{\fX'})$ is an isomorphism by Proposition~\ref{prop:skeleton.blowup}, Lemma~\ref{lem:modif.isom.cohom}, and the hypothesis that $\Sigma_{\fX} = \Sigma_{\fX\tensor_{k^\circ}{k'^\circ}}$.   By Lemmas~\ref{lem:pullback.morphism} and~\ref{lem:pullback.basechange}, we have $\pi^*\tau_\fX^* = \tau_{\fX'}^*\phi^*$, so
  \[ \tau_{\fX'}^*\phi^*\omega
    = \pi^*\tau_\fX^*\omega
    = \pi^*\d''\eta
    = \d''\pi^*\eta
    = \d''\tau_{\fX'}^*\omega'
    = \tau_{\fX'}^*\d''\omega'.
  \]
  By Lemma~\ref{lem:skeleton-pullback-inj} we know that $\tau_{\fX'}^*\colon\cA^{p,q}(\Sigma_{\fX'},\del\Sigma_{\fX'})\to\cA^{p,q}(X')$ is injective, so $\phi^*\omega = \d''\omega'$, and hence $\phi^*[\omega] = 0$.  But $\phi^*$ is an isomorphism, so $[\omega] = 0$, which proves injectivity.

  Now we show that $\tau_\fX^*$ is surjective.  Let $\eta\in\cA^{p,q}(X)$ be a $\d''$-closed form.  Using Theorem~\ref{thm:pullback-form-skeleton}, we find as above a finite, separable field extension $k'/k$, a strictly semistable model $\fX'$ of $X' = X\tensor_k{k'}$ with a morphism $f\colon\fX'\to\fX\tensor_{k^\circ}{k'^\circ}$, and $\omega'\in\cA^{p,q}(\Sigma_{\fX'},\del\Sigma_{\fX'})$ such that $\pi^*\eta=\tau_{\fX'}^*\omega'$, where $\pi\colon X'\to X$ is the structure morphism.  The form $\omega'$ is $\d''$-closed by injectivity of $\tau_{\fX'}^*$.  Using the above notation, since $\phi^*\colon H^{p,q}(\Sigma_{\fX},\del\Sigma_{\fX})\to H^{p,q}(\Sigma_{\fX'},\del\Sigma_{\fX'})$ is an isomorphism,  there exists a $\d''$-closed form $\omega\in\cA^{p,q}(\Sigma_\fX,\del\Sigma_\fX)$ such that $\phi^*[\omega] = [\omega']$, so $\phi^*\omega = \omega' + \d''\rho'$ for some $\rho'\in\cA^{p,q-1}(\Sigma_{\fX'},\del\Sigma_{\fX'})$.  We compute
  \[ \pi^*\tau_\fX^*\omega = \tau_{\fX'}^*\phi^*\omega
    = \tau_{\fX'}^*\omega' + \tau_{\fX'}^*\d''\rho'
    = \pi^*\eta + \d''\tau_{\fX'}^*\rho'.
  \]
  It follows that $\pi^*[\tau_{\fX}^*\omega] = \pi^*[\eta]$; since $\pi^*\colon H^{p,q}(X)\to H^{p,q}(X')$ is injective by~\otherfile{Corollary~\ref*{I-Galois action on cohomology}}, we have $\tau_\fX^*[\omega] = [\eta]$, as desired.

  Finally, we drop the hypothesis that all irreducible components of $\fX_s$ are geometrically irreducible and that all singular points of $\fX_s$ are $\td k$-rational. A base change of $\fX_s$ to a suitable finite separable extension $F$ of $\td k$ has geometrically irreducible components and all singular points become $F$-rational (using the \'etale morphism $\fU\to\fS(a_\fU)$ from Definition~\ref{def:strictly semistable over valuation ring}). Lifting the minimal polynomial of a primitive element of $F$ to a polynomial $p$ with coefficients in $k^\circ$, the decomposition field of $p$  is a finite Galois extension $k'/k$ such that the above hypotheses hold true for $\fX' = \fX\tensor_{k^\circ}{k'^\circ}$.  Consider the square
  \[ \xymatrix{
      {H^{p,q}(\Sigma_\fX,\del\Sigma_\fX)} \ar[r]^(.49){\Sigma_\pi^*} \ar[d]_{\tau_\fX^*} &
      {H^{p,q}(\Sigma_{\fX'},\del\Sigma_{\fX'})} \ar[d]^{\tau_{\fX'}^*} \\
      {H^{p,q}(X)} \ar[r]_{\pi^*} & {H^{p,q}(X'),}
    } \]
  where $X' = \fX'_\eta = X\tensor_k{k'}$ and $\pi\colon X'\to X$ is the structure morphism.  This square commutes by Lemma~\ref{lem:pullback.basechange}.  The right arrow is an isomorphism by the above, and it is Galois-equivariant by Lemma~\ref{lem:pullback.galois}.
  Consider the square
  \[ \xymatrix{
      {H^{p,q}(\Sigma_\fX,\del\Sigma_\fX)} \ar[r]^(.49){\Sigma_\pi^*} \ar[d]_{\tau_\fX^*} &
      {H^{p,q}(\Sigma_{\fX'},\del\Sigma_{\fX'})^{\rlap{$\scriptstyle\Gal(k'/k)$}}} \ar[d]^{\tau_{\fX'}^*} \\
      {H^{p,q}(X)} \ar[r]_{\pi^*} & {H^{p,q}(X')^{\rlap{$\scriptstyle\Gal(k'/k)$}}.}
    } \]
  It follows from~\otherfile{Corollary~\ref*{I-Galois action on cohomology}} that the bottom arrow is an isomorphism; by Propositions~\ref{prop:quotient.graph.cohom} and~\ref{prop:skeleton.basechange}(5), the top arrow is an isomorphism; and by the above, the right arrow is an isomorphism.  Thus $\tau_\fX^*$ is an isomorphism.
\end{proof}

\begin{thm}[Dolbeault Cohomology of Curves]\label{thm:dolbeault.cohom.curve}
  Let $X$ be a connected curve, and let $g = h^1_{\sing}(|X|)$ be the first Betti number of its underlying topological space.  Then $h^{i,j} = h^{i,j}(X)$ are equal to:
  \begin{center}
    \begin{tikzpicture}
      \node (a) at (0,0) {If $\del X=\emptyset$:};
      \node[below=5mm] (b) at (a) {$h^{1,1} = 1$};
      \node[below left=1mm] (c) at (b.south west) {$h^{1,0} = g$};
      \node[below right=1mm] (d) at (b.south east) {$h^{0,1} = g$};
      \node[below right=1mm] (e) at (c.south east) {$h^{0,0} = 1$};
      \coordinate (x) at ($(a.south east)!(current bounding box.north west)!(a.south west)$);
      \coordinate (y) at ($(a.south east)!(current bounding box.north east)!(a.south west)$);
      \draw[thick] (x) -- (y);
    \end{tikzpicture}\qquad
    \begin{tikzpicture}
      \node (a) at (0,0) {If $\del X\neq\emptyset$:};
      \node[below=5mm] (b) at (a) {$h^{1,1} = 0$};
      \node[below left=1mm] (c) at (b.south west) {$h^{1,0} = g+\#\del X-1$};
      \node[below right=1mm] (d) at (b.south east) {$h^{0,1} = g\phantom{+\#\del X-1}$};
      \node[below right=1mm] (e) at (c.south east) {$h^{0,0} = 1$};
      \coordinate (x) at ($(a.south east)!(current bounding box.north west)!(a.south west)$);
      \coordinate (y) at ($(a.south east)!(current bounding box.north east)!(a.south west)$);
      \draw[thick] (x) -- (y);
    \end{tikzpicture}
  \end{center}
  More precisely, if $\del X=\emptyset$ (i.e.\ if $X$ is proper) then $\d''\cA^{1,0}(X) = \ker\int_X$.
\end{thm}

\begin{proof}
  We have seen in Theorem~\ref{thm:pullback-form-skeleton} that the semistable reduction theorem gives a finite Galois extension $k'/k$ such that $X' = X\tensor_k k'$ admits a strictly semistable formal model $\fX'$.  We may assume that $\fX'_s$ has no smooth connected components by blowing up.  By Remark~\ref{rem:galois.stable.model}, after passing to an admissible formal blowup, we can find $\fX'$ such that $\rS_0(\fX') = \pi\inv(\pi(\rS_0(\fX')))$, where $\pi\colon X'\to X$ is the structure morphism.  In this case, the Galois group $G = \Gal(k'/k)$ acts on $\Sigma_{\fX'}$ by harmonic maps by Lemma~\ref{lem:skeleton.galois}; let $\Sigma = \Sigma'/G$ be the quotient.   Note that $\pi$ is finite and hence boundaryless. By~\cite[Proposition~3.1.3(ii)]{berkovic90:analytic_geometry}, we get $\pi\inv(\del X) = \del X'$. 
  Using the definition of the quotient $\Sigma=\Sigma'/G$, we have $\pi\inv(\del \Sigma) = \del \Sigma'$. By~\cite[Proposition~1.3.5]{berkovic90:analytic_geometry}, we have $X=X'/G$ 
   and so we have a canonical identification of $\del\Sigma$ with $\del X$. 
  For all $p,q$ we have
  \[ H^{p,q}(X) = H^{p,q}(X')^G = H^{p,q}(\Sigma',\del\Sigma')^G = H^{p,q}(\Sigma,\del\Sigma) \]
  by~\otherfile{Corollary~\ref*{I-Galois action on cohomology}}, Lemma~\ref{lem:pullback.galois}, Theorem~\ref{thm:pullback-skeleton-isom}, and Proposition~\ref{prop:quotient.graph.cohom}.

  By~\otherfile{Corollary~\ref*{I-identification with singular cohomology}}, the cohomology groups $H^{0,q}(X)$ are canonically isomorphic to the singular cohomology groups $H^q_{\sing}(|X|,\R)$ for $q\geq 0$; it follows that $h^{0,0}(X) = 1$ and $h^{0,1}(X) = g$.  This implies that $\Sigma$ is connected, and that $h^1_{\sing}(\Sigma) = g$ by Proposition~\ref{prop:dolbeault.graphs}.  Applying Proposition~\ref{prop:dolbeault.graphs} again, we have $h^{1,1}(X) = h^{1,1}(\Sigma) = 1$ and $h^{1,0}(X) = h^{1,0}(\Sigma) = g$ when $\del X=\emptyset$, and we have $h^{1,1}(X) = h^{1,1}(\Sigma,\del\Sigma) = 0$ and
  \[ h^{1,0}(X) = h^{1,0}(\Sigma,\del\Sigma) = g + \#\del\Sigma - 1 = g + \#\del X-1 \]
  when $\del\Sigma\neq\emptyset$.

  The final assertion follows from Stokes' theorem~\otherfile{Theorem~\ref*{I-theorem of Stokes}} and the fact that $h^{1,1}(X) = 1$.
\end{proof}

Let $X$ be a proper curve.  As in~\secref{sec:poincare-duality-graphs}, the wedge product followed by integration defines bilinear pairings
\[ \cA^{0,0}(X)\times\cA^{1,1}(X)\To\R, \qquad
  \cA^{1,0}(X)\times\cA^{0,1}(X)\To\R.
\]
By Stokes' theorem~\otherfile{Theorem~\ref*{I-theorem of Stokes}}, these descend to pairings
\begin{equation}\label{eq:poincare.pairing.curves}
  H^{0,0}(X)\times H^{1,1}(X)\To\R, \qquad
  H^{1,0}(X)\times H^{0,1}(X)\To\R.
\end{equation}

\begin{thm}[Poincar\'e Duality]\label{thm:poincare.duality}
  If $X$ is a proper $k$-curve, then the  pairings~\eqref{eq:poincare.pairing.curves} are perfect.
\end{thm}

\begin{proof}
  Suppose first that $X$ admits a strictly semistable model $\fX$.  After blowing up, we may assume that $\fX$ has no smooth connected components.  Let $\Sigma = \Sigma_\fX$ and $\tau = \tau_\fX$.  For $p,q\in\{0,1\}$, we have a square
  \[ \xymatrix{
      {H^{p,q}(\Sigma)\times H^{1-p,1-q}(\Sigma)} \ar[rr]^(.7){\int_\Sigma\circ\wedge} \ar[d]_{\tau^*}^{\cong} &&
      {\R} \ar @{=}[d] \\
      {H^{p,q}(X)\times H^{1-p,1-q}(X)} \ar[rr]_(.7){\int_X\circ\wedge} && {\R}
    }\]
  which is commutative by Lemma~\ref{lem:pullback-form-properties} and Proposition~\ref{prop:integration.pullback}, and in which the left vertical arrow is an isomorphism by Theorem~\ref{thm:pullback-skeleton-isom}.  Hence~\eqref{eq:poincare.pairing.curves} is perfect by Poincar\'e duality for graphs, Proposition~\ref{prop:poincare.duality.graphs}.

  In general, there exists a finite Galois extension $k'/k$ such that $X' = X\tensor_k k'$ admits a semistable model.  The pairing $H^{p,q}(X')\times H^{1-p,1-q}(X')\to\R$ is equivariant with respect to the Galois action on the left side and the trivial action on the right side by~\otherfile{Proposition~\ref*{I-integration and conjugation}}.  Since $H^{p,q}(X) = H^{p,q}(X')^{\Gal(k'/k)}$ for all $p,q\in\{0,1\}$ by~~\otherfile{Corollary~\ref*{I-Galois action on cohomology}}, showing that the pairings~\eqref{eq:poincare.pairing.curves} are perfect amounts to proving the elementary linear algebra fact claimed in the next paragraph.

  Let $G$ be a finite group acting on a finite-dimensional real vector space $V$.  Then $G$ acts on $\Hom(V,\R)$ by $(\sigma\cdot\lambda)(v) = \lambda(\sigma\inv v)$.  We claim that the natural map $\Hom(V,\R)^G\to\Hom(V^G,\R)$ is an isomorphism.  Consider the exact sequence
  \[ 0 \To V^G \To V \To V/V^G \To 0 \]
  of $\R[G]$-modules.  This gives rise to an exact sequence of $\R[G]$-modules
  \[ 0 \To \Hom(V/V^G,\R) \To \Hom(V,\R) \To \Hom(V^G,\R) \To 0. \]
  Since $G$ is a finite group, the higher cohomology groups of $G$ are torsion by~\cite[Proposition~VII.6]{serre68:corps_locaux}, so $W\mapsto W^G$ is an exact functor on the category of $\R$-vector spaces.  Hence it suffices to show that $\Hom(V/V^G,\R)^G = 0$.  In other words, if $\lambda\colon V\to\R$ is zero on $V^G$ and satisfies $\lambda(\sigma v) = \lambda(v)$ for all $\sigma\in G$ and $v\in V$, then we must show $\lambda=0$.  Indeed, for $v\in V$ we have $w = \sum_{\sigma\in G}\sigma v\in V^G$, so
  \[ 0 = \lambda(w) = \sum_{\sigma\in G}\lambda(\sigma v) = (\#G)\,\lambda(v) \]
  and hence $v=0$. This concludes the proof.
\end{proof}

\bibliographystyle{egabibstyle}
\bibliography{papers}

\end{document}